\documentclass{amsart}
\title{Geometry of Kottwitz-Viehmann varieties}
\author{Jingren Chi}
\address{Department of Mathematics, The University of Chicago,
5734 S. University Ave., Chicago, IL, 60637}
\email{chijingren@gmail.com}
\date{}
\usepackage{amsmath,amsthm,amssymb,amscd}
\usepackage[all,cmtip]{xy}
\usepackage{dsfont}
\usepackage{mathrsfs}
\usepackage{leftidx}
\usepackage[colorinlistoftodos]{todonotes}
\usepackage[square,sort,comma,numbers]{natbib}

\usepackage{color}
\usepackage{hyperref}
\hypersetup{
   colorlinks=true,
   linktoc=all,
   linkcolor=black,
   citecolor=black,
   urlcolor=black,
}
\usepackage{geometry}
\newcommand{\into}{\hookrightarrow}

\newcommand{\fg}{\mathfrak{g}}

\newcommand{\bA}{\mathbb{A}}
\newcommand{\cA}{\mathcal{A}}

\newcommand{\bB}{\mathbb{B}}
\newcommand{\cB}{\mathcal{B}}

\newcommand{\fc}{\mathfrak{c}}
\newcommand{\fC}{\mathfrak{C}}

\newcommand{\cD}{\mathcal{D}}
\newcommand{\fD}{\mathfrak{D}}
\newcommand{\sfD}{\mathsf{D}}

\newcommand{\cE}{\mathcal{E}}

\newcommand{\ga}{\gamma}
\newcommand{\bG}{\mathbb{G}}

\newcommand{\cH}{\mathcal{H}}

\newcommand{\cI}{\mathcal{I}}

\newcommand{\cJ}{\mathcal{J}}

\newcommand{\la}{\lambda}
\newcommand{\La}{\Lambda}
\newcommand{\cL}{\mathcal{L}}

\newcommand{\fm}{\mathfrak{m}}
\newcommand{\cM}{\mathcal{M}}
\newcommand{\bN}{\mathbb{N}}
\newcommand{\cN}{\mathcal{N}}

\newcommand{\cO}{\mathcal{O}}

\newcommand{\cP}{\mathcal{P}}
\newcommand{\sfP}{\mathsf{P}}

\newcommand{\bQ}{\mathbb{Q}}

\newcommand{\bR}{\mathbb{R}}

\newcommand{\ft}{\mathfrak{t}}
\newcommand{\bT}{\mathbb{T}}

\newcommand{\bZ}{\mathbb{Z}}

\def\spec{\mathop{\rm Spec}\nolimits}
\def\Hom{\mathop{\rm Hom}\nolimits}

\newtheorem{thm}[subsubsection]{Theorem}
\newtheorem*{thm*}{Theorem}
\newtheorem{cor}[subsubsection]{Corollary}
\newtheorem{lem}[subsubsection]{Lemma}
\newtheorem{prop}[subsubsection]{Proposition}

\theoremstyle{definition}
\newtheorem{defn}[subsubsection]{Definition}

\newtheorem{conjecture}[subsubsection]{Conjecture}
\newtheorem*{conjecture*}{Conjecture}

\theoremstyle{remark}
\newtheorem{rem}[subsubsection]{Remark}

\numberwithin{equation}{section}

\newcommand{\introthmname}{}
\newtheorem{introthminn}{\introthmname}

\newcommand{\gad}{G_{\mathrm{ad}}}
\newcommand{\gsc}{G^{\mathrm{sc}}}
\newcommand{\Sp}{\mathrm{Sp}}
\newcommand{\lad}{\lambda_{\mathrm{ad}}}
\newcommand{\tad}{T_{\mathrm{ad}}}
\newcommand{\tsc}{T^{\mathrm{sc}}}
\newcommand{\vin}{\mathrm{Vin}}
\newcommand{\ving}{\mathrm{Vin}_{G}}

\begin{document}
\maketitle
\begin{abstract}
We study basic geometric properties of Kottwitz-Viehmann varieties, which are certain generalizations of affine Springer fibers that encode orbital integrals of spherical Hecke functions. Based on previous work of A. Bouthier and the author, we show that these varieties are equidimensional and give a precise formula for their dimension. Also we give a conjectural description of their number of irreducible components in terms of certain weight multiplicities of  the Langlands dual group and we prove the conjecture in the case of unramified conjugacy class.
\end{abstract}
\tableofcontents
\section{Introduction}
\subsection{Background and motivation}
In this article we study certain analogue of affine Springer fibres that we call \emph{Kottwitz-Viehmann varieties} whose underlying set is defined as
\[X_\gamma^\la=\{g\in G(F)/G(\cO)| g^{-1}\gamma g\in G(\cO)\varpi^\la G(\cO)\}\]
where 
\begin{itemize}
\item $G$ is a connected reductive algebraic group over a field $k$;
\item $F=k((\varpi))$ is the field of Laurent series with coefficients in $k$ and $\cO=k[[\varpi]]$ is the ring of power series;
\item $\gamma\in G(F)$ is a regular semisimple element;
\item $\la:\bG_m\to T$ is a cocharacter of a maximal torus $T$ of $G$ and 
\[\varpi^\la:=\la(\varpi)\in G(F).\]
\end{itemize}
Also we will consider a closely related set $X_\ga^{\le\la}$ defined by Replacing the double coset $G(\cO)\varpi^\la G(\cO)$ in the definition of $X_\ga^\la$ by the union
\[\overline{G(\cO)\varpi^\la G(\cO)}=\subset_{\mu\le\la}G(\cO)\varpi^\mu G(\cO)\]

These sets were first studied by Kottwitz and Viehmann in \citep{KoV}. More general versions of them (replacing $G(\cO)$ by parahoric subgroups of $G(F)$) have also been studied by Lusztig in \cite{Lu15}. When $k$ is a finite field, they arise naturally in the study of orbital integrals of functions in the spherical Hecke algebra $\cH(G(F), G(\cO))$ consisting of $G(\cO)$-biinvariant locally constant functions with compact support on $G(F)$.\par 
It turns out that $X_\gamma^\la$ can be realized as the set of $k$-rational points of some algebraic variety over $k$. We view them as group analogue of affine Springer fibers for Lie algebras studied by Kazhdan and Lusztig in \cite{KL}:
\[X_\ga=\{g\in G(F)/G(\cO)|\mathrm{ad}(g)^{-1}\gamma\in\mathfrak{g}(\cO)\}.\]
Here $\fg$ is the Lie algebra of $G$, $\ga\in\fg(F)$ is a regular semisimple element and ``ad" denotes the adjoint action of $G$ on $\fg$.\par 
Basic geometric properties of these affine Springer fibers $X_\ga$   have been well understood through the works of Kazhdan and Lusztig \cite{KL}, Bezrukavnikov \cite{Be96}, Ng\^o \cite{Ngo10}. A key ingredient in their approach is the symmetry on $X_\gamma$ arising from the centralizer $G_\ga(F)$. More precisely, the group $G_\ga(F)$ has a dense open orbit $X_\ga^{\mathrm{reg}}$
(the ``regular locus") and geometric properties of $X_\ga^\la$ are reduced to the commutative algebraic group $G_\ga(F)$ (more precisely certain finite dimensional quotient $P_\ga$ of the infinite dimensional loop group $G_\ga(F)$).\par 
We would like to generalize these methods to study the Kottwitz-Viehmann varieties $X_\ga^\la$. Similar to Lie algebra case, the (connected) centralizer $G_\ga^0(F)$ acts naturally on $X_\ga^\la$ and we consider the open orbits $X_\ga^{\la,\mathrm{reg}}$ (the ``regular locus"). However, there are the following notable differences from the Lie algebra situation:
\begin{itemize}
\item In general the action of $G_\ga^0(F)$ on $X_\ga^{\la,\mathrm{reg}}$ is not transitive.
\item A more serious problem is that in general the ``regular locus" $X_\ga^{\la,\mathrm{reg}}$ is not dense in $X_\ga$ and there might be irreducible components disjoint from $X_\ga^{\la,\mathrm{reg}}$.
\end{itemize}

Thus $X_\ga^\la$ may have more irreducible components than $X_\ga^{\la,\mathrm{reg}}$. This makes it more difficult to reduce geometric properties of $X_\ga^\la$ to the commutative group $G_\ga^0(F)$. 

\subsection{Main results}
Our first goal is to prove a dimension formula of $X_\ga^\la$. 
\begin{thm}\label{thm:dim-formula}
$X_\ga^\la$ and $X_\ga^{\le\la}$ are $k$-schemes locally of finite type, equidimensional with dimension 
\[\dim X_\ga^\la=\dim X_\ga^{\le\la}=\langle\rho,\la\rangle+\frac{1}{2}(d(\ga)-c(\ga))\]
where 
\begin{itemize}
\item $\rho$ is half sum of the positive roots for $G$;
\item $d(\ga)$ is the discriminant valuation of $\ga$ (cf. Definition~\ref{def:disc-valuation});
\item $c(\ga)=\mathrm{rank}(G)-\mathrm{rank}_F(G_\ga)$, the difference between the dimension of the maximal torus of $G$ and the dimension of the maximal $F$-split subtorus of the centralizer $G_\ga$.
\end{itemize}
\end{thm}

In \cite{Bou15} and \cite{BC17}, this theorem is proved when $G$ is semisimple and simply-connected. In this article we prove it for any split connected reductive group. \par

As in the Lie algebra case, there are two major steps. First we prove the dimension formula for the regular open subset, this step generalize the method of Kazhdan-Lusztig in \cite{KL}. The second step is to show that
\[\dim X_\ga^{\la,\mathrm{reg}}=\dim X_\ga^\la\]
For this the argument of Kazhdan-Lusztig in \cite{KL} does not generalize, since otherwise it would imply that the complement of the regular open subset has strictly smaller dimension (see \cite[Proposition 3.7.1]{Ngo10}), which in our situation may not be true due to the possible existence of irregular components. In general, actually most components of $X_\ga^\la$ will be irregular, see Remark~\ref{regular-components-rem}. Instead, we
bypass this diffulty by studying the global analogue of Kottwitz-Viehmann varieties, the Hitchin-Frenkel-Ng\^o fibration. Similar ideas occured previously in \cite{BC17}.\par 
This major difference from Lie algebra case lead us naturally to the question of determining the number of irreducible components of $X_\ga^\la$, which is our second goal. We will formulate a conjecture on the number of irreducible components of $X_\ga^\la$ and prove the conjecture in the case where $\ga$ is an unramified (or split) conjugacy class. One formulation of the conjecture involves the Newton point $\nu_\ga\in (X_*(T)\otimes\bQ)^+$ of $\ga$, which is an element in the dominant rational coweight cone.
By the discussion in \S\ref{irr-components-conjecture-section}, if $X_\ga^\la$ is nonempty, there exists a unique \emph{smallest} dominant \emph{integral} coweight $\mu$ such that $\nu_\ga\le_\bQ\mu$ and $\mu\le\la$.
\begin{conjecture*}[Conjecture~\ref{irr-components-conjecture}]
Let $\mu$ be as above. The number of $G_\ga^0(F)$-orbits on the set of irreducible components of $X_\ga^\la$ equals to $m_{\la\mu}$, which is the dimension of $\mu$-weight space in the irreducible representation $V_\lambda$ of the Langlands dual group $\hat G$ with highest weight $\lambda$.  
\end{conjecture*}
We remark that there is a similar conjecture made by Miaofen Chen and Xinwen Zhu on the irreducible components of affine Deligne-Lusztig varieties, see \citep{HaVi17} and \citep{XiaoZhu17} for statements.\par 
In fact we will also give a conceptually better formulation of this Conjecture using the extended Steinberg base of Vinberg monoid. See Conjecture~\ref{irr-components-conjecture} for more details.
\begin{thm}
The Conjecture is true if $\ga\in G(F)^{\mathrm{rs}}$ is split.
\end{thm}
This is proved in Corollary~\ref{dim-unr-cor}.

\begin{rem}
Although we restrict to equal characteristic local field, we expect that most results involving only local arguments in this paper could also be generalize to mixed characteristic Kottwitz-Viehmann varieties, which could be defined based on the work of X.Zhu \cite{Zhu17}. However, the dimension formula in full generality involves global argument and currently it's not clear how to generalize this to mixed characteristic case. It would be interesting to see if there is a purely local argument to prove dimension formula.
\end{rem}
\subsection{Organization of the article}
In \S\ref{chapter:monoid}, we review certain facts needed from the theory of reductive monoids. In \S\ref{chapter:Sp-fibre}, we prove dimension formula and the conjecture on irreducible components in the unramified case. In \S~\ref{chapter:HFN}, we review basic facts of Hitchin-Frenkel-Ng\^o fibration. The main result we establish in this chapter is properness of the fibration over anisotropic open subset.
In \S~\ref{chapter:global-to-local}, we relate Kottwitz-Viehmann varieties and Hitchin-Frenkel-Ng\^o fibrations and finish the prove of dimension formula for $X_\ga^\la$. 

\subsection{Notations and conventions}\label{notation-subsection}
\subsubsection{Group theoretic notations}
We assume throughout the article that $k$ is an algebraically closed field. $F=k((\varpi))$ and $\cO=k[[\varpi]]$. We let $G$ be a (split) connected reductive group over $k$. Assume that either $\mathrm{char}(k)=0$ or $\mathrm{char}(k)>0$ does not divide the order of Weyl group of $G$.\par 
Denote by $G_{\mathrm{der}}$ the derived group of $G$, a semisimple group of rank $r$. Let $\gsc$ be the simply-connected cover of $G_{\mathrm{der}}$ and $\gad$ the adjoint group of $G$.
\par 
Fix a maximal torus $T$ of $G$ and a Borel subgroup $B$ containing $G$. Let $\Delta=\{\alpha_1,\dotsc,\alpha_r\}$ be the set of simple roots determined by $T\subset B$. Let $\check{\Lambda}:=X^*(T)$ (resp. $\Lambda:=X_*(T)$) be the weight (resp. coweight) lattice. Let $\check{\Lambda}^+$ (resp. $\Lambda^+$) be the set of dominant weights (resp. dominant coweights). Let $W$ be the Weyl group of $G$ and $S\subset W$ the set of simple reflections associated to the simple roots $\Delta$. There is a unique longest element $w_0$ of $W$ under the Bruhat order determined by $S$. Then $w_0$ is a reflection and $-w_0$ defines a bijection on the sets $\Delta$, $\Lambda^+$ and $\check\Lambda^+$.\par 
Let $\hat G$ be the Langlands dual group of $G$, viewed as a complex reductive group. For each $\lambda\in\Lambda^+$, viewed as a dominant weight for $\hat G$, let $V(\lambda)$ be the irreducible representation of $\hat G$ with highest weight $\la$. For any $\mu\in\La^+$ with $\mu\le\la$, let $m_{\la\mu}$ be the dimension of $\mu$ weight space in $V(\la)$.

\subsubsection{Scheme theoretic notations}
For a scheme $X$ over $\spec F$, let $LX$ be the loop space of $X$. More precisely, $LX$ is the $k$-space that associates to any $k$-algebra $R$ the set $LX(R)=X(R((t)))$.\par
For a scheme $X$ over $\spec\cO$, let $L_n^+X$ be its $n$-th jet space. In other words, $L_n^+X$ is the $k$-space whose set of $R$ points is $L_n^+X(R)=X(R[t]/t^n)$ for any $k$ algebra $R$. Let $L^+X:=\varprojlim L^+_nX$ be the arc space of $X$.\par 
If $X$ is a $k$-scheme, then we denote $LX:=L(X\otimes_kF)$, $L_n^+X:=L_n^+(X\otimes_k\cO)$  and $L^+X:=L^+(X\otimes_k\cO)$.\par 
For any scheme $X$, we denote by $\mathrm{Irr}(X)$ the set of its irreducible components.
\subsection*{Acknowledgement}
I would like to thank my advisor Ng\^o Bao Ch\^au for his interst in this paper and many helpful discussions. I am grateful to Zhiwei Yun for sending me his unpublished notes on the materials in \S~\ref{sec:singularity-HFN}. I thank Alexis Bouthier, Tsao-Hsien Chen, Xuhua He, Cheng-Chiang Tsai, Zhiwei Yun and Xinwen Zhu for their interest in this work as well as various discussions and suggestions.  Also I thank Jim Humphreys, Wen-Wei Li and Xinwen Zhu for valuable feedbacks on the first draft of this paper. The terminology ``Kottwitz-Viehmann variety" was suggested by B.C. Ng\^o.

\section{Review on reductive monoids}\label{chapter:monoid}
In this section we summarize some results on reductive monoids needed later. We will roughly follow the exposition in \cite{Bou15}, with several modifications and improvements. We refer the reader to \cite{Vin}, \cite{Rit98}, \cite{Rit01} for more backgrounds on this subject.
\subsection{Construction of Vinberg monoid}
In this section, we assume that $G$ is semisimple \emph{simply connected}.\par 
The Vinberg monoid for $G$ is an algebraic monoid $\ving$ such that the derived group of its unit group is isomorphic to $G$, and it is characterized by certain nice universal properties. For our purpose, we construct it in an explicit manner as follows.\par 
Let $\omega_1,\dotsc,\omega_r\in X_*(T)_+$ be the fundamental weights. For each $1\le i\le r$, let $\rho_{\omega_i}:G\to\mathrm{GL}(V_{\omega_i})$ be the irreducible representation with highest weight $\omega_i$.\par 
We introduce the extended group $G_+:=(T\times G)/Z$ where $Z$, the center of $G$, embeds anti-diagonally in $T\times G$. Then $G_+$ is a reductive group with center $Z_+=(T\times Z)/Z\cong T$ and derived group $G$.  Let $T_+=(T\times T)/Z$ be a maximal torus of $G^+$. We extend the representations $\rho_{\omega_i}$ representations of $G_+$: 
\[\xymatrix@R=1pt{
\rho_i^+:G_+\ar[r] & \mathrm{GL}(V_{\omega_i})\\
(t,g)\ar@{|->}[r] & \omega_i(t)\rho_{\omega_i}(g)
}\]
For each $1\le i\le r$, we also extend the simple roots $\alpha_i$ to $\alpha_i^+:G^+\to\bG_m$ by $\alpha_i^+(t,g)=\alpha_i(t)$. Altogether, we get the following homomorphism
\[(\alpha^+,\rho^+):G^+\to\bG_m^r\times\prod_{i=1}^r\mathrm{GL}(V_{\omega_i})\]
\begin{defn}
The \emph{Vinberg monoid} of $G$, denoted by $\vin_{G}$, is the normalization of the closure of $G_+$ in the product 
\[\bA^r\times\prod_{i=1}^r\mathrm{End}(V_{\omega_i}).\]
\end{defn}
Then $\vin_{G}$ is an algebraic monoid with unit group $G_+$. It has a smooth dense open subvariety $\vin_{G}^0$ defined as the inverse image of the following product in $\ving$
\[\bA^r\times\prod_{i=1}^r(\mathrm{End}(V_{\omega_i})-\{0\})\]

\begin{defn}
The \emph{abelianization} of the monoid $\vin_{G}$ is the invariant quotient
\[A_G:=\vin_{G}//(G\times G)\]
Let $\alpha:\vin_{G}\to A_G$ be the quotient map. 
\end{defn}
Using the maps $\alpha^+$ we get a canonical isomorphism
$A_G\cong\bA^r$.
The adjoint torus $\tad$ embeds via the simple roots as the open subset where all the $r$-coordinates are nonzero.\par 
Note that the fibers of $\alpha$ over points in $\tad$ are isomorphic to $G$. One can construct a canonical section of the abelianization map $\alpha$ as follows.\par 
Let $T_{\mathrm{diag}}$ be the image of the diagonal embedding $T\to T_+$. Then there is a canonical isomorphism $T_{\mathrm{diag}}\cong\tad$ which extends to an isomorphism $\overline{T_{\mathrm{diag}}}\cong A_G$ between the closure of $T_{\mathrm{diag}}$ in $\ving$ and $A_G$. The inverse of this isomorphism defines a section of the abelianization map $\alpha$, which we denote by
\begin{equation}\label{eq:can-sec-alpha}
\mathfrak{s}:A_G\to\ving
\end{equation}

The group $G_+\times G_+$ acts by left and right multiplication on $\ving$. More precisely, for all $(x,y)\in G_+\times G_+$ and $\ga\in\ving$, the action is given by $(x,y)\cdot\ga=x\ga y^{-1}$. The $G_+\times G_+$-orbits on $\ving$ corresponds bijectively to pairs $(I,J)$ of subsets of $\Delta$ such that no connected component (in the sense of Dynkin diagram) of the complement of $J$ is entirely contained in $I$. Each orbit $O_{I,J}$ containes an idempotent $e_{I,J}\in\ving$, defined up to conjugation. We can choose $e_{I,J}\in\overline{T_+}$, the closure of $T_+$ in $\ving$. Then it is well-defined up to $W$-conjugation.\par 
Fix such a pair $(I,J)$. Let $J^c$ be the complement of $J$ in $\Delta$ and $J^0$ be the interior of $J$, i.e. the elements in $J$ that is not connected to any element of $J^c$ in the Dynkin diagram. Let $M:=I\cap J^0\sqcup J^c$. Let $P_+(M)$ be the corresponding standard parabolic subgroup of $G_+$, $P_+(M)^-$ be the opposite of $P_+(M)$ and $L(M)$ their common Levi subgroup. Denote by $\delta:P_+(M)\to L_+(M)$ and $\delta_-:P_+(M)^-\to L_+(M)$ the canonical projections. The following lemma is \cite[Thm 21]{Rit97}:
\begin{lem}
The stabilizer of $e_{I,J}$ under $G_+\times G_+$ is the subgroup of $P_+(M)\times P_+(M)_-$ consisting of pairs $(g,g_-)$ such that 
\[\delta(g)\equiv\delta(g_-)\mod L_+(J^c)_{\mathrm{der}}T_{I,J}\]
where $T_{I,J}$ is certain subtorus of $T_+$. 
\end{lem}

\subsubsection{}
The adjoint action of $G$ on the Vinberg monoid $\ving$ is the restriction of left and right multiplication by $G\times G$ along the diagonal. In other words, for any $g\in G$ and $\ga\in\ving$, the adjoint action is given by $\mathrm{Ad}(g)(\ga):=g\ga g^{-1}$. Note that this action factors through the adjoint group $G_{\mathrm{ad}}$.\par 
For any $\ga\in\ving$, we let $G_\ga$ be the centralizer of $\ga$ in $G$, i.e. the stabilizer of $\ga$ under the adjoint action of $G$. If $\ga\in G_+$ belongs to the unit group of $\ving$, we know that $\dim G_\ga\ge\dim T=r$. By upper-semicontinuity of stabilizer dimension (cf. \cite[VI B.4, Prop. 4.1]{SGA3}), we see that $\dim G_\ga\ge r$ for all $\ga\in\ving$. 
\begin{defn}
An element $\ga\in\ving$ is \emph{regular} if $\dim G_\ga=r$ (i.e. smallest possible). 
Let $\ving^{\mathrm{reg}}\subset\ving$ be the open subset consisting of regular elements.
\end{defn}
\begin{defn}
The \emph{extended Steinberg base} is defined to be the invariant quotient $\fC_+:=\ving//\mathrm{Ad}(G)$. We denote the canonical quotient map by $\chi_+:\ving\to\fC_+$.
\end{defn}
The functions $\alpha_i^+$ define a canonical map $\beta:\fC_+\to A_G$ so that $\alpha=\beta\circ\chi_+$. The following result is \cite[Proposition 1.7]{Bou15}:
\begin{thm}
The closed embedding $\overline{T_+}\subset\ving$ induces an isomorphism $\overline{T_+}//W\cong\fC_+$. Moreover, the functions $\alpha_+$ and $\mathrm{Tr}(\rho_i^+)$ define isomorphism
\[\fC_+\cong A_G\times\bA^r\cong\bA^{2r}.\]
\end{thm}
The canonical projection $q:\overline{T_+}\to\fC_+$ a finite flat, generically Galois \'etale with Galois group $W$.

\subsection{Adjoint orbits}
We keep the assumption that $G$ is semisimple simply-connected. \par 
Let $S=\{s_1,\dotsc,s_r\}$ be the set of simple reflections in $W$ corresponding to our choice of simple roots $\Delta$. Let $l:W\to\bN$ be the length function determined by $S$. For any subsets $J\subset S$, let $W_J$ be the subgroup of $W$ generated by elements in $J$. Let $W^J$ (resp. ${}^JW$) be the set of minimal length representatives of the cosets $W/W_J$ (resp. $W_J\backslash W$). For any two subsets $J_1,J_2$ of $S$, denote ${}^{J_1}W^{J_2}:={}^{J_1}W\cap W^{J_2}$.\par
 For each $w\in W$, let $\mathrm{Supp}(w)\subset S$ be the subset consisting of those simple reflections which occurs in one (and hence every) reduced word expression of $w$. 
\begin{defn}\label{coxeter-definition}
An element $w\in W$ is called an \emph{$S$-Coxeter element} if it can be written as products of simple reflections in $S$, each occuring precisely once. In particular, $l(w)=r$ and $\mathrm{Supp}(w)=S$. Denote by $\mathrm{Cox}(W,S)$ the set of $S$-Coxeter elements in $W$. 
\end{defn}
In general, an element $w\in W$ is called a \emph{Coxeter element} if it is conjugate to an $S$-Coxeter element in $W$.\par 
 Let $\cN:=\chi_+^{-1}(0)$ be the \emph{nilpotent cone} in the Vinberg monoid $\ving$. Let $\cN^0:=\cN\cap\ving^0$ and $\cN^{\mathrm{reg}}:=\cN\cap \ving^{\mathrm{reg}}$ be the corresponding open subsets.
\subsubsection{}
Our approach in this part follows a suggestion of Xinwen Zhu. For any subset $J\subset\Delta$, denote $J^c:=\Delta\setminus J$, then we have
\[O_{\varnothing, J}\cong (G/G_{J^c}U_{J^c}\times G/G_{J^c}U_{J^c}^-)/Z(L_{J^c}).\]
where $Z(L_{J^c})$, the center of the Levi $L_{J^c}$ acts diagonally on the product. There is a canonical map
\[\pi_{\varnothing,J}:O_{\varnothing,J}\to G/P_{J^c}\times G/P_{J^c}^-\]
The diagonal $G$-orbits on the product $G/P_{J^c}\times G/P_{J^c}^-$ corresponds bijectively to ${}^{J^c}W^{J^c}$. The element $w\in{}^{J^c}W^{J^c}$ corresponds to the $G$-orbit of $(\dot{w},1)$ for any representative $\dot{w}$ of $w$ in $G$. We denote this $G$-orbit by $Y_{\varnothing,J,w}$ and let $X_{\varnothing,J,w}$ be its inverse image under $\pi_{\varnothing,J}$. Then we have
\begin{equation}\label{eq:X-empty-J-w}
X_{\varnothing,J,w}=\mathrm{Ad}(G)(Z(L_{J^c})\dot{w}e_{\varnothing,J}).
\end{equation}
The $G$-orbit $Y_{\varnothing,J,w}$ has codimension $l(w)$ in $G/P_{J^c}\times G/P_{J^c}^-$. Hence we have

\begin{equation}\label{eq:dim-X-empty-J-w}
\begin{split}
\dim X_{\varnothing,J,w}&=2\dim(G/P_{J^c})-l(w)+\dim Z(L_{J^c})\\
&=\dim G-\dim L_{J^c}-l(w)+|J|.
\end{split}
\end{equation}

\begin{lem}\label{lem:N-strata}
$X_{\varnothing,J,w}\subset\cN$ if and only if $J\subset\mathrm{Supp}(w)$.
\end{lem}
\begin{proof}
First suppose $X_{\varnothing,J,w}\subset\cN$. Then in particular $\dot{w}e_{\varnothing,J}\in\cN$. Recall that the idempotent $e_{\varnothing,J}$ acts as projector to highest weight space in the representation $V_{\omega_i}$ if $i\in J$ and acts by $0$ if $i\notin J$. If there exists $j\in J$ but $j\notin\mathrm{Supp}(w)$, then $\rho_{\omega_j}(\dot{w})$ preserves the highest weight space in $V_{\omega_j}$ and hence $\mathrm{Tr}(\rho_{\omega_j}(\dot{w}e_{\varnothing,J}))\ne0$, contradiction the assumption that $\dot{w}e_{\varnothing,J}\in\cN$.\par 
Conversely suppose that $J\subset\mathrm{Supp}(w)$. Let $x=t\dot{w}e_{\varnothing,J}$ where $t\in Z(L_{J^c})\subset T$. Then $\rho_{\omega_i}(x)=0$ if $i\notin J$. If $i\in J$, so $i\in\mathrm{Supp}(w)$, then by a standard result in root system we have $w(\omega_i)\ne\omega_i$ (see, for example \cite[Lemma 3.5]{HT}). Thus we have $\mathrm{Tr}(\rho_{\omega_i}(x))=0$ as $t\in T$ preserve the weight spaces and $\dot{w}$ maps the highest weight space into the weight space with weight $w(\omega_i)$. Thus $x\in\cN$.
\end{proof}
\begin{cor}\label{cor:N-strata}
\begin{enumerate}
\item[(a)] There is a stratification of $\cN$ into $\mathrm{Ad}(G)$-stable pieces
\[\cN=\bigsqcup_{J\subset\Delta}\bigsqcup_{\substack{w\in{}^{J^c}W^{J^c}\\ \mathrm{Supp}(w)\supset J}}X_{\varnothing,J,w}.\]
\item[(b)] $\cN^0=\bigsqcup\limits_{w\in W}X_{\varnothing,\Delta,w}$.
\item[(c)] For each $w\in\mathrm{Cox}(W,S)$(cf. Definition~\ref{coxeter-definition}), $X_{\varnothing,\Delta,w}$ is a single $\mathrm{Ad}(G)$-orbit and
$\cN^{\mathrm{reg}}=\bigsqcup\limits_{w\in\mathrm{Cox}(W,S)}X_{\varnothing,\Delta,w}$. In particular $\cN^{\mathrm{reg}}\subset\cN^0$.
\item[(d)] $\dim\cN=\dim\cN^{\mathrm{reg}}=\dim G-r$ and the dimension of the complement $\cN\setminus\cN^{\mathrm{reg}}$ is strictly less than $\dim\cN$. 
\end{enumerate}
\end{cor}
\begin{proof}
Part (a) and (b) are immediate from Lemma~\ref{lem:N-strata}.
For each strata $X_{\varnothing,J,w}\subset\cN$ as in Lemma~\ref{lem:N-strata}, we have $l(w)\ge |J|$ since $J\subset\mathrm{Supp}(w)$. From \eqref{eq:dim-X-empty-J-w} we see that 
\[\dim X_{\varnothing,J,w}\ge\dim G-\dim L_{J^c}\ge\dim G-r\] 
and equality is reached precisely when $J=\Delta$ and $l(w)=r$. This condition means that $w\in\mathrm{Cox}(W,S)$. Hence part (d) follows from part (c).\par 
It remains to show that for each $w\in\mathrm{Cox}(W,S)$, $X_{\varnothing,\Delta,w}$ is a single $\mathrm{Ad}(G)$-orbit. By \eqref{eq:X-empty-J-w}, we have
\[X_{\varnothing,\Delta,w}=\mathrm{Ad}(G)(Twe_{\varnothing,\Delta}).\]
So it suffices to show that for each $t\in T$, the elements $t\dot{w}e_{\varnothing,\Delta}$ and $\dot{w}e_{\varnothing,\Delta}$ are conjugate. Since $w$ is a Coxeter element, by \cite[Lemma 7.6]{St65} there exists $s\in T$ such that $t=s^{-1}\dot{w}s\dot{w}^{-1}$. This implies that $s^{-1}\dot{w}e_{\varnothing,\Delta}s=t\dot{w}e_{\varnothing,\Delta}$ since $s,t\in T$ and hence commute with $e_{\varnothing,\Delta}$.
\end{proof}
\begin{rem}
Another way to show that $X_{\varnothing,\Delta,w}$ consists of a single $\mathrm{Ad}(G)$-orbit is to show that the centralizer of $we_{\varnothing,\Delta}$ in $G$ has dimension $r$, i.e. $we_{\varnothing,\Delta}\in\cN^{\mathrm{reg}}$. For then the $\mathrm{Ad}(G)$-orbit of $we_{\varnothing,\Delta}$ is contained in the irreducible set $X_{\varnothing,\Delta,w}$ and 
has the same dimension, thus equals to $X_{\varnothing,\Delta,w}$. 
\end{rem}
\begin{cor}\label{cor:chi-flat}
The morphism $\chi_+:\ving\to\fC_+$ is flat.
\end{cor}
\begin{proof}
There exists a nonempty open subset $U\subset\fC_+$ such that the fibres of $\chi_+$ over $U$ have dimension $\dim\ving-\dim\fC_+=\dim G-r$. Since $\chi_+$ is $Z_+$ equivariant, $U$ is $Z_+$-stable. By Corollary~\ref{cor:N-strata}(d) we know that $0\in U$ and hence we have $U=\fC_+$.  By \cite[6.2.9]{BK}, $\ving$ is Cohen-Macaulay. Moreover, $\fC_+\cong\bA^{2r}$ is regular and hence $\chi_+$ is flat.
\end{proof}

\begin{cor}\label{cor:vin-reg-in-v0}
$\ving^{\mathrm{reg}}\subset\ving^0$.
\end{cor}
\begin{proof}
Let $F:=\ving^{\mathrm{reg}}\setminus\ving^0$. By Corollary~\ref{cor:N-strata}(c), we have $\cN^{\mathrm{reg}}\subset\cN^0$ and hence $F\cap\cN=\varnothing$. On the other hand, $F$ is a $Z_+$-stable closed subset of $\ving^{\mathrm{reg}}$, so we must have $F=\varnothing$. 
\end{proof}

\begin{prop}\label{prop:N-components}
The nilpotent cone $\cN$ is connected and equidimensional. Moreover, there exist bijections 
\[\mathrm{Cox}(W,S)\xrightarrow{\sim}\mathrm{Irr}(\cN^{\mathrm{reg}})\xrightarrow{\sim}\mathrm{Irr}(\cN^0)\xrightarrow{\sim}\mathrm{Irr}(\cN)\]
which send $w\in\mathrm{Cox}(W,S)$ to the irreducible component containing $\dot{w}e_{\varnothing,\Delta}$. 
\end{prop}
\begin{proof}
Since $\chi_+$ is flat, its fibre $\cN=\chi_+^{-1}(0)$ is equidimensional. Since $\chi_+$ is the invariant quotient under a reductive group, there is a unique closed orbit  in $\cN$, namely $0\in\cN$. In particular, $\cN$ is connected.\par 
From Corollary~\ref{cor:N-strata}, we see that $\cN^{\mathrm{reg}}$ is dense in  $\cN^0$ and $\cN$ and $\mathrm{Irr}(\cN^{\mathrm{reg}})$ is in bijection with $\mathrm{Cox}(W,S)$. Hence $\mathrm{Irr}(\cN^0)$ and $\mathrm{Irr}(\cN)$ are also in bijection with $\mathrm{Cox}(W,S)$.
\end{proof}
\begin{rem}
It is worthwhile to compare with the Lie algebra case. The fibers of the Chevalley map $\fg\to\fc$ are irreducible and $\fg^{\mathrm{reg}}$ is the union of the unique open $G$-orbit in each fiber. 
\end{rem}

\subsubsection{Discriminant divisor}\label{sec:discr-divisor}
Recall that on $T$ we have the discriminant function
\[\mathrm{Disc}(t):=\prod_{\alpha\in\Phi}(1-\alpha(t))\]
which is $W$-equivariant and descends to a regular function on the Steinberg base $\fC:=T//W$. We extend the function $\mathrm{Disc}$ to a function $\mathrm{Disc}_+$ on $T_+=(T\times T)/Z_G$ by 
\[\mathrm{Disc}_+(t_1,t_2):=2\rho(t_1)\mathrm{Disc}(t_2).\]
Then $\mathrm{Disc}_+$ extends to a regular function on $\overline{T_+}$, which further descends to a regular function on $\fC_+$. The vanishin loci of $\mathrm{Disc}_+$ is a principal divisor on $\fC_+$ which we call \emph{extended discriminant divisor} and denote by $\fD_+$.\par 
From the definition, we see that $\mathrm{Disc}_+$ is an eigenfunction for the $Z_+$-action on $\overline{T_+}$ and $\fC_+$, with eigen-value $2\rho$. Hence the subschemes $\fD_+$ is $Z_+$-invariant.\par
For $t_+=(t,t^{-1})\in T_{\mathrm{diag}}\subset T_+$, we have
\begin{equation}
\begin{split}
D_+(t_+)&=2\rho(t)\prod_{\alpha\in\Phi_+}(1-\alpha(t))(1-\alpha(t^{-1}))\\
&=(-1)^{|\Phi_+|}\prod_{\alpha\in\Phi_+}(1-\alpha(t))^2
\end{split}
\end{equation}
For each $\alpha\in\Phi_+$, $D_\alpha:=(1-\alpha(t))^2$ extends to a polynomial function on $\overline{T_{\mathrm{diag}}}\cong\bA^r$.

\subsubsection{Adjoint orbits in extended Steinberg fibre}
An element $\ga\in\ving$ is called \emph{semisimple} if it is $G$-conjugate to an element in $\overline{T_+}$. Let $\ving^{\mathrm{rs}}$ be the subset of $\ving$ consisting of elements that are \emph{both regular and semisimple}. 
\begin{lem}\label{lem:centralizer-semisimple-Levi}
The centralizer of any semisimple element $\ga\in\ving$ in $G$ is a Levi subgroup of $G$. 
\end{lem}
\begin{proof}
We may assume that $\ga\in\overline{T_+}$ so that $\ga=te_{I,J}$ for some $t\in T_+$ and idempotent $e_{I,J}$.\par 
For any $g\in G_+$, we have $g\ga g^{-1}=\ga$ if and only if 
\[t^{-1}gt e_{I,J}g^{-1}=e_{I,J}.\]
By the description of the stabilizer of $e_{I,J}$ under the action of $G_+\times G_+$, we see that $g\in (G_+)_\ga$ if and only if  the following 2 conditions are satisfied:
\begin{itemize}
\item $(t^{-1}gt,g)\in P_M\times P_M^-$;
\item $\delta(t^{-1}gt)\delta_-(g)^{-1}\in (L_{J^c})_{\mathrm{der}}T_{I,J}$.
\end{itemize}
Here $M:=I\cap J^0\sqcup J^c$. Since $t\in L_M$, the first condition implies that $g\in L_M$. Since the roots in $I\cap J^0$ and $J^c$ are orthogonal to each other, the second condition implies that $(G_+)_\ga$ is the subgroup of $L_M$ generated by $T_+$, $L_{J^c}$ and the centralizer of $t$ in $L_{I\cap J^0}$. This shows that $(G_+)_\ga$ is a Levi subgroup of $G_+$ and hence $G_\ga$ is a Levi subgroup of $G$.
\end{proof}
\begin{lem}\label{lem:orbit-in-Steinberg-fibre}
For any closed point $c\in\fC_+$, the fibre $\chi_+^{-1}(c)$ is connected and equidimensional of dimension $\dim G-r$. 
The open $\mathrm{Ad}(G)$-orbits in $\chi_+^{-1}(c)$ are precisely the regular conjugacy classes in $\chi_+^{-1}(c)$. On the other hand, there is a \emph{unique} closed $\mathrm{Ad}(G)$-orbit in $\chi_+^{-1}(c)$ which is also the unique semisimple conjugacy class in $\chi_+^{-1}(c)$.
\end{lem}
\begin{proof}
By Corollary~\ref{cor:chi-flat}, $\chi_+$ is flat. Hence $\chi_+^{-1}(c)$ is equidimensional of dimension $\dim G-r$. 
Since $\chi_+$ is the invariant quotient by the reductive group $G$, there is a unique closed orbit in $\chi_+^{-1}(c)$. This closed orbit is connected since $G$ is connected. Consequently $\chi_+^{-1}(c)$ is also connected.\par
The regular conjugacy classes in $\chi_+^{-1}(c)$ are locally closed subsets of the same dimension as $\chi_+^{-1}(c)$. Hence they are precisly the open $\mathrm{Ad}(G)$-orbits in $\chi_+^{-1}(c)$.\par 
Finally by \cite{Ren88}, closed $\mathrm{Ad}(G)$-orbits are precisely the semisimple conjugacy classes.
\end{proof}

Unlike the group case, there might be more than one regular conjugacy class in an extended Steinberg fibre $\chi_+^{-1}(c)$, as we see in Proposition~\ref{prop:N-components} for the nilpotent cone $\cN=\chi_+^{-1}(0)$. On the other hand, regular semisimple conjugacy classes are the only $\mathrm{Ad}(G)$ orbit in the extended Steinberg fibre they live in. We give another characterization of regular semisimple conjugacy classes using the discriminant function $\mathrm{Disc}_+$.
The following is a generalization of \cite[2.19]{Bou15}
\begin{prop}\label{prop:rs-disc-nonzero}
Denote $\overline{T_+}^{\mathrm{reg}}:=\overline{T_+}\cap\ving^{\mathrm{reg}}$. For any $\ga\in\overline{T_+}$, the following are equivalent:
\begin{enumerate}
\item $\ga\in\overline{T_+}^{\mathrm{reg}}$;
\item $\mathrm{Disc}_+(\ga)\ne 0$;
\item The map $q:\overline{T_+}\to\fC_+$ is \'etale at $\ga$;
\item $G_\ga=T$.
\end{enumerate}
\end{prop}
\begin{proof}
(1)$\Rightarrow$(2): Suppose $\ga\in\overline{T_+}^{\mathrm{reg}}$. By Corollary~\ref{cor:vin-reg-in-v0}, we have $\ga\in\ving^0\cap\overline{T_+}$. After conjugation and multiplying by the center $Z_+$, we may assume that $\ga\in\overline{T_{\mathrm{diag}}}$. If $\mathrm{Disc}_+(\ga)=0$, then there exists $\alpha\in\Phi_+$ such that $D_\alpha(\ga)=0$. This implies that $\ga$ lies in the closure of the diagonal embedding of $\ker(\alpha)$. Since the centralizers of elements in $\ker(\alpha)$ have dimension at least $r+1$, the same is true for $G_\ga$ by upper semicontinuity of centralizer dimension. This contradicts the assumption that $\ga$ is regular and we must have $\mathrm{Disc}_+(\ga)\ne0$.\par 
(1)$\Leftrightarrow$(3)$\Leftrightarrow$(4): Since $\fC_+=\overline{T_+}//W$, the finite cover $q:\overline{T_+}\to\fC_+$ is \'etale at $\ga$ if and only if the stabilizer of $\ga$ in $W$ is trivial, which is equivalent to the fact $G_\ga=T$ since $G_\ga$ is a standard Levi subgroup of $G$ by the proof of Lemma~\ref{lem:centralizer-semisimple-Levi}.\par 
(2)$\Rightarrow$(1): Let $V\subset\overline{T_+}$ be the open subset where $\mathrm{Disc}_+$ is nonzero and we need to show that $V=\overline{T_+}^{\mathrm{reg}}$. In the implication ``(1)$\Rightarrow$(2)" we proved that $\overline{T_+}^{\mathrm{reg}}\subset V$.\par 
Consider the stratification of $\overline{T_+}$ induced by the $\tad$-orbits on $A_G=\bA^r$. The open strata is $T_+$, the unit group of $\overline{T_+}$. The codimension $1$ stratas are described as follows: for each $1\le i\le r$, let $\cO_i$ be the codimension $1$ strata consisting of $x\in\overline{T_+} $ such that the $i$-th coordinate of $\alpha(x)$ vanishes and the other coordinates are nonzero. Consider the complement $F:=V\setminus\overline{T_+}^{\mathrm{reg}}$, which is a closed subset of $V$. It is a classical fact that $F\cap T_+=\varnothing$. Also, we have $e_{\varnothing,\Delta}\in\overline{T_+}^{\mathrm{reg}}$ by direct calculation of its centralizer. Hence $e_{\varnothing,\Delta}$ lies in the closure $\overline{\cO_i}$ for all $1\le i\le r$. This shows that the generic point of $\cO_i$ lies in $\overline{T_+}^{\mathrm{reg}}$ for all $i$, which implies that $F$ has codimension at least 2 in $\overline{T_+}$. But by the equivalence ``(1)$\Leftrightarrow$(3)" we just proved and purity of branch locus (see, for example \cite[\href{http://stacks.math.columbia.edu/tag/0BMB}{Tag 0BMB}]{stacks-project}), the complement $\overline{T_+}\setminus\overline{T_+}^{\mathrm{reg}}$ is pure of codimension 1 in $\overline{T_+}$. This forces $F$, an open subset of $\overline{T_+}\setminus\overline{T_+}^{\mathrm{reg}}$ to be empty and hence $V=\overline{T_+}^{\mathrm{reg}}$.
\end{proof}

\begin{cor}\label{cor:ving-rs}
$\ving^{\mathrm{rs}}=\chi_+^{-1}(\fC_+\setminus\fD_+)$. Moreover, $G$ acts transitively on each fibre of $\chi_+$ over $\fC_+\setminus\fD_+$. 
\end{cor}
\begin{proof}
By Proposition~\ref{prop:rs-disc-nonzero}, we have $\ving^{\mathrm{rs}}\subset\chi_+^{-1}(\fC_+\setminus\fD_+)$.\par  Let $c\in\fC_+\setminus\fD_+$. By Lemma~\ref{lem:orbit-in-Steinberg-fibre} and Proposition~\ref{prop:rs-disc-nonzero}, the unique closed orbit in $\chi_+^{-1}(c)$ is also open. Hence $\chi^{-1}(c)$ is a single $\mathrm{Ad}(G)$-orbit consisting of elements that are both regular and semisimple. This proves the inverse inclusion.
\end{proof}
For this reason, we denote $\fC_+^{\mathrm{rs}}:=\fC_+\setminus\fD_+$ and call it the \emph{regular semisimple} open subset of $\fC_+$.

\subsubsection{Extended Steinberg section}
For each $S$-Coxeter element $w\in\mathrm{Cox}(W,S)$ (cf. Definition~\ref{coxeter-definition}), each choice of representatives $\dot{s}_i\in N_G(T)$ of the simple roots $s_i$, Steinberg defines a section $\epsilon^w:\fC_{G}\to G$ of the adjoint quotient map $\chi_{G}:G\to\fC_{G}$. Moreover, it is shown that the equivalence class of $\epsilon^w$ depends neither on $w$ nor the choices $\dot{s}_i$, see \cite[7.5 and 7.8]{St65}. Here we say that two sections $\epsilon,\epsilon'$ are \emph{equivalent} if 
for all $a\in\fC_{G}$, $\epsilon(a)$ and $\epsilon'(a)$ are conjugate under $G$.\par 
Following \citep{Bou15}, we extend the Steinberg sections $\epsilon^w$ to the Vinberg monoid $\ving$ as follows. For each $(b,a)\in\fC_+\cong\bA^{2r}$ where $b\in A_G\cong\bA^r$, define a map
\[\epsilon_+^w:\fC_+\to\ving\]
by $\epsilon_+^w(b,a):=\epsilon^w(a)\mathfrak{s}(b)$
where $\mathfrak{s}:A_G\to\ving$ is the section of the abelianization map $\alpha$ defined in \S~\ref{eq:can-sec-alpha}.
\begin{prop}\label{extend-steinberg-section-prop}
The map $\epsilon_+^w$ is a section of the adjoint quotient $\chi_+: \vin_G\to \fC_+$. Moreover, the image of $\epsilon_+^w$ is contained in $\vin_G^{\mathrm{reg}}$.
\end{prop}
\begin{proof}
The first statement is \cite[Proposition 1.10 ]{Bou15}. The second statement is Proposition 1.16 in \textit{loc. cit.}
\end{proof}
\begin{rem}
For each $w\in\mathrm{Cox}(W,S)$, the equivalence class of the extended section $\epsilon_+^w$ is independant of the choice of representatives $\dot{s}_i$ of the simple reflections. However, for two \emph{different} $w,w'\in\mathrm{Cox}(W,S)$, the sections $\epsilon_+^w$ and $\epsilon_+^{w'}$ are not equivalent since, as we will see, $\epsilon_+^w(0)$ and $\epsilon_+^{w'}(0)$ are not conjugate.
\end{rem}
Next we examine the interaction of the extended Steinberg section $\epsilon_+^w$ with the action of the central torus $Z_+$.\par 
To this end, we drop the semisimple simply connected assumption and allow $G$ to be any connected reductive group. Then the adjoint action of $\gad$ on $\vin_{\gsc}$ induces an action of $G$ on $\vin_{\gsc}$ which we also denote by ``$\mathrm{Ad}$".  Let $\fC_+=\vin_{\gsc}//\mathrm{Ad}(G)=\vin_{\gsc}//\mathrm{Ad}(\gsc)$ be the extended Steinberg base for $\vin_{\gsc}$. The central torus $Z_+^{\mathrm{sc}}=\tsc$ acts naturally on $\vin_{\gsc}$ and $\fC_+$ such that the morphism $\chi_+:\vin_{\gsc}\to\fC_+$ is $\tsc$-equivariant. Hence $\chi_+$ induces a morphism between stacks
\begin{equation}
[\chi_+]:[\vin_{\gsc}/(\mathrm{Ad}(G)\times\tsc)]\to [\fC_+/\tsc]
\end{equation}
We would like to see if $\epsilon_+^w$ induces a section $[\chi_+]$. It turns out that this is not true in general. To remedy it we consider the homomorphism $\psi:\tsc\to\gad$ defined as the following composition
\begin{equation}
\psi:\tsc\xrightarrow{\omega_\bullet}\bG_m^r\xrightarrow{\mathfrak{s}}\gsc_+\to\gad
\end{equation}
where the first arrow is $\omega_\bullet:=(\omega_1,\dotsc,\omega_r)$, the second arrow is induced by the canonical section of the abelianization $\alpha$ (cf. Equation~\ref{eq:can-sec-alpha}) and the third arrow is the canonical quotient morphism.\par 
Consider the action of $\tsc\times\tsc$ on $\vin_{\gsc}$ where the first copy of $\tsc$ acts by composing $\psi$ with the adjoint action of $\gad$ and the second copy of $\tsc$ acts as central torus. In \cite[Proposition 1.11]{Bou15}, by examining the action on weight vectors of fundamental representations, it is shown that for all $a_+\in\fC_+$ and $z\in\tsc$ we have
\[\epsilon_+^w(z\cdot a_+)=z\cdot\mathfrak{s}(\omega_\bullet(z))\epsilon_+^w(a_+)\mathfrak{s}(\omega_\bullet(z))^{-1}\]
This shows that $\epsilon_+^w$ is equivariant with respect to the diagonal embedding $\tsc\to\tsc\times\tsc$ and hence induces a morphism
 \[[\fC_+/\tsc]\to[\vin_{\gsc}/\psi(\tsc)\times\tsc]\]
If $G=\gad$, then this leads to a section $[\epsilon_+^w]$ of $[\chi_+]$. In general, let $c=|Z(G_{\mathrm{der}})|$ be the order of the center of the derived group $G_{\mathrm{der}}$. Then by extracting $c$-th roots, we would get a lifting $\psi_{[c]}:\tsc\to G_{\mathrm{der}}\subset G$ of $\psi$. More precisely, $\psi_{[c]}$ is defined by the following commutative diagram
\[\xymatrix{
\tsc\ar[r]^{\psi_{[c]}}\ar[d]^c & G\ar[d]\\
\tsc\ar[r]^\psi & \gad
}\]
where the left vertical map is raising to $c$-th power.\par 
The $c$-th power map $\tsc\to\tsc$ induces a morphism between classifying stacks $\bB\tsc\to\bB\tsc$. Base changing $[\chi_+]$ along this map, we obtain a Cartesian diagram
\[\xymatrix{
[\vin_{\gsc}/(\mathrm{Ad}(G)\times\tsc)]_{[c]}\ar[r]\ar[d]_{[\chi_+]_{[c]}} & [\vin_{\gsc}/(\mathrm{Ad}(G)\times\tsc)]\ar[d]^{[\chi_+]}\\
[\fC_+/\tsc]_{[c]}\ar[r] & [\fC_+/\tsc]
}\]
where on the left, the $\tsc$ action is the composition of the $c$-th power map and the usual action.
\begin{prop}\label{prop:twisted-Steinberg-section}
The map $\epsilon_+^w$ induces a section $\epsilon_{+,[c]}^w:[\fC_+/\tsc]_{[c]}\to[\vin_{\gsc}/(\mathrm{Ad}(G)\times\tsc)]_{[c]}$ of $[\chi_+]_{[c]}$ whose image lies in the open substack 
\[[\vin_{\gsc}^{\mathrm{reg}}/(\mathrm{Ad}(G)\times\tsc)]_{[c]}\]
\end{prop}
\begin{proof}
By what we have discussed, $\epsilon_+^w$ induces a morphism
\[[\fC_+/\tsc]_{[c]}\to[\vin_{\gsc}/\psi_{[c]}(\tsc)\times\tsc]\]
where on the right, the second copy of $\tsc$ acts by composing the $c$-th power map and the usual action. Since $\psi_{[c]}(\tsc)\subset G$, there is a canonical morphism
\[[\vin_{\gsc}/\psi_{[c]}(\tsc)\times\tsc]\to[\vin_{\gsc}/(\mathrm{Ad}(G)\times\tsc)]_{[c]}.\]
Composing the two morphisms above we obtain the morphism $\epsilon_{+,[c]}^w$ with the desired property.
\end{proof}
\subsection{Regular centralizer for the group}
In this section we let $(G,G')$ be a pair of connected reductive groups equipped with an isomorphism of their derived groups $\gad\cong\gad'$. Assume moreover that the derived group of $G$ is simply connected. Then there is a natural adjoint action of $G'$ on $G$ and the action factors through $\gad'\cong\gad$.  
Let $\fC_G:=G//\mathrm{Ad}(G')$ be the invariant quotient. Then there is a canonical isomorphism $\fC_G\cong T//W$. The natural map $T\to\fC_G$ is finite flat and its restriction to $\fC_G^{\mathrm{rs}}$ is a Galois \'etale cover with Galois group $W$.\par 
Consider the centralizer group scheme $I_{G'}$ over $G$ defined by
\[I_{G'}:=\{(g,x)\in G'\times G | \mathrm{Ad}(g)x=x\}.\]
In other words, the fiber of $I_{G'}$ over $x\in G$ is the centralizer $G_x'$ of $x$ in $G'$. Since the derived group of $G$ is simply connected, $G_x'$ is connected for semisimple $x\in G$ If moreover $x\in G^{\mathrm{rs}}$ is regular semisimple, then $G_x'$ is a maximal torus in $G'$. More generally, the restriction $I_{G'}|_{G^{\mathrm{reg}}}$ to the regular open subscheme $G^{\mathrm{reg}}$ is a smooth commutative group scheme of relative dimension $\dim(T)$. The following lemma is the group version of \cite[Lemme 2.1.1]{Ngo10}

\begin{lem}\label{lem: reg-centralizer-group}
There exists a unique smooth commutative group scheme $J_{G'}$ over $\fC_G$ such that we have a $G'$-equivariant isomorphism
\[(\chi^*J_{G'})_{G^{\mathrm{reg}}}\cong I_{G'}|_{G^{\mathrm{reg}}}.\]
Moreover, this isomorphism extends uniquely to a homomorphism
$\chi^*J_{G'}\to I_{G'}$.
\end{lem}
\begin{proof}
The proof of \cite[Lemme 2.1.1]{Ngo10} generalize verbatim to our situation. For the last statement, we use the fact that the complement of $G^{\mathrm{reg}}$ in $G$ has codimension at least 2, c.f. \cite{St65}.
\end{proof}
Fix a maximal torus $T'\subset G'$. Consider the Weil restriction of the torus $T'\times T$ on $T$ to $\fC_G$:
\[\Pi_G:=\Pi_{T/\fC_G}(T'\times T).\]
In other words, for any $\fC$-scheme $S$, we have
\[\Pi_G(S)=\mathrm{Hom}_{T}(S\times_C T,T'\times T)\]
The diagonal action of $W$ on $T'\times T$ induces an action of $W$ on $\Pi_G$. The fixed point subscheme of  $\Pi_G^W$ is a closed smooth subscheme of $\Pi_G$ since the characteristic of the base field does not divide the order of $W$.
\begin{prop}
There exists a canonical open embedding $J_{G'}\to\Pi_G^W$.
\end{prop}
\begin{proof}
We follow the argument for the Lie algebra case in \cite[\S 2.4]{Ngo10}. First we define a morphism $J\to\Pi_G^W$. By adjunction, this is the same as giving a morphism $q^*J\to T\times T$ where $q:T\to\fC$ and we view $T\times T$ as a constant group scheme over $T$. One construct this morphism by descent along the smooth morphism $\tilde\chi^{\mathrm{reg}}:\widetilde{G}^{\mathrm{reg}}\to T$ which sits in the Cartesian diagram
\[\xymatrix{
\widetilde{G}^{\mathrm{reg}}\ar[r]^{\tilde{q}}\ar[d]_{\tilde{\chi}} & G^{\mathrm{reg}}\ar[d]^{\chi}\\
T\ar[r]^q & \fC_G
}\]
 Hence it suffices to construct a $G$-equivariant morphism $(\tilde{\chi}^\mathrm{reg})^*q^*J_G\to T\times\widetilde{G}^{\mathrm{reg}}$. The upshot is that for all $x\in G$ and Borel subgroup $x\in B\subset G$, we have $I_x\subset B$ by the argument of \cite[Lemme 2.4.3]{Ngo10}. Hence when composed with the quotient $B\to T$, we obtain a map $I_x\to T$ depending on the choice of Borel $B$ containing $x$. Thus we get the desired morphism $(\tilde{\chi}^\mathrm{reg})^*q^*J_G\cong \tilde{q}^*I_{G^{\mathrm{reg}}}\to T\times\widetilde{G}^{\mathrm{reg}}$ which is $G$-equivariant by construction.\par 
 To show that the morphism $J_G\to\Pi_G^W$ constructed above is an isomorphism, it suffices to show the isomorphism over an open subset of $\fC$ whose complement has codimension at least 2.\par 
For each simple root $\alpha\in\Phi^+$, let $T_\alpha$ be the kernel of $\alpha$, which is a subscheme of codimension 1 in $T$. Then the discriminant divisor $\mathfrak{D}\subset\fC$ is the union of $q(T_\alpha)$ for all simple root $\alpha$. Let $T^\circ_\alpha\subset T_\alpha$ be the open subscheme consisting of points that does not lie in $T_\beta$ for any $\beta\ne\alpha$. Then 
\[\fC^{\mathrm{rs}}\cup (\bigcup_{\alpha\in\Phi^+}q(T^\circ_\alpha))\] 
is an open subset of $\fC$ whose complement has codimension 2. It follows from construction that it is an isomorphism over $\fC^{\mathrm{rs}}$. Hence it remains to show that $J_G\to \Pi_G^W$ is an isomorphism when restricted to $q(T^\circ_\alpha)$ for each positive root $\alpha$.\par 
Let $t\in T_\alpha^\circ$ and we will show that $J\to\Pi_G^W$ is an isomorphism in an \'etale neighbourhood of $t$. Let $G_\alpha$ be the centralizer of $T_\alpha$ in $G$ and $\fC_{G_\alpha}$ its adjoint quotient. Then the natural morphism $\pi_\alpha:\fC_{G_\alpha}\to\fC$ is \'etale in a neighbourhood of $q_\alpha(t)$ where $q_\alpha: T\to\fC_{G_\alpha}$ is the natural map. This implies that in an \'etale neighbourhood of $q_\alpha(t)$ the group schemes $\Pi_G^W\times_\fC \fC_{G_\alpha}$ and $\Pi_{G_\alpha}^{s_\alpha}$ are isomorphic. \par 
There is a natural open embedding $G^{\mathrm{reg}}\cap G_\alpha\subset G^{\mathrm{reg}}_\alpha$. Consider the open subset 
\[\fC_\alpha^{G-\mathrm{reg}}:=\chi_\alpha(G^{\mathrm{reg}}\cap G_\alpha)\]
As $t\in T_\alpha^\circ$, one can choose a unipotent element $u\in G_\alpha$ such that $tu\in G^{\mathrm{reg}}\cap G_\alpha$. In particular, $q_\alpha(t)\in\fC_\alpha^{G-\mathrm{reg}}$. It is clear that 
\[I_{G_\alpha}|_{G^{\mathrm{reg}}\cap G_\alpha}\cong I_G|_{G^{\mathrm{reg}}\cap G_\alpha}\]
This implies that $(\pi^*_\alpha J_G)|_{\fC_\alpha^{G-\mathrm{reg}}}\cong (J_{G_\alpha})|_{\fC_\alpha^{G-\mathrm{reg}}}$.\par 
In summary, the base change of $J_G$ and $\Pi_G^W$ to an \'etale neighbourhood of $q(t)$ are isomorphic to the corresponding groups defined for the group $G_\alpha$. Note that by assmption, $G_\alpha$ is of rank 1 and has semisimple derived group, thus isomorphic to the product of a torus with either $\mathrm{GL}_2$ or $\mathrm{SL}_2$. So we are finally reduced to the case of $\mathrm{GL}_2$ and $\mathrm{SL}_2$, on which the isomorphism follows by direct calculation.
\end{proof}

\subsection{Regular centralizer for Vinberg monoid}
In this section we let $G$ be an arbitrary connected reductive group over $k$. Let $\gsc$ be the simply-connected cover of its derived group. Then there is a natural adjoint action of $G$ on $\vin_{\gsc}$ and the action factors through $\gad$.\par 

Consider the centralizer group scheme $\cI$ over $\vin_{\gsc}$ defined by
\[\cI=\{(g,\gamma)\in G\times \vin_{\gsc} | \mathrm{Ad}(g)\gamma = \gamma\}\]
Then $\cI|_{\vin_{\gsc}^{\rm{reg}}}$ is smooth of relative dimension $r$. By \cite{Ren88}, the fibres of $\cI$ over $\vin_{\gsc}^{\mathrm{rs}}$ are maximal tori in $G$. In particular, $\cI|_{\vin_{\gsc}^{\mathrm{rs}}}$ is commutative. Hence $\cI|_{\vin_{\gsc}^{\rm{reg}}}$ is also commutative.
\subsubsection{Open cover of regular locus}
For each $w\in\mathrm{Cox}(W,S)$, define $\cJ^w:=(\epsilon_+^w)^*\cI$. Then $\cJ^w$ is a smooth commutative group scheme on $\fC_+$. 
The morphism
\[\xymatrix@R=1pt{
c_w: G\times\fC_+\ar[r] & \vin_{\gsc}^{\mathrm{reg}}\\
(g,a)\ar@{|->}[r] & g\epsilon_+^w(a)g^{-1}
}\]
factors through $(G\times \fC_+)/\cJ^w$ and induces a quasi-finite morphism
\[\bar{c}_w:(G\times \fC_+)/\cJ^w\to \vin_{\gsc}^{\mathrm{reg}}\]
Since $\bar{c}_{w}$ is an isomorphism over $G_+^{sc,\mathrm{reg}}$, it is birational. Since $\vin_{\gsc}^{\mathrm{reg}}$ is normal, $\bar{c}_{w}$ is an open embedding by Zariski Main Theorem.\par 
Denote by $\vin_{\gsc}^{w}$ the image of $\bar{c}_{w}$, which is an open subscheme of $\vin_{\gsc}^{\mathrm{reg}}$.  The union $U:=\bigcup_{w\in\mathrm{Cox}(W,S)}\vin_{\gsc}^w$ is a $Z^{\mathrm{sc}}_+$-stable open subset of $\vin_{\gsc}^{\mathrm{reg}}$. By Proposition~\ref{prop:N-components}, it coincides with $\vin_{\gsc}^{\mathrm{reg}}$ over $0\in\fC_+$. Hence it equals to $\vin_{\gsc}^{\mathrm{reg}}$. In other words, the sets $\vin_{\gsc}^{w}$ form an open cover of 
$\vin_{\gsc}^{\mathrm{reg}}$:

\begin{equation}\label{open-cover-V-reg-eq}
\vin_{\gsc}^{\mathrm{reg}}=\bigcup_{w\in\mathrm{Cox}(W,S)}\vin_{\gsc}^w
\end{equation}
We generalize Lemma~\ref{lem: reg-centralizer-group} to $\vin_{\gsc}$:
\begin{lem}\label{lem:reg-centralizer-Vin}
There is a unique smooth commutative group scheme $\cJ$ over $\fC_+$ such that we have a $G$-equivariant isomorphism $(\chi_+^\mathrm{reg})^*\cJ\cong\cI|_{\vin_{\gsc}^{\mathrm{reg}}}$. Moreover, this isomorphism extends uniquely to a homomorphism $\chi_+^*\cJ\to\cI$. 
\end{lem}
\begin{proof}
By the same argument as Lemma~\ref{lem: reg-centralizer-group}, for each $w\in\mathrm{Cox}(W,S)$, $\cJ^w$ is the unique commutative smooth group scheme over $\fC_+$ such that 
\[(\chi_+^*\cJ^w)|_{\vin_{\gsc}^w}\cong\cI|_{\vin_{\gsc}^w}\]
Next we show that for any $w,w'\in\mathrm{Cox}(W,S)$, the group schemes $\cJ^w$ and $\cJ^{w'}$ are canonically isomorphic. It suffices to show that they are canonically isomorphic over certain open subset whose complement has codimension at least 2. From Lemma~\ref{lem: reg-centralizer-group}, we have the isomorphism over the open subset $\fC_{\gsc_+}$. Over $\fC_+^{\mathrm{rs}}$, each fiber of $\chi_+$ consists of a single $\mathrm{Ad}(G)$ orbit by Lemma~\ref{lem:orbit-in-Steinberg-fibre}. In other words, $G$ acts transitively on each fibre of $\chi_+$ over $\fC_+^{\mathrm{rs}}$. Hence $\vin_{\gsc}^{\mathrm{rs}}\subset \vin_{\gsc}^w$ for all $w\in\mathrm{Cox}(W,S)$. Thus by uniqueness of $\cJ^w$ we see that $\cJ^w$ and $\cJ^{w'}$ are isomorphic over $\fC_+^{\mathrm{rs}}$.\par 
The complement of $\fC_{\gsc_+}$ is the union of the closure of codimension 1 stratas in $\fC_+$. Since the idempotent $e_{\varnothing}$ is regular semisimple and belongs each of the strata closure we see that on each strata, the regular semisimple locus is nonempty open. Hence the complement of $\fC_{\gsc_+}\cup\fC_+^{\mathrm{rs}}\subset\fC_+$ has codimension at least 2. \par 
Consequently there is a unique commutative smooth group scheme $\cJ$ over $\fC_+$ which comes with a unique isomorphism $(\chi_+^\mathrm{reg})^*\cJ\cong\cI|_{\vin_{\gsc}^{\mathrm{reg}}}$. We know from Lemma~\ref{lem: reg-centralizer-group} that this isomorphism extends uniquely to a homomorphism between $\chi_+^*\cJ$ and $\cI$ over the open subset $\gsc_+\cup\vin_{\gsc}^{\mathrm{reg}}$ whose complement has codimension at least 2. Hence it extends further to the whole space $\vin_{\gsc}$.
\end{proof}

\begin{prop}\label{BJ-action-vin-G-prop}
The classifying stack $\bB\cJ$ acts naturally on $[\vin_{\gsc}/\mathrm{Ad}(G)]$. The action preserves the open substacks $[\vin_{\gsc}^0/\mathrm{Ad}(G)]$, $[\vin_{\gsc}^\mathrm{reg}/\mathrm{Ad}(G)]$ and $[\vin_{\gsc}^w/\mathrm{Ad}(G)]$ for each $w\in\mathrm{Cox}(W,S)$. Moreover, the morphism
\[[\chi_+^w]: [\vin_{\gsc}^w/\mathrm{Ad}(G)]\to\fC_+\]
induced by $\chi_+$ is a $\bB\cJ$ gerbe, neutralized by the extended Steinberg section $\epsilon_+^w$.
\end{prop}
The proof is the same as \cite[Proposition 2.2.1]{Ngo10}.
\begin{prop}
The number of irreducible components of the fibers of the map 
\[\chi_+^{\mathrm{reg}}:\vin_{\gsc}^{\mathrm{reg}}\to\fC_+\]
is bounded above by $|\mathrm{Cox}(W,S)|$ and equality is achieved at $\cN^{\mathrm{reg}}=(\chi_+^{\mathrm{reg}})^{-1}(0)$.
\end{prop}
\begin{proof}
The first statement follows from \eqref{open-cover-V-reg-eq}. The second statement is in Proposition~\ref{prop:N-components}.
\end{proof}
\begin{rem}
Consequently, unless all simple factors of $\gsc$ are $\mathrm{SL}_2$, the action of $\bB\cJ$ on $[\vin_{\gsc}^{\mathrm{reg}}/\mathrm{Ad}(G)]$ is not transitive. In other words, $[\vin_{\gsc}^{\mathrm{reg}}/\mathrm{G}]$ is not a $\bB\cJ$-gerbe, but rather a finite union of $\bB\cJ$ gerbes as in Proposition~\ref{BJ-action-vin-G-prop}. This is different from Lie algebra situation, cf \citep[Proposition 2.2.1]{Ngo10}.
\end{rem}

\subsubsection{Galois description of universal centralizer}
Let $\prod\limits_{\overline{\tsc_+}/\fC_+}(T\times \overline{\tsc_+})$ be the restriction of scalar which associates to any $\fC_+$-scheme $S$ the set
\[\prod_{\overline{T_+}/\fC_+}(T\times\overline{T_+})(S)=\Hom_{\overline{T_+}}(S\times_{\fC_+}\overline{T_+},T\times \overline{T_+})\]
Then $W$ acts diagonally on $\prod\limits_{\overline{T_+}/\fC_+}(T' \times \overline{T_+})$ and we consider its fixed point subscheme
\[\cJ^1:=(\prod_{\overline{T_+}/\fC_+}^{}T\times \overline{T_+})^W.\]
The following is proved in \cite[Proposition 11]{Bou17}.
\begin{prop}\label{Galois-description-J-prop}
$\cJ^1$ is a smooth commutative group scheme over $\fC_+$. Moreover, there exists an open embedding $\cJ\to\cJ^1$ whose restriction to $\fC_+^{\mathrm{rs}}$ is an isomorphism.
\end{prop}

\subsection{Arc space of Vinberg monoid}
In this section, we assume that $G$ is semisimple and simply connected.\par 
For each $\la\in X_*(\tad)_+$, define the affine scheme $\ving^\la$ over $\spec\cO$ by the following Cartesian diagram
\[\xymatrix{
\ving^\la\ar[r]\ar[d] & \ving\times\tad\ar[d]\\
\spec\cO\ar[r]^{\varpi^{-w_0(\la)}} & A_G
}\]
where the right vertical arrow is the product of the abelianization map $\alpha_G$ and the natural embedding $\tad\into A_G$, the lower horizontal arrow corresponds to the point $\varpi^{-w_0(\la)}\in A_G(\cO)$. Replacing $\ving$ by its open subscheme $\ving^0$ in the above diagram, we define an open subscheme $\ving^{\la,0}\subset\ving^\la$.

There is a stratification of the space of nondegenerate arcs of $A_G\supset\tad$ by $\tad(\cO)$ orbits:
\[A_G(\cO)\cap\tad(F) = \bigsqcup_{\la\in X_*(\tad)_+}\tad(\cO)\varpi^{-w_0(\la)}\]
The inverse image of $\tad(\cO)\varpi^{-w_0(\la)}$ under the abelianization map is precisely $L^+\ving^\la(k)$. In other words we get a stratification of the space of nondegenerate arcs of $\ving\supset G_+$ into $G_+(\cO)$-stable pieces:
\[\ving(\cO)\cap G_+(F)=\bigsqcup_{\la\in X_*(\tad)_+}L^+\ving^\la(k)\]
Also we note that
\[L^+\ving^{\la,0}(k)=L^+\ving^\la(k)\cap \ving^0(\cO).\]

\begin{lem}\label{lem:vin-strata-double-coset}
For any $g_+\in G_+(F)$, we have $g_+\in L^+\ving^\la$ if and only if $\alpha(g_+)\in\varpi^{-w_0(\la)}\tad(\cO)$ and the image of $g_+$ in $\gad(F)$ belongs to 
\[\overline{\gad(\cO)\varpi^\la\gad(\cO)}=\bigcup_{\substack{\mu\in X_*(\tad)^+\\
\mu\le\la}} \gad(\cO)\varpi^\mu\gad(\cO).
\]
Moreover, $g_+\in L^+\ving^{\la,0}$ if and only if $\alpha(g_+)\in\varpi^{-w_0(\la)}\tad(\cO)$ and the image of $g_+$ in $\gad(F)$ belongs to the double coset $\gad(\cO)\varpi^{\la}\gad(\cO)$.
\end{lem}
\begin{proof}
The coweight lattice for $T_+$ can be expressed as
\[X_+(T_+)=\{(\la_1,\la_2)\in X_*(\tad)\times X_*(\tad)|\la_1+\la_2\in X_*(T)\}\]
For $(\la_1,\la_2)\in X_*(T_+)$, we have $\varpi^{(\la_1,\la_2)}\in L^+\ving^\la$ if and only if 
\begin{itemize}
\item $\alpha(\varpi^{(\la_1,\la_2)})\in\varpi^{-w_0(\la)}\tad(\cO)$ and
\item The matrix $\rho_{\omega_i}^+(\varpi^{(\la_1,\la_2)})\in\mathrm{End}(V_{\omega_i})$ has entries in $\cO$ for all $1\le i\le r$.
\end{itemize}
Since $\alpha(\varpi^{(\la_1,\la_2)})=\varpi^{\la_1}$, the first condition means that $\la_1=-w_0(\la)$. Then the second condition means that 
\[\langle(-w_0(\la),\la_2), \chi_+\rangle\ge 0\]
for all $1\le i\le r$ and all weights $\chi_+$ in the $G_+$-representation $\rho_{\omega_i}^+$. Since the weights of the representation $\rho_{\omega_i}^+$ lie in the convex hull of the $W$-orbit of the highest weight $(\omega_i,\omega_i)$ where $W$ acts on the second factor, the above inequality is equivalent to 
\[\langle-w_0(\la),\omega_i\rangle+\langle\la_2,w(\omega_i)\rangle=\langle(-w_0(\la),\la_2), (\omega_i,w(\omega_i))\rangle\ge0\]
for all $w\in W$ and $1\le i\le r$. This can be further reformulated as 
\[\langle\la-w(\la_2),\omega_i\rangle\ge0\] for all $w\in W$ and $1\le i\le r$.\par 
By the discussion so far, we have
\[L^+\ving^\la(k)\cap T_+(F)=\bigcup_{\substack{\mu\in X_*(\tad)\\ \mu_{\mathrm{dom}}\le\la}}\varpi^{(-w_0(\la),\mu)}T_+(\cO)\]
where $\mu_{\mathrm{dom}}$ denotes the unique dominant coweight in the $W$-orbit of $\mu$. As $L^\la\ving$ is stable under the action of $G_+(\cO)\times G_+(\cO)$, it is a union of $G_+(\cO)$ double cosets in $G_+(F)$. Thus by Cartan decomposition we get
\[L^+\ving^\la(k)=\bigsqcup_{\substack{\mu\in X_*(\tad)^+\\ \mu\le \la}}G_+(\cO)\varpi^{(-w_0(\la),\mu)}G_+(\cO)\]
Similarly we can get a description of $L^\la\ving^0$. The difference is that we require furthermore that $\rho_{\omega_i}^+(\varpi^{(-w_0(\la),\la_2)})$ have nonzero reduction mod $\varpi$ for all $1\le i\le r$. Hence besides the inequality $\langle\la-w(\la_2),\omega_i\rangle\ge0$ for all $w\in W$ and $1\le i\le r$, we require furthermore that for each $i$, there exists $w\in W$ such that $\langle\la-w(\la_w),\omega_i\rangle=0$. This condition means that implies that $\la_2$ is in the $W$-orbit of $\la$ and hence
\[L^+\ving^{\la,0}(k)=G_+(\cO)\varpi^{(-w_0(\la),\la)}G_+(\cO).\]
From these description the lemma follows.
\end{proof}

\begin{lem}\label{lem:vinG-approx}
Suppose $n\ge b(\la):=\max\limits_{1\le i\le r}\langle\la,\omega_i-w_0(\omega_i)\rangle$. Then for all $\ga,\ga'\in L^+\ving^\la(k)$ having the same image in $\ving(\cO/\varpi^n\cO)$, there exists $g\in G_+(\cO)$ such that $\ga'=\ga g$.
\end{lem}
\begin{proof}
The following argument is due to Zhiwei Yun. Let $i\mapsto i^*$ be the involution on the set $\{1,\dotsc,r\}$ such that $\omega_{i^*}=-w_0(\omega_i)$. For each $1\le i\le r$, there is a natural pairing between $V_i$ and $V_{i^*}$ such that for all $x\in G_+$, $v\in V_i$ and $v^*\in V_{i^*}$, we have
\[\langle\rho_i^+(x)v,\rho_{i^*}^+(x)v^*\rangle=(\omega_i+\omega_{i^*})(\alpha(x))\langle v,v^*\rangle.\]
Thus for each $x\in G_+(F)$, under the natural pairing above, the lattice $\rho_i^+(x)V_i(\cO)$ in $V_i(F)$ is dual to the lattice 
\[(\omega_i+\omega_{i^*})(\alpha(x)^{-1})\rho_{i^*}^+(g)V_{i^*}(\cO)\subset V_{i^*}(F).\]
For $\ga\in L^+\ving^\la\subset\ving(\cO)$, we have $\rho_i^+(\ga)V_i(\cO)\subset V_i(\cO)$ for all $1\le i\le r$. Taing duals, we get
\[V_{i^*}(\cO)\subset \varpi^{-\langle\la,\omega_i+\omega_{i^*}\rangle}\rho_{i^*}^+(\ga)V_{i^*}(\cO).\] 
In other words, we have shown that for all $1\le i\le r$,
\[\varpi^{\langle\la,\omega_i+\omega_{i^*}\rangle}V_i(\cO)\subset\rho_i^+(\ga)V_i(\cO)\subset V_i(\cO).\]
Thus if $\ga$ and $\ga'$ have the same image in $\ving(\cO/\varpi^n\cO)$ for $n\ge b(\la)$, the lattices $\rho_i^+(\ga)V_i(\cO)$ and $\rho_i(\ga')V_i(\cO)$ are the same and hence $\ga'=\ga g$ for some $g\in G_+(\cO)$.
\end{proof}

\subsection{Local lifting property of extended Steinberg map}
We review certain infinitesimal lifting property of the extended Steinberg morphism $\chi_+$ needed later in \S~\ref{sec:singularity-HFN}. This is based on some result of Gabber-Ramero in \cite{GR}. Our exposition below follow the treatment in \cite{Bou17} and \cite{Yun}.
\subsubsection{}
We start by recalling certain results in \cite[\S5.4]{GR}. Let $A$ be a ring and $B$ an $A$-algebra of finite presentation. Let $f:\spec B\to\spec A$ be the natural morphism. Choose a presentation $B\cong P/J$ where $P:=A[X_1,\dotsc,X_N]$ and $J\subset P$ is a finitely generated ideal. Then the map $f$ is factored as the composition of a closed embedding $i:\spec B\into \bA_A^N$ and the natural projection $p:\bA_A^N\to\spec A$. Define the ideal
\[H_A(P,J):=\mathrm{Ann}_P\mathrm{Ext}_B^1(\mathbb{L}_{B/A},J/J^2)\subset P\]
where $\mathbb{L}_{B/A}$ is the cotangent complex of the morphism $\chi$. Notice that $J\subset H_A(P,J)$.\par 
Consider the ideal in $B$ defined by $H_{B/A}:=H_A(P,J)B=H_A(P,J)/J$. Let $\Sigma_f:=\spec B/H_{B/A}$ be the closed subscheme of $\spec B$ defined by $H_{B/A}$. We remark that $H_{B/A}$ depends on the choice of presentation $B\cong P/J$. The following is \cite[Lemma 5.4.2]{GR}:
\begin{lem}\label{lem:GR}
\begin{enumerate}
\item For all $B$-module $N$, $H_{B/A}$ annihilates $\mathrm{Ext}^1_B(\mathbb{L}_{B/A},N)$.
\item The complement of $\Sigma_f=\spec B/H_{B/A}$ in $\spec B$ is the smooth locus of the morphism $\chi:\spec B\to\spec A$.
\item For any $A$-algebra $A'$, let $B'=B\otimes_AA'$ and $f':\spec B'\to\spec A'$ be the induced morphism. Define the ideal $H_{B'/A'}\subset B'$ in the same way as above, using the presentation of $B'$ induced from $B\cong P/J$. Then we have $H_{B/A}B'\subset H_{B'/A'}$, or in other words $\Sigma_{f'}\subset\Sigma_f\times_ZZ'$.
\end{enumerate}
\end{lem}

\subsubsection{}
When $A=k[[\varpi]]$, we define the \emph{conductor} of $f$ to be the smallest integer $h$ such that $\varpi^h\in H_{B/A}$. Note that the conductor depends on the presentation of $B$.\par 
The next lemma is the key step in establishing the local lifting result. To state it we let $R$ be an artin local $k$-algebra with maximal ideal $\fm$ and $I\subset R$ an ideal with $I\cdot\fm=0$. Suppose $A=R[[\varpi]]$. Let $B=P/J$ be a finitely presented $A$-algebra and $f:\spec B\to\spec A$ the induced morphism as above. Let $f_0$ be the reduction of $f$ mod $\fm$ and $h\ge0$ the conductor of $f_0$. 
\begin{lem}\label{lem:key-lifting}
Suppose $n\ge h$ and $\sigma:\spec A/\varpi^nI\to\spec B$ is a morphism such that the composition $f\circ\sigma$ is the natural embedding $\spec A/\varpi^nI\into\spec A$. Then there exists a section $\tilde\sigma:\spec A\to\spec B$ of $f$ such that the restriction of $\tilde\sigma$ to $\spec A/\varpi^{n-h}I$ coincides with $\sigma$. 
\end{lem}
\begin{proof}
The obstruction of extending $\sigma$ to $\spec A$ is an element $\omega\in\mathrm{Ext}_B^1(\mathbb{L}_{B/A}, \varpi^nI)$ where we view $\varpi^nI$ as a $B$-module via the map $\sigma^*:B\to A/\varpi^nI$. By the definition of conductor $h$, we have $\varpi^h\in H_{B/A}+\fm B$. By Lemma~\ref{lem:GR}(1) and the assumption that $\fm\cdot I=0$ we see that $\varpi^h\omega=0$. This implies that the image of $\omega$ in $\mathrm{Ext}_B^1(\mathbb{L}_{B/A},\varpi^{n-h}I)$ vanishes by noticing that the multiplication map $\varpi^h:\varpi^nI\to\varpi^nI$ can be factored as the composition of the natural embedding $\varpi^nI\into\varpi^{n-h}I$ and an isomorphism $\varpi^{n-h}I\cong\varpi^nI$. Hence we get the desired lifting of the restriction of $\sigma$ to $\spec A/\varpi^{n-h}I$. 
\end{proof}
\subsubsection{}
We apply the general discussion above to the situation where $\spec A=\fC_+$, $\spec B=\ving$ and $f=\chi_+$. Choose a presentation $B=P/J$ where $P=A[X_1,\dotsc,X_N]$ and $J\subset P$ is a finitely generated ideal as above. Recall that we have the discriminant divisor $\fD_+\subset\fC_+$ defined by the extended discriminant function $\mathrm{Disc}_+$. By Corollary~\ref{cor:ving-rs}, $\chi_+$ is smooth over $\fC_+-\fD_+$. Hence $\chi_+(\Sigma_{\chi_+})$ is contained in $\fD_+$ set-theoretically by Lemma~\ref{lem:GR}(2).\par 
Since $\fD_+$ is a principal divisor, there exists a positive integer $m_0$ (depending on the presentation $B\cong P/J$) such that $\chi_+(\Sigma_{\chi_+})\subset m_0\fD_+$ scheme-theoretically. To state the main result in this section, we consider an artin local $k$-algebra $R$ with maximal ideal $\fm$ and let $I\subset R$ be an ideal such that $I\cdot\fm=0$. 

\begin{prop}\label{prop:chi-lifting}
Let $\delta\in\ving(R[[\varpi]])$ and $a_0\in\fC_+(k[[\varpi]])$ the reduction mod $\fm$ of $\chi_+(\delta)$. Let $d:=\mathrm{val}(a_0^*\fD_+)$ be the discriminant valuation of $a_0$ (suppose that $d$ is a finite number). Then for any integer $N\ge m_0d$ and any $a\in\fC_+(R[[\varpi]])$ such that $a\equiv\chi_+(\delta)\mod \varpi^NI$, there exists $\ga\in\ving(R[[\varpi]])$ such that $\chi_+(\ga)=a$ and $\ga\equiv\delta\mod\varpi^{N-m_0d}I$.
\end{prop}
\begin{proof}
Consider the following diagram in which the right square is Cartesian
\[\xymatrix{
& V_a\ar[r]\ar[d]^{\chi_a} & \ving\ar[d]^{\chi_+}\\
\spec (R[[\varpi]]/\varpi^NI)\ar[r]\ar[ru]^\delta & \spec R[[\varpi]]\ar[r]^a& \fC_+
}\]
Also, let $\chi_{a_0}:V_{a_0}\to\spec k[[\varpi]]$ be the reduction mod $\fm$ of $\chi_a$. Let $h$ be the conductor of $\chi_{a_0}$. By Lemma~\ref{lem:GR}(3), we have $\Sigma_{\chi_{a_0}}\subset V_{a_0}\cap\Sigma_{\chi_+}$. Since $\chi_+(\Sigma_{\chi_+})\subset m_0\fD_+$, we have $h\le m_0d$. By Lemma~\ref{lem:key-lifting}, there exists a section $\ga$ of $\chi_a$ such that the restriction of $\ga$ to $\spec R[[\varpi]]/\varpi^{N-m_0d}$ coincides with $\delta$. Thus the element in $\ving(R[[\varpi]])$ determined by $\ga$ is the lifting we want.
\end{proof}
\section{Kottwitz-Viehmann varieties}\label{chapter:Sp-fibre}
We fix a connected reductive group $G$. Let $T\subset G$ be a maximal torus and $\la\in X_*(T)_+$ a dominant coweight. Let $\ga\in G^{\mathrm{rs}}(F)$ be a regular semisimple element.\par 
We study the following sets associated to the pair $(\ga,\la)$, which we both refer to as \emph{Kottwitz-Viehmann varieties}:
\[X_\ga^{\la}=\{g\in G(F)/G(\cO) | \mathrm{Ad}(g)^{-1}(\ga)\in G(\cO)\varpi^\la G(\cO)\}\]
\[X_{\ga}^{\le\la}=\{g\in G(F)/G(\cO) | \mathrm{Ad}(g)^{-1}(\ga)\in \overline{G(\cO)\varpi^\la G(\cO)}\}\]

\subsection{Nonemptiness}
The first immediate question is when the sets $X_{\ga}^{\la},X_{\ga}^{\le\la}$ are nonempty. To answer this we need to recall the notion of Newton points and Kottwitz map.

\subsubsection{Newton Points}\label{Newton-point-section}
Following \cite[\S4]{KoV}, for each $\ga\in G(F)^{\mathrm{rs}}$, one associate a rational dominant coweight $\nu_\ga\in X_*(T)^+_\bQ$, called the \emph{Newton point of $\ga$}.
\begin{defn}\label{def:disc-valuation}
The \emph{discriminant valuation} for $\ga\in G(F)^{\mathrm{rs}}$ is defined by
\[d(\ga):=\mathrm{val}\det(\mathrm{Id}-\mathrm{ad}_\ga:\fg(F)/\fg_\ga(F)\to\fg(F)/\fg_\ga(F))\]
where $\fg$ is the Lie algebra of $G$ and $\ga_\ga$ is the centralizer of $\ga$, i.e. the fixed locus of the adjoint action $\mathrm{ad}_\ga$. 
\end{defn}
\begin{lem}\label{d-ga-newton-lem}
Let $\ga\in G(F)^{\mathrm{rs}}$ and $\nu_\ga\in\La^+_\bQ$ its Newton point. Let $\bar\ga\in T(\bar F)^{\mathrm{rs}}$ 
be a $G(\bar F)$-conjugate of $\ga$ such that $\mathrm{val}(\alpha(\ga))\ge0$ for all positive root $\alpha$.  Then we have
\[d(\ga)=2\sum_{\alpha\in\Phi^+}\mathrm{val}(\alpha(\bar\gamma)-1)-\langle2\rho,\nu_\ga\rangle\]
where we have extended the valuation on $F$ to its separable closure $\bar F$.
\end{lem}
\begin{proof}
From the definition we see that
\[d(\ga)=\sum_{\alpha\in\Phi}\mathrm{val}(\alpha(\ga)-1).\]
Separate the sum over $\Phi$ according to whether $\langle\alpha,\nu_\ga\rangle=0$ or not, then we get
\begin{equation}\label{d-ga-newton-eq}
d(\gamma)=\sum_{\substack{\alpha\in\Phi\\ \langle\alpha,\nu_\ga\rangle=0}}\mathrm{val}(\alpha(\bar\gamma)-1)+
\sum_{\substack{\alpha\in\Phi\\ \langle\alpha,\nu_\ga\rangle<0}}\langle\alpha,\nu_\ga\rangle.
\end{equation}

By our assumption that $\mathrm{val}(\alpha(\ga))\ge0$ for $\alpha\in\Phi^+$, the first term in \eqref{d-ga-newton-eq} equals to 
\[2\sum_{\substack{\alpha\in\Phi^+\\ \langle\alpha,\nu_\ga\rangle=0}}\mathrm{val}(\alpha(\gamma)-1)=2\sum_{\alpha\in\Phi^+}\mathrm{val}(\alpha(\gamma)-1)\]
while the second term of \eqref{d-ga-newton-eq} equals to
\[\sum_{\alpha\in\Phi^-}\langle\alpha,\nu_\ga\rangle=-\sum_{\alpha\in\Phi^+}\langle\alpha,\nu_\ga\rangle=-\langle2\rho,\nu_\ga\rangle.\]
Hence the lemma follows.
\end{proof}

\subsubsection{Kottwitz map}
Let $\pi_1(G):=X_*(T)/X_*(\tsc)$ be the quotient of the coweight lattice by the coroot lattice and $p_G:X_*(T)\to\pi_1(G)$ be the canonical projection. Following \cite{KoV}, one defines a group homomorphism
\[\kappa_G: G(F)\to\pi_1(G)\]
which we refer to as Kottwitz homomorphism. Note that in \emph{loc. cit.}, this map is denoted by $w_G$.
\begin{lem}\label{lem:ga-la-lem}
Suppose that $\kappa_G(\ga)=p_G(\la)$. Then there exists an element $\ga_\la\in\gsc_+(F)$ such that 
\begin{itemize}
\item the image of $\ga_\la$ in $\gad(F)$ coincides with the image of $\ga$ in $\gad(F)$;
\item $\alpha(\ga_\la)=\varpi^{-w_0(\la_{\mathrm{ad}})}\in\tad(F)\cap A_{\gsc}(\cO)$ where $\lad\in X_*(\tad)^+$ is the image of $\la\in X_*(T)^+$.
\end{itemize}
Moreover, $\ga_\la$ is uniquely determined up to multiplication by an element in $Z_{\gsc}(F)$.
\end{lem}
\begin{proof}
Let $\ga_{\mathrm{ad}}\in\gad(F)$ be the image of $\ga$. Choose any $\tilde{\ga}\in\gsc_+(F)$ that maps to $\ga_{\mathrm{ad}}$. Suppose $\alpha(\tilde{\ga})\in\varpi^\mu\tad(\cO)$ for $\mu\in X_*(\tad)^+$.
By the assumption $\kappa_G(\ga)=p_G(\la)$, we have $\la_{\mathrm{ad}}-\mu\in X_*(T)$. Let $\ga_\la:=\varpi^{\la_{\mathrm{ad}}-\mu}\tilde{\ga}$ where we view $\varpi^{\la_{\mathrm{ad}}-\mu}\in T(F)=Z_+(F)$ as a central element in $\gsc_+(F)$. Then we have $\alpha(\ga_\la)=\varpi^{-w_0(\la_{\mathrm{ad}})}$ and the image of $\ga_\la$ in $\gad(F)$ equals to $\ga_{\mathrm{ad}}$.\par 
Suppose $\ga_\la,\ga_\la'\in\gsc_+(F)$ both satisfy the requirement of the Lemma. Then $\ga_\la'\ga_\la^{-1}\in\gsc(F)$ and its image in $\gad(F)$ is the identity. Hence $\ga_\la'\ga_\la^{-1}\in Z_{\gsc}(F)$.
\end{proof}

Now we can state the non-emptiness criterions.
\begin{prop}\label{prop:nonempty}
The following are equivalent:
\begin{enumerate}
\item $X_{\ga}^\lambda$ is nonempty;
\item $X_{\ga}^{\le\lambda}$ is nonempty;
\item $\kappa_G(\ga)=p_G(\la)$ and $\nu_\ga\le_\bQ\la$, i.e. $\la-\nu_\ga$ is a $\bQ$-linear combinition of simple coroots with non-negative coefficients;
\item $\kappa_G(\ga)=p_G(\la)$ and $\chi_+(\ga_\la)\in\fC_+(\cO)$, where $\ga_\la\in\gsc_+(F)$ is defined in Lemma~\ref{lem:ga-la-lem}.
\end{enumerate}
\end{prop}
\begin{proof}
The implication ``(1)$\Rightarrow$(2)" is tautological. The implication ``(1)$\Rightarrow$(3)" is done in \cite[Corollary 3.6]{KoV}.\par 
(3)$\Rightarrow$(4): Let $F'/F$ be a finite extension of degree $e$ so that $\ga$ (and hence $\ga_\la$) is split in $G(F')$. Let $\varpi'=\varpi^{\frac{1}{e}}$ be a uniformizer of $F'$ and $\cO'=k[[\varpi']]\subset F'$ be the ring of integers. Then $e\cdot\nu_\ga\in X_*(T)_+$ and $\ga$ is $G(F')$-conjugate to an element in $(\varpi')^{e\cdot\nu_\ga}T(\cO')$. From (3) we deduce that $\ga_\la$ is $\gsc_+(F')$-conjugate to an element in $\vin_\gsc(\cO')$. Therefore
\[\chi_+(\ga_\la)\in \fC_+(\cO')\cap\fC_+(F)=\fC_+(\cO).\]

(2)$\Rightarrow$(4): Let $g\in X_{\ga}^{\le\la}$. Then $\mathrm{Ad}(g)^{-1}(\ga)\in G(\cO)\varpi^\mu G(\cO)$ for some $\mu\in X_*(T)_+$ with $\mu\le\la$. Then we have $\omega_G(\ga)=p_G(\mu)=p_G(\la)$. 
In particular we can define the element $\ga_\la\in\gsc_+(F)$ as in Lemma~\ref{lem:ga-la-lem}. Then by Lemma~\ref{lem:vin-strata-double-coset} we have $\mathrm{Ad}(g)^{-1}(\ga_\la)\in L^\la\vin_{\gsc}\subset\vin_{\gsc}(\cO)$. Thus $\chi_+(\ga_\la)\in\fC_+(\cO)$.

(4)$\Rightarrow$(1):
Let $a_+:=\chi_+(\ga_\la)$. So $a\in\fC_+(\cO)$ by condition (4). Then for any Coxeter element $w\in\mathrm{Cox}(W,S)$ (cf. Definition~\ref{coxeter-definition}), we have $\epsilon_+^w(a)\in \vin_{\gsc}^0(\cO)$. It remains to show that there exists $h\in G(F)$ such that $\mathrm{Ad}(h)^{-1}(\ga_\la)=\epsilon_\la^w(a)$, for then $h\in X_{\ga}^\la$. To see this, notice that the transporter from $\ga$ to $\epsilon_+^w(a)$ in $G$ is a torsor under the torus $G_{\ga_\la}$ over $F$. Any such torsor is trivial since $H^1(F,G_{\ga_\la})$ by a theorem of Steinberg (using the fact that the residue field $k$ is algebraic closed). Thus the transporter has an $F$-point $h\in G(F)$. 
\end{proof}

\subsection{Ind-scheme structure}
\subsubsection{First approach}
We will equip the sets $X_{\ga}^\la$ and $X_{\ga}^{\le\la}$ with an ind-scheme structure. We present two approaches, one based on the original definition, the other using Vinberg monoid.\par 
Let $\mathrm{Gr}_{G}:=LG/L^+G$ be the affine Grassmanian for $G$, which are known to be ind-projective ind-scheme over $k$. The positive loop group $L^+G$ acts by left multiplication on $\mathrm{Gr}_G$. Let $(LG)_\la:= L^+G\varpi^\la L^+G$ (resp. $(LG)_{\le\la}$) be the $k$-scheme whose set of $k$-points is  $G(\cO)\varpi^\la G(\cO)$ (resp. $\overline{G(\cO)\varpi^\la G(\cO)}$). 

\begin{defn}
Let $\mathcal{X}_{\ga}^\la$ be the $k$-functor which associates to any $k$-algebra $R$ the set
\[\mathcal{X}_\ga^\la(R)=\{g\in\mathrm{Gr}_{G}(R) | g^{-1}\ga g\in (LG)_\la(R)\}.\]
Also, we define the $k$-functor $\mathcal{X}_{\ga}^{\le\la}$ by replacing $(LG)_\la$ with $(LG)_{\le\la}$ in the above definition
\end{defn}
By definition, $\mathcal{X}_{\ga}^{\le\la}$ is a closed sub-indscheme of $\mathrm{Gr}_G$ and $\mathcal{X}_{\ga}^\la$ is an open sub-indscheme of $\mathcal{X}_{\ga}^{\le\la}$. Let $X_\ga^\la$ (resp. $X_\ga^{\le\la}$) be the reduced structure of $\mathcal{X}_\ga^\la$ (resp. $\mathcal{X}_\ga^{\le\la}$).

\subsubsection{Second approach}
Now we use Vinberg monoids to define certain analogue of affine Springer fibers, which turns out to be isomorphic to Kottwitz-Viehmann varieties. \par 
Let $\fC_+$ be the extended Steinberg base for the monoid $\vin_{\gsc}$. Let $a\in\fC_+(\cO)\cap\fC_+(F)^{\mathrm{rs}}$ and suppose that 
\[\beta(a)\in \varpi^{-w_0(\la_{\mathrm{ad}})}\tad(\cO)\subset A_{\gsc}(\cO)\cap\tad(F)\]
where $\la_{\mathrm{ad}}\in X_*(\tad)$ is the image of $\la\in X_*(T)$. Moreover, let $\ga_+\in\gsc_+(F)$ be an element such that $\chi_+(\ga_+)=a$.

\begin{defn}
The \emph{generalized affine Springer fibre} $\mathcal{S}p_{G,\ga_+}$ associates to any $k$-algebra $R$ the set of isomorphism classes of pairs $(h,\iota)$ where $h$ is the horizontal arrow in the following commutative diagram
\[\xymatrix{
\spec R[[\varpi]] \ar[r]^h\ar[rd]_a & [\vin_{\gsc}/\mathrm{Ad}(G)]\ar[d]\\
 & \fC_+
}\] 
and $\iota$ is an isomorphism between the restriction of $h$ to $\spec R((\varpi))$ and the composition
\[\spec R((\varpi))\xrightarrow{\ga_+} \vin_{\gsc}\to [\vin_{\gsc}/\mathrm{Ad}(G)].\]
Also, we define $k$-functors $\mathcal{S}p_{G,\ga_+}^0$ (resp. $\mathcal{S}p_{G,\ga_+}^{\mathrm{reg}}$) by replacing $\vin_{\gsc}$ with $\vin_{\gsc}^0$ (resp. $\vin_{\gsc}^{\mathrm{reg}}$).
\end{defn}

By definition $\mathcal{S}p_{G,\ga_+}$ is a closed sub-indscheme of $\mathrm{Gr}_{G}$ and $\mathcal{S}p_{G,\ga_+}^{\mathrm{reg}}\subset\mathcal{S}p_{G,\ga_+}^0$ are its open sub-indschemes. We let $\Sp_{G,\ga_+}$ (resp. $\Sp_{G,\ga_+}^0$, $\Sp_{G,\ga_+}^{\mathrm{reg}}$) be the reduced structures of $\mathcal{S}p_{G,\ga_+}$ (resp. $\mathcal{S}p_{G,\ga_+}^0$, $\mathcal{S}p_{G,\ga_+}^{\mathrm{reg}}$).\par 
The isomorphism class of $\Sp_{G,\ga_+}$ and $\Sp_{G,\ga_+}^0$ only depends on $a=\chi_+(\ga_+)$, so we will also denote them by $\Sp_{G,a}$ and $\Sp_{G,a}^0$.
 We will simplify notation as $\Sp_{\ga_+}$, $\Sp_a$ etc. if the group $G$ is clear from the context. \par 
Next we relate the two definitions given above. Let $(\ga,\la)$ be as in the beginning of this chapter. Suppose that the ind-scheme $X_{\ga}^\la$ is nonempty. Then by Proposition~\ref{prop:nonempty} we have $\kappa_G(\ga)=p_G(\la)$ and $a:=\chi_+(\ga_\la)\in\fC_+(\cO)$ where $\ga_\la\in\gsc_+(F)$ is defined in Lemma~\ref{lem:ga-la-lem}. It is not hard to see that 
\[X_{\ga}^\la\cong \Sp_{a}^0\quad\text{and}\quad X_{\ga}^{\le\la}\cong \Sp_{a}.\]
Conversely, let $a\in\fC_+(\cO)\cap\fC_+(F)^{\mathrm{rs}}$ and suppose that 
\[\beta(a)\in \varpi^{-w_0(\la)}\tad(\cO)\subset A_{\gsc}(\cO)\cap\tad(F)\]
for some $\la\in X_*(\tad)_+$.  Let $\ga_a^w\in\gad(F)$ be the image of $\epsilon_+^w(a)\in G_+(F)\cap\vin_{\gsc}^0(\cO)$ under the natural quotient $G_+(F)\to\gad(F)$. Then we have
\[\Sp_{a}\cong X_{\ga_a^w}^{\le\la}\quad\text{and}\quad\Sp_{a}^0\cong X_{\ga_a^w}^\la.\]
Note that the isomorphism class of $X_{\ga_a^w}^{\le\la}$ and $X_{\ga_a^w}^\la$ does not depend on the choice of $w\in\mathrm{Cox}(W,S)$.

\subsection{Symmetries}
Assume $X_{\ga}^\la$ is nonempty. Then by Proposition~\ref{prop:nonempty} we have $\kappa_G(\ga)=p_G(\la)$ and
\[a=\chi_+(\ga_\la)\in \fC_{\gsc_+}^{\mathrm{rs}}(F)\cap\fC_+(\cO).\] 
Let $J_a$ be the commutative group scheme over $\spec\cO$ obtained by pulling back $\cJ$ along $a:\spec\cO\to\fC_+$. Since $a$ is generically regular semisimple, there is a canonical isomorphism 
$LJ_a\cong LG_\ga^0$ which allows us to identify the positive loop group $L^+J_a$ as a subgroup of $LG_\ga^0$. Consider the quotient group
\[P_a:=LJ_a/L^+J_a\cong LG_\ga^0/L^+J_a.\]

In other words, $\cP_a$ is the affine Grassmanian of $J_a$ classifying isomorphism classes of $J_a$-torsors on $\spec\cO$ with a trivialization of its restriction to $\spec F$. 

The loop group $LG_\ga^0$ acts naturally on $X_{\ga}^\la$ and this action factors through $P_a$. Using the isomorphism $X_\ga^\la\cong\Sp_a^0$, the $P_a$ action is induced by the $\bB \cJ$ action on $[V_{\gsc}^0/\mathrm{Ad}(G)]$ in Proposition~\ref{BJ-action-vin-G-prop}. Moreover, $P_a$ preserve the open subspaces $\Sp_a^{\mathrm{reg}}$ and $\Sp_a^w$ for each $w\in\mathrm{Cox}(W,S)$.

\begin{prop}\label{X-ga-w-torsor-prop}
For each $w\in\mathrm{Cox}(W,S)$, $\Sp_a^w$ is a torsor under $P_a$.
\end{prop}
\begin{proof}
This is a consequence of \ref{BJ-action-vin-G-prop}.
\end{proof}

\begin{rem}
Unlike the Lie algebra case, $\Sp_a^{\mathrm{reg}}$ may not be a $P_a$-torsor in general. See the discussion in \S~\ref{regular-components-section}.
\end{rem}

Let $R_a$ be the finite free $\cO$-algebra defined by the Cartesian diagram
\begin{equation}\label{cameral-cover-diagram-eq}
\xymatrix{
\widetilde{X}_a:=\spec R_a\ar[r]\ar[d] & \overline{T_+}\ar[d]\\
\spec\cO\ar[r]^a & \fC_+
}
\end{equation}
Let $R_a^\flat$ be the normalization of $R_a$ and $\widetilde{X}_a^\flat:=\spec R_a^\flat$. Then $W$ acts naturally on the $\cO$-algebras $R_a$ and $R_a^\flat$.\par 
Let $J_a^\flat$ be the finite type Neron model of $J_a$. Hence $J_a^\flat$ is a smooth commutative group scheme over $\cO$ such that $J_a^\flat(F)=J_a(F)=G_\ga^0(F)$ and $J_a^\flat(\cO)$ is the maximal bounded subgroup of $G_\ga^0(F)$. 

\begin{lem}\label{J_a-flat-Galois-description-lem}
There is a canocical isomorphism
\[J_a^\flat\cong(\prod_{R_a^\flat/\cO}T\times \widetilde{X}_a^\flat)^W\]
\end{lem} 
\begin{proof}
The proof is the same as \cite[Proposition 3.8.2]{Ngo10}.
\end{proof}

\begin{cor}\label{Lie-P_a-cor}
$\mathrm{Lie}(P_a)=(\ft\otimes_k (R_a^\flat/R_a))^W$
\end{cor}
\begin{proof}
The quotient $L^+J_a^\flat/L^+J_a$ is an open subgroup of $P_a$. Hence we have isomorphism of $\cO$ modules
\[\mathrm{Lie} P_a\cong\mathrm{Lie}(L^+J_a^\flat)/\mathrm{Lie}(L^+J_a).\]
On the other hand, by \ref{Galois-description-J-prop}, we have
\[\mathrm{Lie}L^+J_a=(\ft\otimes_k R_a)^W\]
and by \ref{J_a-flat-Galois-description-lem},
\[\mathrm{Lie}L^+J_a^\flat=(\ft\otimes_k R_a^\flat)^W.\]
Hence the Corollary follows.
\end{proof}

\subsection{Admissible subsets of loop spaces}
In this section we closely follow \cite[\S5]{GHKR}.\par 
Let  $M$ be a standard Levi subgroup of $G$ and $P=MN$ the  standard parabolic subgroup where $N$ is the unipotent radical of $P$. Let $Z(M)^0$ be the neutral component of the center of $M$. Then $Z(M)^0$ is a subtorus of $T$. Let $\Phi_N$ be the set of roots of $Z(M)^0$ acting on $N$ and $\Phi_N^\vee$ the corresponding set of coroots. For each $\alpha\in\Phi_N$, let $N_\alpha$ be the corresponding root subgroup. Then each $N_\alpha$ is isomorphic to a product of several copies of $\bG_a$ and is preserved by the adjoint action of $M$. Denote $\delta_N$ half sum of elements in $\Delta_N^\vee$.\par 
For each $\alpha\in\Delta_N$, denote $\mathrm{ht}_N(\alpha):=\langle\delta_N,\alpha\rangle$. Let $l=\max_{\alpha\in\Phi_N}\mathrm{ht}_N(\alpha)$. For each $1\le i\le l$, let $N[i]$ be the subgroup of $N$ generated by root groups $N_\alpha$ with $\mathrm{ht}_N(\alpha)\ge i$. Also we denote $N[l+1]=1$. Then $N[1]=N$ and for each $1\le i\le s+1$, $N[i]$ is a normal subgroup of $N$ and the successive quotients $N\langle i\rangle:=N[i]/N[i+1]$ are commutative groups isomorphic to products of some copies of $\bG_a$. Let $LN$ and $L^+N$ be the loop space and arc space of $N$. For each integer $n\ge0$, let $N_n:=\ker(L^+N\to L^+_nN)$. Then $\{N_n\}_{n\ge0}$ form a decreasing sequence of compact open subgroups of $LN$.\par 
For each $\ga\in M(F)\cap G(F)^{\mathrm{rs}}$, consider the map
\begin{equation}\label{f-gamma-eq}
\xymatrix@R=1pt{
f_\gamma: LN\ar[r] & LN\\
u\ar@{|->}[r] & u^{-1}\gamma u\gamma^{-1}
}
\end{equation}
Then $f_\ga$ preserves the root subgroups $N_\alpha$ and hence each normal subgroup $N[i]$. In particular, $f_\ga$ induces morphism $f_\ga[i]:LN[i]\to LN[i]$ and $f_\ga\langle i\rangle:LN\langle i\rangle\to LN\langle i\rangle$. \par 
For each $1\le i\le l$, denote $r_i:=\mathrm{val}\det(f_\ga\langle i\rangle)$. Note that there is a $M$-equivariant isomorphism $N\langle i\rangle\cong\mathrm{Lie}N\langle i\rangle$ from which we see that 
\[r_i=\mathrm{val}\det(\mathrm{ad}_\ga:\mathrm{Lie}N\langle i\rangle(F)\to\mathrm{Lie}N\langle i\rangle(F)).\]
Consider the following invariant of $\ga$:
\begin{equation}\label{eq:r-N}
r_N(\ga):=\mathrm{val}\det(\mathrm{ad}_\ga:\mathrm{Lie}N(F)\to\mathrm{Lie}N(F))
\end{equation}
Then we also have $r_N(\gamma)=\sum_{i=1}^l r_i$.\par 
Now assume that $\ga\in M(F)_+$, we have $f_\gamma(U_n)\subset U_n$ for all $n\ge0$.\par 
Let $f_0: L^+N\to L^+N$  be the restriction of $f_\gamma$ to the arc space $L^+N$.
\begin{lem}\label{lem:image-L^+N}
For any $1\le i\le l+1$ and any positive integer $n$ such that $n\ge \sum\limits_{j=i}^{l+1}r_j$ we have $N[i]_n\subset f_\gamma(L^+N[i])$.
\end{lem}
\begin{proof}
We prove by desending induction on $i$. The case $i=l+1$ is trivial since $N[l+1]=1$. Assume the statement is true for $i+1$.
Let $x\in N[i]_n$. To show that  $x\in f_\gamma(L^+N[i])$ it suffices to find $u\in N[i](\cO)$ with $x*u=1$, for then $f_\gamma(u^{-1})=x$.\par 
Let $x_i\in N\langle i\rangle_n$ be the image of $x$. Since $\mathrm{val}\det(f_\ga\langle i\rangle)=r_i$, we have $\varpi^{r_i}N\langle i\rangle(\cO)\subset f_\ga\langle i\rangle(N\langle i\rangle(\cO))$. Hence there exists $u_i\in N[i]_{n-r_i}$ such that $x_i*u_i=1$ in $N\langle i\rangle(\cO)$ and hence $x*u_i\in N[i+1]_{n-r_i}$. By induction hypothesis, there exists $v\in N[i+1](\cO)$ such that $(x*u_i)*v=1$. Then $u=u_iv$ satisfies $x*u=1$.
\end{proof}
A subset of $L^+N$ is \emph{admissible} if it is the pre-image of a locally closed subset of $L^+_nN$ for some $n$. A subset $Z$ of $LN$ is \emph{admissible} of it is conjugate under $G(F)$ to an admissible subset of $L^+N$.
\begin{lem}\label{f_0-smooth-lem}
Let $V$ be an admissible subset of $L^+N$. Let $n\ge r_N(\ga)$ be a positive integer such that $V$ is right invariant under $N_n$.
Suppose moreover that $V\subset f_0(L^+N)$. Then the set $f_0^{-1}(V)$ is admissible and right invariant under $N_n$. Moreover, $f_0$ induces a smooth surjective map 
\[f_0^{-1}(V)/N_n\to V/N_n\] whose fibers are isomorphic to $\bA^{r_N(\gamma)}$.
\end{lem}
\begin{proof}
Let $\bar f_0: L_n^+N\to L_n^+N$ be the map induced by $f_0$. Since $V$ is right invariant under $N_n$, a straightforward calculation shows that $f_0^{-1}(V)$ is also right invariant under $N_n$. Denote $\overline{V}:=V/U_n$. Then we have $f_0^{-1}(V)/U_n=\bar f_0^{-1}(\overline{V})$, a locally closed subset of $L_n^+N$. In particular, $f_0^{-1}(V)$ is admissible. Since $V\subset f_0(L^+N)$, the induced map $\bar f_0^{-1}(\overline{V})\to\overline{V}$ is surjective and it remains to show that it is smooth with fibers isomorphic to $\bA^{r(\ga)}$.\par 
Denote $H:=L_n^+N$, $H[i]:=L_n^+(N[i])$ and $H\langle i\rangle:=L_n^+(N\langle i\rangle)$. Then for each $1\le i\le l+1$, $H[i]$ is a normal subgroup of $H$ and $H[i]/H[i+1]\cong H\langle i\rangle$. For each $1\le j\le n$, we define a normal subgroup $H_j:=\ker(H\to L^+_jN)$ of $H$; and similarly we define normal subgroups $H[i]_j$ (resp. $H\langle i\rangle_j=\varpi^jH\langle i\rangle$) of $H[i]$ (resp. $H\langle i\rangle$).\par 
Consider the right action of $H$ on itself defined by $v*u:=u^{-1}v\ga u\ga^{-1}$ for $u,v\in H(k)=N(\cO/\varpi^n\cO)$. Then $\bar f_0(u)=1*u$ and hence $\bar f_0$ is the orbit map at $1$ of the $H$-action. In particular, all fibres of $\bar f_0$ are isomorphic to the stabilizer $S:=\bar f_0^{-1}(1)$.\par 
Now we take a closer look at the structure of the stabilizer $S$.  First note that the action $*$ induces actions of $H[i]$ and $H\langle i\rangle$ on themselves. Let $S[i]$ (resp. $S\langle i\rangle$) be the stabilizer of $1$ under the $H[i]$ (resp. $H\langle i\rangle$) action.\par 
We claim that for all $i$, the canonical homomorphism $S[i]\to S\langle i\rangle$ is surjective. Let $s\in S\langle i\rangle$ and choose a representative $h\in H[i]$ of $s$. Since 
\[S\langle i\rangle=\ker(\bar{f}_0\langle i\rangle)\subset\varpi^{n-r_i}H\langle i\rangle\]
we have $h\in H[i]_{n-r_i}$ and $1*h\in H[i]_{n-r_i}\cap H[i+1]=H[i+1]_{n-r_i}$. By assumption $n-r_i\ge\sum_{j=i+1}^{l+1} r_j$, then we can apply Lemma~\ref{lem:image-L^+N} to obtain an element $h'\in H[i+1]$ such that $1*(hh')=1$. Thus $hh'\in S[i]$ maps to $s\in S\langle i\rangle$ and the claim follows.\par
The kernal of the surjective homomorphism $S[i]\to S\langle i\rangle$ is $S[i]\cap H[i+1]=S[i+1]$. Moreover, we have
\[S\langle i\rangle\cong (f_0\langle i\rangle)^{-1}(\varpi^n N\langle i\rangle)/\varpi^n N\langle i\rangle\cong \bA^{r_i}\]
From this we see that $S\cong \bA^{r_N(\ga)}$ as a scheme.
\end{proof}

The proof of the following lemma is inspired by \cite[Lemma 3.8]{KoV}.
\begin{lem}\label{inverse-image-U_n-lem}
For any $n\ge r_N(\gamma)$, we have $f_\gamma^{-1}(N_n)\subset N_{n-r_N(\gamma)}$.
\end{lem}
\begin{proof}
Let $u\in N(F)$ with $f_\gamma(u)\in N_n$. We will show by induction that 
\[u\in N[i](F)\cdot N_{n-\sum_{j<i}r_j}.\]
The case $i=1$ says $u\in N[1](F)=N(F)$ which is clear and the case $i=s+1$ gives the lemma since $\sum_{i=1}^sr_i=r_N(\ga)$ and $N[s+1]=1$.\par 
It remains to finish the induction step. By induction hypothesis we have $u=u_iv$ with $u_i\in N[i](F)$ and $v\in N_{n-\sum_{j<i}r_j}$. 
By assumption,
\[f_\gamma(u)=f_\gamma(u_iv)=v^{-1}\cdot u_i^{-1}\gamma u_i\gamma^{-1}\cdot \gamma v\gamma^{-1}\in N_n\] 
from which it follows that 
\[ u_i^{-1}\gamma u_i\gamma^{-1}\in N[i](F)\cap v\cdot N_n \cdot(\gamma v^{-1}\gamma^{-1})\subset N[i]_{n-\sum_{j<i}r_j}\]
Let $\bar{u}_i\in N\langle i\rangle$ be the image of $u_i$. Then we have 
\[f_\ga\langle i\rangle (\bar{u}_i)\in N\langle i\rangle_{n-\sum_{j<i}r_j}\]
Since $\mathrm{val}\det(f_\ga\langle i\rangle)=r_i$, we get that
$\bar{u}_i\in N\langle i\rangle_{n-\sum_{j<i+1}r_j}$ and hence
\[u=u_iv\in N[i+1](F)\cdot N_{n-\sum_{j<i+1}r_j}\]
This finishes the induction step.
\end{proof}

\begin{prop}\label{prop:adm-set}
Let $Z$ be an admissible subset of the loop space $LN$. Then 
$f^{-1}_\gamma(Z)$ is admissible and there exists a positive integer $m$ such that for all $n\ge m$,  $f^{-1}_\gamma(Z)$ and $Z$ are right invariant under the group $N_n$
and the map 
\[f^{-1}_\gamma(Z)/N_n\to Z/N_n\] 
induced by $f_\gamma$ is smooth surjective whose geometric fibers are irreducible of dimension $r_N(\gamma)$.
\end{prop}
\begin{proof}
Let $n_0\ge r(\gamma)$ be a positive integer. Choose  a coweight $\mu_0\in X_*(Z(M)^0)$ such that 
\[Z^{\mu_0}:=\mathrm{Ad}(\varpi^{\mu_0})(Z)\subset N_{n_0}.\]
Then by Lemma~\ref{inverse-image-U_n-lem} we have 
\[f_\gamma^{-1}(Z^{\mu_0})\subset f_\gamma^{-1}(N_{n_0})\subset N_{n_0-r(\gamma)}\subset L^+N\]
Hence in particular 
\[\mathrm{Ad}(\varpi^{\mu_0})(f_\ga^{-1}(Z))=f_\gamma^{-1}(Z^{\mu_0})=f_0^{-1}(Z^{\mu_0})\]
Moreover, since $Z^{\mu_0}$ is an admissible subset of $L^+N$, $f_0^{-1}(Z^{\mu_0})$ is an admissible subset of $L^+N$ by Lemma~\ref{f_0-smooth-lem}. This shows that $f_\ga^{-1}(Z)$ is admissible.\par 
Let $n_1>n_0$ be a positive integer such that  $Z^{\mu_0}$ and $f_\gamma^{-1}(Z^{\mu_0})$ are invariant under right multiplication by $N_{n_1}$. For all $n\ge n_1$, since the map $f_\gamma$ commutes with conjugation by $\varpi^{\mu_0}$, $Z$ and $f_\gamma^{-1}(Z)$ are right invariant under the group $N_n^{-\mu_0}:=\varpi^{-\mu_0}N_n\varpi^{\mu_0}$. Then we get the following commutative diagram
\[\xymatrix{
f_\gamma^{-1}(Z)/N_n^{-\mu_0}\ar[r]\ar[d]^{\simeq} & Z/N_n^{-\mu_0}\ar[d]^{\simeq}\\
f_\gamma^{-1}(Z^{\mu_0})/N_n\ar[r] & Z^{\mu_0}/N_n
}\]
where the horizontal arrows are induced by $f_\gamma$ and the vertical arrows are isomorphisms induced by $\mathrm{Ad}(\varpi^{\mu_0})$.\par 
By Lemma~\ref{lem:image-L^+N}, $Z^{\mu_0}\subset N_{n_0}\subset f_\gamma(L^+N)$. Therefore we can apply Lemma~\ref{f_0-smooth-lem} to conclude that the lower horizontal map is surjective smooth whose fibers are isomorphic to $\bA^{r_N(\gamma)}$. Hence the same is true for the upper horizontal map.\par 
Let $m$ be a positie integer such that for all $n\ge m$, $N_n\supset N_{n'}^{-\mu_0}$ for some $n'\ge n_1$. Consider the following diagram
\[\xymatrix{
f_\ga^{-1}(Z)/N_{n'}^{-\mu_0}\ar[r]\ar[d] & Z/N_{n'}^{-\mu_0}\ar[d]\\
f_\ga^{-1}(Z)/N_n\ar[r] & Z/N_n
}\]
The two vertical maps are smooth surjective with fibers isomorphic to the irreducible scheme $U_n/U_{n'}^{-\mu_0}$ and the upper horizontal map is smooth surjective with fibers isomorphic to $\bA^{r_N(\ga)}$ as we have just seen. Hence the lower horizontal map is smooth surjective with irreducible fibers of dimension $r_N(\ga)$.
\end{proof}

\subsection{The case of unramified conjugacy class}\label{unr-section}
In this section we assume that $\gamma\in G(F)^{\mathrm{rs}}$ is an \emph{unramified} regular semisimple element. Since the residue field $k$ is algebraically closed,  after conjugation  we may assume that $\gamma\in\varpi^\mu T(\cO)\cap G^{\mathrm{rs}}(F)$, where $\mu=\nu_\ga\in X_*(T)_+$ is the Newton points of $\ga$. In this case, we have $G_\ga^0=T$. By Lemma~\ref{d-ga-newton-lem} the discriminant valuation for $\ga$ is
\[d(\ga)=2\sum_{\alpha\in\Phi^+}\mathrm{val}(\alpha(\gamma)-1)-\langle2\rho,\mu\rangle.\]
We will apply the results in previous section to the case $N=U$ is a maximal unipotent subgroup. In this case, the corresponding invariant for $\ga$ is 
\begin{equation}\label{r-gamma-eq}
r(\gamma):=r_U(\ga)=\sum_{\alpha\in\Phi^+}\mathrm{val}(\alpha(\gamma)-1)=\frac{1}{2}d(\gamma)+\langle\rho,\mu\rangle.
\end{equation} 
Fix a dominant coweight $\la\in\Lambda_+$ such that $\mu\le \la$. By Proposition~\ref{prop:nonempty}, this implies that $X_\ga^\la$ is nonempty.

\subsubsection{Relation with MV-cycles}
Let $Y_\ga^\la$ be the locally closed sub-indscheme of $X_\ga^\la$ whose set of $k$-points is
\[Y_\ga^\la(k)=\{u\in U(F)/U(\cO) | \mathrm{Ad}(u)^{-1}\ga\in G(\cO)\varpi^\la G(\cO)\}\]
To understand the structure of $Y_\gamma^\la$, we use the map $f_\gamma: LU\to LU$ (cf. \eqref{f-gamma-eq}).  
In the following, we denote $K:=L^+G$. Then we have
\[Y_\gamma^\lambda=(f_\gamma^{-1}(K \varpi^\lambda K\varpi^{-\mu}\cap LU)/L^+U\]
Recall the Mirkovic-Vilonen cycles in the affine Grassmanian:
\[S_\mu\cap\mathrm{Gr}_\la=(LU\varpi^\mu K\cap K\varpi^\lambda K)/K\]
From this description we get an isomorphism
\begin{equation}\label{MV-cycle-presentation-equation}
\xymatrix@R=1pt{
(LU\cap K\varpi^\lambda K\varpi^{-\mu})/\varpi^\mu L^+U\varpi^{-\mu}\ar[r] &
S_\mu\cap\mathrm{Gr}_\lambda\\
u\ar@{|->}[r] & u\varpi^\mu
}\end{equation}
In summary, we have the following diagram
\[\xymatrix{
f_\gamma^{-1}(K\varpi^\lambda K\varpi^{-\mu}\cap LU)\ar[r]^{f_\gamma}\ar[d] & K\varpi^\lambda K\varpi^{-\mu}\cap LU\ar[d]\\
Y_\gamma^\lambda & S_\mu\cap\mathrm{Gr}_\lambda
}\]
where the left vertical arrow is an $L^+U$-torsor and the right vertical arrow is a torsor under the group $\varpi^\mu L^+U\varpi^{-\mu}$.\par 
\begin{thm}\label{Y-gamma-thm}
 $Y_\ga^\la$ is an equi-dimensional quasi-projective variety of dimension $\langle\rho,\la\rangle+\frac{1}{2}d(\ga)$, where $d(\ga)$ is the discriminant valuation, cf. Definition~\ref{def:disc-valuation}. Moreover, the number of irreducible components of $Y_\ga^\la$ equals to $m_{\la\mu}$, the dimension of $\mu$-weight space in the irreducible representation $V_\la$ of $\hat G$ with highest weight $\la$. 
\end{thm}
\begin{proof}
Apply Proposition~\ref{prop:adm-set} to the admissible subset $Z=K\varpi^\lambda K\varpi^{-\mu}\cap LU$ of $LU$, we see that there exists a large enough positive integer $n$ such that in the following diagram
\[\xymatrix{
f_\gamma^{-1}(K\varpi^\lambda K\varpi^{-\mu}\cap LU)/U_n\ar[r]^{\bar f_\gamma}\ar[d] & (K\varpi^\lambda K\varpi^{-\mu}\cap LU)/U_n\ar[d]\\
Y_\gamma^\lambda & S_\mu\cap\mathrm{Gr}_\lambda
}\]
\begin{enumerate}
\item All schemes are of finite type;
\item The map $\bar f_\gamma$ induced by $f_\gamma$ is smooth surjective whose geometric fibers are irreducible of dimension $r(\gamma)$, where we recall that $r(\ga)$ is defined in \eqref{r-gamma-eq};
\item $U_n$ is contained in $\varpi^\mu L^+U\varpi^{-\mu}$, hence also $L^+U$;
\item The left vertical map is smooth surjective with fibers isomorphic to the irreducible scheme $L^+U/U_n$;
\item The right vertical map is smooth with fibers isomorphic to the irreducible scheme $\varpi^\mu L^+U\varpi^{-\mu}/U_n$.
\end{enumerate}
Since $Y_\ga^\la$ is of finite type, it is a locally closed subscheme of a closed Schubert variety. In particular, $Y_\ga^\la$ is quasi-projective since closed Schubert varieties are projective.\par 
Recall that the MV-cycle $S_\mu\cap\mathrm{Gr}_\la$ is equidimensional of dimension $\langle\rho,\la+\mu\rangle$. Hence by (2)-(5) we see that $Y_\ga^\la$ is equidimensional of dimension  
\begin{equation}
\begin{split}
\dim Y_\ga^\la&=\dim(S_\mu\cap\mathrm{Gr}_\la)+\dim\varpi^\mu U(\cO)\varpi^{-\mu}/U_n^{-\lambda_0}+r(\gamma)-\dim U(\cO)/U_n^{-\lambda_0}\\
&=\langle\rho,\la+\mu\rangle-\langle2\rho,\mu\rangle+r(\gamma)=\langle\rho,\la\rangle+\frac{1}{2}d(\gamma)
\end{split}
\end{equation}

Moreover, by \cite[\href{http://stacks.math.columbia.edu/tag/037A}{Tag 037A}]{stacks-project} the 3 maps in the diagram above induces a canonical bijections between set of irreducible components
\[\mathrm{Irr}(Y_\ga^\la)\xrightarrow{\sim}\mathrm{Irr}(S_\mu\cap\mathrm{Gr}_\la).\]
Hence the number of irreducible components of $Y_\ga^\la$ equals to the number of irreducible components of the MV-cycle $S_\mu\cap\mathrm{Gr}_\la$, which is known to be $m_{\la\mu}$.
\end{proof}

\begin{cor}\label{dim-unr-cor}
Suppose $\ga\in G(F)^{\mathrm{rs}}$ is unramified (i.e. split) and $\nu_\ga=\mu\in X_*(T)_+$, then $X_\ga^\la$ is a scheme locally of finite type, equidimensional of dimension 
\[\dim X_\ga^\la=\langle\rho,\la\rangle+\frac{1}{2}d(\ga).\]
Moreover, the number of $G_\ga^0(F)$-orbits on its set of irreducible component $\mathrm{Irr}(X_\ga^\la)$ equals to $m_{\la\mu}$.
\end{cor}
\begin{proof}
There is a natural morphism
\[\xymatrix@R=1pt{
Y_\gamma^\lambda\times X_*(T)\ar[r] & X_\gamma^\la \\
(u,\nu)\ar@{|->}[r] & u\varpi^\nu
}\]
which induces bijection on $k$-points and a stratification of $X_\ga^\la$ such that each strata is isomorphic to $Y_\ga^\la$. Thus $X_\ga^\la$ is a scheme locally of finite type and the assertions about equidimensionality and dimension formula follows from the corresponding statements for $Y_\ga^\la$.\par 
The $LG_\ga^0$ action on the set $\mathrm{Irr}(X_\ga^\la)$ factors through $\pi_0(LG_\ga^0)=X_*(T)$ and hence $LG_\ga$-orbits on $\mathrm{Irr}(X_\ga^\la)$ corresponds bijectively to the set $\mathrm{Irr}(Y_\ga^\la)$. Thus the number of orbits equals to the weight multiplicity $m_{\la\mu}$. 
\end{proof}

\subsection{Finiteness of Kottwitz-Viehmann varieties}\label{sec:fin-spr}
In this section we let $\ga\in G(F)^{\mathrm{rs}}$ be any regular semisimple element and $\la\in\Lambda^+$. Assume without loss of generality that $X_\ga^\la$ is nonempty and $\det(\ga)=\det(\varpi^\la)$. Then we get an element $\ga_\la\in \vin_{G}^\la(F)$ as in Lemma~\ref{lem:ga-la-lem}. Moreover, the Newton point of $\ga$ satisfies $\nu_\ga\le_\bQ\la$ and $\chi(\ga)\in\fC_{\le\la}$ by Proposition~\ref{prop:nonempty}. 

We show in this section that $X_\ga^\la$, a'priori an ind-scheme, is actually a scheme locally of finite type. This has already been proved for unramified conjugacy classes in Corollary~\ref{dim-unr-cor}. It remains to reduce the general case to the unramified case. This reduction step is completely analogous to the Lie algebra case. For the reader's convenience, we include the details, following the exposition in \cite[\S 2.5]{Yun15}. See also \citep{Bou15}.\par 

Let $F'/F$ be a finite extension of degree $e$ so that $\gamma$ splits over $F'$.  Let $\varpi'=\varpi^{1/e}\in F'$ be a uniformizer and $\cO'=k[[\varpi']]$ the ring of integers in $F'$. Let $\sigma$ be a generator of the cyclic group $\mathrm{Gal}(F'/F)$\par 
Choose $h\in G(F')$ such that $\mathrm
{Ad}(h) G_\ga^0 =T$. Then $h\sigma(h)^{-1}\in N_G(T)(F')$ and we let $w\in W$ be its image.\par 
Consider the embedding
\[\xymatrix@R=1pt{
\iota_\ga: \La:=X_*(T)\ar[r] & G_\ga(F')\\
\mu\ar@{|->}[r] & \mathrm{Ad}(h)^{-1}\varpi^\mu
}\]
Let $\La_\ga:=\iota_\ga^{-1}(G_\ga(F))$. It follows immediately that $\La_\ga\subset\La^w$ where $\La^w$ is the fixed point set of $w$ on $\La$. Moreover, $\La_\ga$ can be identified with the coweight lattice of the maximal $F$-split subtorus of $G_\ga$. In particular, $(\La_\ga)_\bQ=(\La^w)_\bQ$ so that $\La_\ga\subset\La^w$ is a subgroup of finite index. 
\begin{prop}
There exists a closed subscheme $Z\subset X_\ga^\la$ which is projective over $k$ such that $X_\ga^\la=\cup_{\ell\in\La_\ga}\ell\cdot Z$. Here $\ell\in\Lambda_\ga$ acts on $X_\ga^\la$ via the embedding $\iota_\ga$.
\end{prop}
\begin{proof}
We reprase the argument in \cite[\S 2.5.7]{Yun15}. Let $\widetilde{X}_\ga^{e\la}$ be the generalized affine Springer fiber of coweight $e\la$ for $\ga$ in $\mathrm{Gr}_{G_{F'}}$, the affine Grassmanian of $G_{F'}$. Then $\sigma$ acts naturally on $\widetilde{X}_\ga^{e\la}$ and the fixed points sub-indscheme $(\widetilde{X}_\ga^{e\la})^\sigma$ contains $X_\ga^\la$ (but they are not equal in general). 
Let $\ga'=h\ga h^{-1}\in T(F')$ and $\widetilde{X}_{\ga'}^{e\la}$ the corresponding generalized affine Springer fiber in $\mathrm{Gr}_{G_{F'}}$. Then 
\[\widetilde{X}_{\ga'}^{e\la}=h\cdot \widetilde{X}_\ga^{e\la}\]
By Theorem~\ref{Y-gamma-thm},
there is a locally closed subscheme $\widetilde{Y}_{\ga'}^{e\la}$ of $\widetilde{X}_{\ga'}^{e\la}$ such that 
\[\widetilde{X}_{\ga'}^{e\la}=\cup_{\ell\in\La}\ell\cdot Y_{\ga'}^{e\la}.\] 
Let $\widetilde{Z}$ be the closure of $h^{-1}\widetilde{Y}_{\ga'}^{e\la}$ in $\widetilde{X}_\ga^{e\la}$. Then $\widetilde{Z}$ is projective over $k$ and $\widetilde{X}_\ga^{e\la}=\cup_{\ell\in\La}\ell\cdot\widetilde{Z}$.\par 
 Recall that $w\in W$ is represented by $h\sigma(h)^{-1}$. One can check that $\sigma(\widetilde Z)=\widetilde Z$ and more generally $\sigma(\ell\cdot\widetilde{Z})=w(\ell)\cdot\widetilde{Z}$ for all $\ell\in\La$.Consequently,
\[(\widetilde{X}_\ga^{e\la})^\sigma=\cup_{\ell\in\La^w}\ell\cdot\widetilde{Z}=\cup_{\ell\in\La_\ga}\ell\cdot (C\cdot\widetilde{Z})\]
where $C\subset\La^w$ is a finite set of representatives of the quotient $\La^w/\La_\ga$. Hence, $C\cdot\widetilde{Z}$ is a finite type scheme.\par 
Finally let $Z:=(C\cdot\widetilde{Z})\cap X_\ga^\la$. Then $Z$ is a finite type subscheme of $X_\ga^\la$. Hence $Z$ is projective over $k$ and $X_\ga^\la=\cup_{\ell\in\La_\ga}\ell\cdot Z$.
\end{proof}

As a consequence, we immediately get:
\begin{thm}\label{finite-type-thm}
The ind-scheme $X_\ga^\la$ is a finite dimensional $k$-scheme, locally of finite type. Moreover, the lattice $\La_\ga$ acts freely on $X_\ga^\la$ and the quotient $X_\ga^\la/\La_\ga$ is representable by a proper algebraic space over $k$. 
\end{thm}

\subsection{Dimension of the regular locus}
Recall that the regular locus $X_\ga^{\la,\mathrm{reg}}$ is an open subscheme of $X_\ga^\la$ on which the action of  $P_a=LG_\ga^0/L^+J_a$ is free (but not necessarily transitive). 
\begin{thm}\label{dim-reg-locus-thm}
\[\dim P_a=\dim X_\ga^{\la,\mathrm{reg}}=\langle\rho,\la\rangle+\frac{d(\ga)-c(\ga)}{2}\]
where 
\begin{itemize}
\item $d(\ga):=\mathrm{val}(\det(\mathrm{Id}-\mathrm{ad}(\ga):\mathfrak{g}(F)/\mathrm{g}_\ga(F)\to \mathfrak{g}(F)/\mathrm{g}_\ga(F)))$.
\item $c(\ga):=\mathrm{rank}(G)-\mathrm{rank}_FG_\ga$, where $\mathrm{rank}_FG_\ga$ is the dimension of the maximal $F$-split subtorus of $G_\ga$.
\end{itemize}
Moreover, $X_\ga^{\la,\mathrm{reg}}$ is equidimensional.
\end{thm}
\begin{proof}
The first equality follows from the fact that the $P_a$-orbits in $X_\ga^{\la,\mathrm{reg}}$ are open and the action is free.\par 
When $\ga$ is unramified (hence split as $k$ is algebraically closed), the second equality follows from Corollary~\ref{dim-unr-cor}. It remains to reduce to this case. The argument is similar to that of Bezrukavnikov's in Lie algebra case, cf. \cite{Be96}, which we reformulate using the Galois description of universal centralizer.\par 
Let $A$ be the finite free $\cO$-algebra defined by the Cartesian diagram \eqref{cameral-cover-diagram-eq} and $A^\flat$ the normalization of $A$. Then $W$ acts naturally on the $\cO$-algebras $A$ and $A^\flat$ and by \ref{Lie-P_a-cor}, we get
\[\dim\cP_a=\dim_k(\ft\otimes_k (A^\flat/A))^W.\]
Let $\tilde F/F$ be a ramified extension of degree $e$, with ring of integers $\tilde\cO=k[[\varpi^{\frac{1}{e}}]]$, such that $\ga$ is split over $\tilde F$. Let $\sigma$ be a generater of the cyclic group $\Gamma:=\mathrm{Gal}(F'/F)$. Let $\tilde A:=A\otimes_\cO\tilde\cO$ and $\tilde A^\flat$ its normalization. We remark that $\tilde A^\flat$ is not the same as $A^\flat\otimes_\cO\tilde\cO$ in general.  Let $\tilde\cP_a=LG_{\ga,F'}/L^+J_{a,F'}$. Then by the dimension formula in split case, we have
\[\dim_k(\ft\otimes_k\tilde A^\flat/\tilde A)^W=\dim\tilde\cP_a=\langle\rho,e\la\rangle+\frac{1}{2}e\cdot d(\ga).\]
As $\ga$ split over $\tilde\cO$, we have 
\[\tilde A^\flat\cong\tilde\cO[W]:=\tilde\cO\otimes_k k[W]\] as $W$-module. Here $W$ acts on $\tilde\cO[W]$ via right regular representation. Moreover, there exists an element $w_\ga\in W$ of order $e$ such that under the above isomorphism, the natural action of $\sigma\in\Gamma$ on $\tilde A^\flat$ becomes $\sigma\otimes l_{w_\ga}$ where $l_{w_\ga}$ denotes the left regular action of $w_\ga$ on $k[W]$. In particular, the action of $W$ and $\Gamma$ commutes with each other. With these considerations, we obtain an isomorphism
\[(\ft\otimes_k\tilde A^\flat)^W\cong\ft\otimes_k\tilde\cO\]
which intertwines the action of $\sigma\in\Gamma$ on the left hand side with the action of $w\otimes\sigma$ on the right hand side. \par 
Moreover, we have an equality
\[(\ft\otimes_k \tilde{A}^\flat)^\Gamma=\ft\otimes_k A^\flat\]
which remains true after taking $W$-invariants since the $\Gamma$ action commutes with $W$ action. In particular, we have
\[M:=(\ft\otimes_k\tilde\cO)^\Gamma=(\ft\otimes_k A^\flat)^W\]
Moreover, it is clear that from the definition of $W$ action that
\[(\ft\otimes_k\tilde A)^W=(\ft\otimes_k A)^W \otimes_\cO \tilde\cO.\]
Thus we get
\begin{equation}\label{P-a-ramified-case-calculation-eq}
\begin{split}
\dim\cP_a & =\dim_k(\ft\otimes_k A^\flat/A)^W=
\frac{1}{e}\dim_k(\ft\otimes_k (A^\flat/A)\otimes_\cO \tilde\cO)^W=\frac{1}{e}\dim_k \left(\frac{M\otimes_\cO \tilde \cO}{(\ft\otimes_k\tilde A)^W}\right)\\
&=\langle\rho,\la\rangle+\frac{1}{2}d(\ga)-\frac{1}{e}\dim_k\left(\frac{\ft\otimes_k\tilde\cO}{M\otimes_\cO\tilde\cO}\right)
\end{split}
\end{equation}
Since the element $w_\ga\in W$ has order $e$, its eigenvalues are $e$-th roots of unit. Let $\zeta$ be a primitive $e$-th root of unit and $\ft(i)$ the subspace of $\ft$ on which $w_\ga$ acts via the scalar $\zeta^i$. In particular, $\ft(0)=\ft^{w_\ga}$ is the $w_\ga$ invariant subspace. Then we have
\[M:=(\ft\otimes_k\tilde\cO)^\Gamma=\bigoplus_{i=0}^{e-1}\ft(i)\otimes_k \varpi^{\frac{e-i}{e}}\] 
The existence of a $W$-invariant nondegenerate symmetric bilinear form on $\ft$ gaurantees that $\dim_k\ft(i)=\dim_k\ft(e-i)$, from this we obtain that
\[\dim_k\left(\frac{\ft\otimes_k\tilde\cO}{M\otimes_\cO \tilde\cO} \right)=e(\dim_k\ft-\dim_k\ft^{w_\ga})= e\cdot c(\ga)\]
Combined with \eqref{P-a-ramified-case-calculation-eq}, we obtain
\[\dim\cP_a=\langle\rho,\la\rangle+\frac{1}{2}(d(\ga)-c(\ga)).\]
Finally, $X_\ga^{\la,\mathrm{reg}}$ is equidimensional since it is a finite union of $\cP_a$-torsors.
\end{proof} 

\subsubsection{Some 0-dimensional generalized affine Springer fibers}
Suppose $X_\ga^\la$ is nonempty. Then there exists $\ga_\la\in \gsc_+$ satisfying the conclusion of Lemma~\ref{lem:ga-la-lem}. Let $a:=\chi_+(\ga_\la)\in\fC_+(\cO)\cap\fC_{\gsc_+}^{\mathrm{rs}}(F)$. Recall the extended discriminant divisor $\fD_+\subset\fC_+$ defined in \S~\ref{sec:discr-divisor}. We define the \emph{extended discriminant valuation} to be
\[d_+(a):=\mathrm{val}(a^*\fD_+)\in\bZ\] 
From equation \eqref{d-ga-newton-eq} we get
\begin{equation}\label{d-la-eq}
\begin{split}
d_+(a)&=2\cdot\mathrm{val}(\rho(\alpha(\ga_\la)))+d(\ga)\\
&=\langle2\rho,\la\rangle+d(\ga)\\
&=\sum_{\substack{\alpha\in\Phi\\ \langle\alpha,\nu_\ga\rangle=0}}\mathrm{val}(\alpha(\ga)-1)+\langle 2\rho, \la-\nu_\ga\rangle
\end{split}
\end{equation}

\begin{prop}
Suppose $d_+(a)=0$. Then $\ga$ is split and $\mathrm{dim}X_\ga^\la=0$. Moreover, $X_\ga^\la=X_\ga^{\la,\mathrm{reg}}$ and it is a torsor under $P_a$.
\end{prop}
\begin{proof}
The assumption $d_+(a)=0$ implies that $a\in\fC_+^{\mathrm{rs}}(\cO)$. Let $\widetilde{X_a}=\spec R_a$ be defined by the Cartesian diagram
\[\xymatrix{
\widetilde{X_a}\ar[r]\ar[d] & \overline{T_+}\ar[d] \\
\spec\cO\ar[r] & \fC_+ 
}\]
By Proposition~\ref{prop:rs-disc-nonzero}, $\widetilde{X_a}$ is an \'etale cover $\spec\cO$ which must be trivial since the residue field $k$ is algebraically closed. Then we see that $\ga_\la\in T_+(F)$ is split and hence $\ga\in T(F)$ is split.\par 
Since $X_\ga^\la$ is nonempty, we have $\nu_\ga\le_\bQ\la$ and hence the terms on the right hand side of \eqref{d-la-eq} are non-negative. In particular, $d_+(a)=0$ implies that $\nu_\ga=\la$. Thus the proposition follows from Corollary~\ref{dim-unr-cor}.
\end{proof}

\subsection{The case of central coweight}
In this section we deal with the case where $\la\in X_*(T)_+$ is a central coweight, i.e. $\langle\la,\alpha\rangle=0$ for all roots $\alpha$. Then $\la\in X_*(Z^0)$ where $Z^0$ is the maximal torus in the center of $G$. Consequently we have $X_\ga^\la\cong X_{\varpi^{-\la}\ga}^0$. Hence the essential case is when $\la=0$ and the corresponding Kottwitz-Viehmann variety becomes
\[X_\ga:=\{g\in G(F)/G(\cO)|\mathrm{Ad}(g)^{-1}\ga\in G(\cO)\}.\]
We first do some routine reductions. Let $P=MN$ be a standard parabolic subgroup with standard Levi $M$ and unipotent radical $N$. For $\ga\in M(\cO)\cap G^{\mathrm{rs}}(F)\subset M(\cO)\cap M^{\mathrm{rs}}(F)$, we consider the Kottwitz-Viehmann variety $X_\ga$ (resp. $X_\ga^M$) defined for the groups $G$ (resp. $M$). We have the discriminant valuation $d(\ga)$ (resp. $d_M(\ga)$) defined for $G$ (resp. $M$). The two discriminant valuations are related by
\begin{equation}\label{eq:dGdM-relation}
d(\ga)=d_M(\ga)+2r_N(\ga)
\end{equation}
where $r_N(\ga)$ is defined in \eqref{eq:r-N}.

\begin{prop}\label{prop:X-ga-Levi}
With notation as above, we have
\[\dim X_\ga=\dim X_\ga^M+\frac{d_G(\ga)-d_M(\ga)}{2}\]
\end{prop}
\begin{proof}
Let $P=MN$ be the standard parabolic subgroup with Levi factor is $M$ and unipotent radical $N$. The connected components of $\mathrm{Gr}_M$ and $\mathrm{Gr}_P$ both corresponds bijectively to $\pi_1(M)$, the quotient of $X_*(T)$ by the coroot lattice of $M$. The canonical map $\mathrm{Gr}_P\to\mathrm{Gr}_G$ induces bijection on $k$-points by generalized Iwasawa decomposition. For each $\la\in\pi_1(M)$, let $X_{\ga,\la}$ be the intersection of $X_\ga$ and the connected component of $\mathrm{Gr}_P$ corresponding to $\la$. Similarly, let $X^M_{\ga,\la}$ be the intersection of $X^M_\ga$ with the connected component of $\mathrm{Gr}_M$ corresponding to $\la$. Then there is a canonical morphism 
\[p_\ga^M: X_{\ga,\la}\to X_{\ga,\la}^M\]
It suffices to show that the fibres of this map have dimension $r_N(\ga)$. \par 
Let $h\in X_{\ga,\la}^M$. Then $\ga_h:=h^{-1}\ga h\in M(\cO)$ and we consider the fibre $Y_h:=(p_\ga^M)^{-1}(h)$. Its set of $k$ points is
\[Y_h(k)=\{u\in N(F)/N(\cO)| u^{-1}\ga_h u\in G(\cO)\}\]
In other words, we have
\[Y_h=f_{\ga_h}^{-1}(N(\cO))/N(\cO)\]
where $f_{\ga_h}: N(F)\to N(F)$ is defined by $f_\ga(u)=u^{-1}\ga_h u\ga_h^{-1}$. Apply Proposition~\ref{prop:adm-set} to the admissible set $Z=N(\cO)$ we see that $Y_h$ is an irreducible affine space of dimension
\[\dim Y_h=r_N(\ga_h)=r_N(\ga)\]
and hence we conclude by \eqref{eq:dGdM-relation}.
\end{proof}
\begin{cor}\label{cor:dim-central-coweight}
Let $\la\in X_*(T)$ be a central coweight and $\ga\in G(F)^{\mathrm{rs}}$. Then 
\[\dim X_\ga^\la=\frac{1}{2}(d_\ga-c_\ga).\]
Moreover, $X_\ga^{\la,\mathrm{reg}}$ is a torsor under $P_a$ and the dimension of the complement of $P_a=X_\ga^{\la,\mathrm{reg}}$ in $X_\ga^\la$ is strictly smaller than the dimension of $X_\ga^\la$.
\end{cor}
\begin{proof}
We first assume that $\ga\in G(\cO)$ is topologically unipotent mod center. In other words, the reduction mod $\varpi$ of $\ga$ is unipotent mod center. After multiplying by an element in $Z(\cO)$, we may assume that $\ga\in\gsc(\cO)$ is topologically unipotent. Then when $G=\gsc$ the argument of \cite[\S4]{KL} and \cite[Proposition 3.7.1]{Ngo10} generalize verbatim to our situation and proves $\dim X_\ga^{\mathrm{reg}}=\dim X_\ga$ and the complement of $X_\ga^{\mathrm{reg}}$ has strictly smaller dimension. In particular the dimension formula follows in this situation. More generally, we argue as in \cite[Lemma 4.1]{Tsai} to reduce to the case $G=\gsc$. \par 
It remains to reduce to the case where $\ga$ is topologicallly unipotent mod center. After multiplying $\ga\in G(\cO)$ by an element in $Z(\cO)$ we may assume that $\ga\in\gsc(\cO)$. Then $G_\ga$ is a maximal torus in $G$ and $\ga\in G_\ga(F)\cap G(\cO)$. Let $S$ be the maximal split subtorus in the centralizer $G_\ga$. After conjugation we may assume that $S\subset T$.  Let $M=C_G(S)$ be the centralizer of $S$ in $G$. Then $M$ is a standard Levi subgroup of $G$ and $\ga\in M(\cO)$. Let $a_M:=\chi_M(\ga)\in\fC_M(\cO)$. Then the pullback of $T$ along $a_M:\spec\cO\to\fC_M$ is a totally ramified cover of $\spec\cO$ and we deduce that $\ga$ is topologically unipotent mod center in $M(\cO)$. Thus the result follows from the case already proved and Proposition~\ref{prop:X-ga-Levi}. 
\end{proof}
We highlight the following special case:
\begin{cor}\label{cor:0-dim-reg}
Let $\la\in X_*(T)$ be a central coweight and $\ga\in G(F)^{\mathrm{rs}}$. If $d_\ga\le1$, then $X_\ga^\la=X_\ga^{\la,\mathrm{reg}}=P_a$ and they are $0$-dimensional. 
\end{cor}
\begin{proof}
By Corollary~\ref{cor:dim-central-coweight}, if $d_\ga\le1$, we must have $\dim X_\ga^\la=0$. Moreover, the complement of $X_\ga^{\la,\mathrm{reg}}$ in $X_\ga^\la$ has strictly smaller dimension, hence must be empty.
\end{proof}
\subsection{Irreducible components}\label{irr-components-conjecture-section}
\subsubsection{Stratification on dominant coweight cone}
Let $\La:=X_*(T)$ and $\La_\bQ:=\La\otimes_\bZ\bQ$. Let $\sfD\subset\La_{\bQ}$ be the positive coroot cone. In other words, $\sfD$ consists of $\bQ$-linear combinations of simple coroots with non-negative coefficients.

For $\la\in\La^+$, we define the \emph{dominant coweight polytope} to be:
\[\sfP_\la:=\La_\bQ^+\cap\mathrm{Conv}(W\cdot\la)=\La_\bQ^+\cap (\la-\sfD).\]
where $\mathrm{Conv}(W\cdot\la)$ denotes the convex hull of the $W$-orbit of $\la$. 

\begin{lem}\label{polytope-intersection-lem}
For each $\la_1,\la_2\in\La^+$ with $\kappa_G(\la_1)=\kappa_G(\la_2)$, there exists $\mu\in\La^+$ such that $\mu\le\la_1,\mu\le\la_2$ and 
\[(\la_1-\sfD)\cap (\la_2-\sfD)=\mu-\sfD.\]
In particular, we have $\sfP_{\la_1}\cap\sfP_{\la_2}=\sfP_\mu$.
\end{lem}
\begin{proof}
Since $\kappa_G(\la_1)=\kappa_G(\la_2)$, the difference $\la_1-\la_2$ lies in the coroot lattice. There exists a partition of the set of simple coroots $\Delta^\vee=\Delta_1^\vee\sqcup\Delta_2^\vee$ such that 
\[\la_1-\la_2=\beta_1-\beta_2\]
where $\beta_i$ is a non-negative integral linear combinations of simple coroots in $\Delta_i^\vee$ for $i\in\{1,2\}$. Let $\Delta=\Delta_1\sqcup\Delta_2$ be the corresponding partion of the set of simple roots. Consider the coweight $\mu:=\la_1-\beta_1=\la_2-\beta_2.$ Then clearly $\mu\le\la_1$ and $\mu\le\la_2$.\par 
We claim that $\mu\in \La^+$. Take any simple root $\alpha\in\Delta_1$. Since $\beta_2$ is positive linear combination of coroots in $\Delta_2$, we have $\langle\alpha,\beta_2\rangle\le0$ and hence $\langle\mu,\alpha\rangle=\langle\la_2-\beta_2,\alpha\rangle\ge0$. Similarly, using $\mu=\la_1-\beta_1$, we see that for all $\alpha\in\Delta_2$, $\langle\mu,\alpha\rangle\ge0$. Thus we conclude that $\mu\in \La_+$.\par 
It is clear that $\mu-\sfD\subset(\la_1-\sfD)\cap(\la_2-\sfD)$. Now we prove the reverse inclusion. Let $\nu\in(\la_1-\sfD)\cap(\la_2-\sfD)$. Then for $i\in\{1,2\}$, $\la_i-\nu\in\sfD$ is a non-negative $\bQ$-linear combination of simple coroots and we need to show that $\mu-\nu\in\sfD$. For any fundamental weight $\omega$, there exists $i\in\{1,2\}$ so that $\omega$ is orthogonal to all coroots in $\Delta_i^\vee$. Without loss of generality assume $i=1$, then we have
\[\langle\mu-\nu,\omega\rangle=\langle\la_1-\beta_1-\nu,\omega\rangle=\langle\la_1-\nu,\omega\rangle\ge 0.\]
This means that $\nu\le_\bQ\mu$, or $\nu\in(\mu-\sfD)$. Therefore we have shown that $\mu-\sfD=(\la_1-\sfD)\cap(\la_2-\sfD)$.\par 
Finally, taking intersection with $\La_\bQ^+$, we get $\sfP_{\la_1}\cap\sfP_{\la_2}=\sfP_{\mu}$.
\end{proof}

For each $\la\in\La^+$, define
\begin{equation}\label{polytope-strata-eq}
\sfP_\la^\circ:=\sfP_\la-\bigcup_{\substack{\mu\in \La^+,\\
\mu<\la}} \sfP_\mu.
\end{equation}

\begin{cor}\label{cor:polytope-strata}
For any $\la_1,\la_2\in\La_+$ with $\la_1\ne\la_2$, we have
$\sfP_{\la_1}^\circ\cap\sfP_{\la_2}^\circ=\varnothing$. In particular, we get a well-defined stratification
\[\{\nu\in\La_{\bQ}^+ | p_{G,\bQ}(\nu)\in X_*(G_\mathrm{ab})\subset\pi_1(G)_\bQ\}=\bigsqcup_{\la\in\La^+}\sfP_\la^\circ.\]
\end{cor}
\begin{proof}
If $\det(\varpi^{\la_1})\ne\det(\varpi^{\la_2})$, it is clear that $\sfP_{\la_1}$ and $\sfP_{\la_2}$ are disjoint. Suppose $\det(\varpi^{\la_1})\ne\det(\varpi^{\la_2})$. Then  by Lemma~\ref{polytope-intersection-lem}, there exists $\mu\in\La^+$ such that $\mu\le\la_1,\mu\le\la_2$ and
\[\sfP_{\la_1}^\circ\cap\sfP_{\la_2}^\circ\subset\sfP_{\la_1}\cap\sfP_{\la_2}=\sfP_\mu.\]
But by \eqref{polytope-strata-eq}, we have $\sfP_\mu\cap\sfP_{\la_i}^\circ=\varnothing$ since $\mu\le\la_i$ for $i\in\{1,2\}$. Therefore $\sfP_{\la_1}^\circ\cap\sfP_{\la_2}^\circ=\varnothing$.
\end{proof}

\subsubsection{Stratification on extended Steinberg base}
To get a conceptually simpler formulation of the conjecture on irreducible components, we introduce a stratification on $\fC_+(\cO)\cap\fC_{\gsc_+}$.\par 
Recall that $\fC_+\cong A_{\gsc}\times\bA^r$. Consider the strata
\[\fC_+^\la:=\varpi^{-w_0(\lad)}\tad(\cO)\times\cO^r\subset\fC_+(\cO)\]
where $\lad\in X_*(\tad)_+$ is the image of $\la$.\par 
For each $\mu\in\La^+$ such that $\mu\le\la$, we have an embedding 
\[i_{\mu\la}:\fC_+^\mu\into\fC_+^\la\]
defined by the formula
\begin{multline}
i_{\mu\la}(a_1,\dotsc,a_r,b_1,\dotsc,b_r)=\\
(\varpi^{\langle-w_0(\la-\mu),\alpha_1\rangle}a_1,\dotsc,\varpi^{\langle-w_0(\la-\mu),\alpha_r\rangle}a_r,\varpi^{\langle-w_0(\la-\mu),\omega_1\rangle})b_1,\dotsc, \varpi^{\langle-w_0(\la-\mu),\omega_r\rangle}b_r)
\end{multline}
Note that we need to choose a uniformiser to define the embedding $i_{\mu\la}$ but its image does not depend on this choice.

\begin{prop}\label{Steinberg-strata-intersection-prop}
For any $\la,\mu_1,\mu_2\in\La^+$ with $\mu_1\le\la$ and $\mu_2\le\la$, there exists $\mu_3\in\La^+$ such that $\mu_3\le\mu_1,\mu_3\le\mu_2$ and 
\[i_{\mu_1\la}(\fC_+^{\mu_1})\cap i_{\mu_2\la}(\fC_+^{\mu_2})=i_{\mu_3\la}(\fC_+^{\mu_3}).\]
\end{prop}
\begin{proof}
By Lemma~\ref{polytope-intersection-lem}, there exists $\mu_3\in\La^+$ such that 
\begin{equation}\label{widecone-intersection-eq}
(\mu_1-\sfD)\cap (\mu_2-\sfD)=\mu_3-\sfD
\end{equation}
To prove the proposition, it suffices to show that 
\[i_{\mu_1\la}(\fC_+^{\mu_1})\cap i_{\mu_2\la}(\fC_+^{\mu_2})\subset i_{\mu_3\la}(\fC_+^{\mu_3})\]
Let $\iota$ be the involution on the set $\{1,\dotsc,r\}$ such that $\omega_{\iota(i)}=-w_0(\omega_i)$ for all $1\le i\le r$. For each $c=(c_1,\dotsc,c_r)\in\cO^r$,
let $a_i:=\mathrm{val}(c_{\iota(i)})$. 

Suppose that $\varpi^{(-w_0(\lad)},c)\in i_{\mu_1\la}(\fC_+^{\mu_1})\cap i_{\mu_2\la}(\fC_+^{\mu_2})$, then we get
\begin{equation}\label{a-i-inequality-eq}
a_i\ge\langle\la-\mu_1,\omega_i\rangle\text{  and  } a_i\ge\langle\la-\mu_2,\omega_i\rangle\quad\text{for all }1\le i\le r
\end{equation}
and we need to show that $a_i\ge\langle\la-\mu_3,\omega_i\rangle$ for all $1\le i\le r$.\par 
Let $\mu_1':=\sum_{i=1}^r\langle\mu_1,\omega_i\rangle\alpha_i^\vee$ and define $\mu_2',\mu_3'$. Then we have $\mu_1',\mu_2',\mu_3'\in\La_0^+$. Consider the coweight $\nu:=\sum_{i=1}^r (\langle\la,\omega_i\rangle-a_i)\alpha_i^\vee\in\La_0$. By \eqref{widecone-intersection-eq} and \eqref{a-i-inequality-eq} we have
\[\nu\in(\mu_1'-\sfD)\cap(\mu_2'-\sfD)=\mu_3'-\sfD.\]
This implies that
\[\langle\la,\omega_i\rangle-a_i=\langle\nu,\omega_i\rangle
\le\langle\mu_3',\omega_i\rangle=\langle\mu_3,\omega_i\rangle\]
which is what we want.
\end{proof}
For any $\la,\mu\in\La^+$ with $\mu\le\la$, define
\begin{equation}\label{Steinberg-strata-eq}
\fC_+^{\la\mu}:=i_{\mu\la}(\fC_+^\mu)-\bigcup_{\substack{\nu\in\Lambda_+\\ \nu<\mu}}i_{\nu\la}(\fC_+^{\nu}).
\end{equation}

\begin{cor}
For any $\la,\mu_1,\mu_2\in\La^+$ with $\mu_1\ne\la$ and $\mu_2\le\la$, we have
$\fC_+^{\la\mu_1}\cap\fC_+{\la\mu_2}=\varnothing$. In particular, we get well-defined stratifications
\[\fC_+^\la=\bigsqcup_{\substack{\mu\in\Lambda_+\\ \mu\le\la}}\fC_+^{\la\mu},\quad\fC_{\gsc_+}(F)\cap\fC_+(\cO)=\bigsqcup_{\substack{\la,\mu\in\Lambda_+\\ \mu\le\la}}\fC_+^{\la\mu}\]
\end{cor}
\begin{proof}
The argument is similar to the proof of Corollary~\ref{cor:polytope-strata}, using Proposition~\ref{Steinberg-strata-intersection-prop} instead of Lemma~\ref{polytope-intersection-lem}.
\end{proof}
The following lemma relates the stratas \eqref{Steinberg-strata-eq} the stratas \eqref{polytope-strata-eq}.

\begin{lem}\label{integral-approximation-lem}
For any $\la\in\La^+$ and $\ga\in G(F)^\mathrm{rs}$ with $\nu_\ga\le_\bQ\la$, there exists a \emph{unique} dominant integral coweight $\mu\in\La_+$ with $\mu\le\la$ that satisfies any (hence all) of the following equivalent conditions:
\begin{enumerate}
\item $\mu\in\La_+$ is a minimal dominant integral coweight such that $\nu_\ga\le_\bQ\mu$;
\item $\nu_\ga\in\sfP_\mu^\circ$, cf. \eqref{polytope-strata-eq};
\item $\chi_+(\ga_\la)\in\fC_+^{\la\mu}$, cf. \eqref{Steinberg-strata-eq}.
\end{enumerate}
\end{lem}
\begin{proof}
The equivalence between (1) and (2) folows from the definition of $\sfP_\mu$. The equivalence of (1) and (3) follows from Proposition~\ref{prop:nonempty}.\par 
Finally, the uniqueness of $\mu$ follows from Lemma~\ref{polytope-intersection-lem} or Proposition~\ref{Steinberg-strata-intersection-prop}.
\end{proof}
Now we state our conjecture on irreducible components of $X_\ga^\la$:
\begin{conjecture}\label{irr-components-conjecture}
Let $\la\in\La^+$ and $\ga\in G(F)^{\mathrm{rs}}$ with $\nu_\ga\le_\bQ\la$. Let $\mu\in\La_+$ be the ``best integral approximation" of $\nu_\ga$, i.e. the unique dominant coweight that satisfies the equivalent conditions in Lemma~\ref{integral-approximation-lem}. Then the number of $G_\ga^0(F)$-orbits on  $\mathrm{Irr}(X_\ga^\la)$ equals to the weight multiplicity $m_{\lambda\mu}$.
\end{conjecture}
By Corollary~\ref{dim-unr-cor}, this conjecture is true when $\ga$ is an unramified conjugacy class. 
\begin{rem}\label{best-integral-approx-rem}
For irreducible components of affine Deligne-Lusztig varieties, there is a similar conjecture made by Chen-Zhu, see the discussion in \citep{HaVi17} and \citep{XiaoZhu17}. In their setting, they also approximate Newton points of twisted conjugacy classes by integral coweight. However, the ``best integral approximation" as defined in \citep{HaVi17} is the largest integral coweight dominated by the Newton point. Whereas in the formulation of Conjecture~\ref{irr-components-conjecture}, we use the smallest integral coweight dominating the Newton point. Simple examples suggest that these two integral approximations are very likely in the same Weyl group orbit, so we expect the two weight multiplicities to be the same.
\end{rem}

\subsubsection{Components of the regular locus}\label{regular-components-section}
The $G_\ga^0(F)$-orbits on $\mathrm{Irr}(X_\ga^{\la,\mathrm{reg}})$ corresponds bijectively to $G_\ga^0(F)$ orbits on $X_\ga^{\la,\mathrm{reg}}$, which are precisely the $P_a$-orbits of maximal dimension on $\Sp_a^0\cong X_\ga^\la$. We know from Proposition~\ref{X-ga-w-torsor-prop} that these are the varieties $X_\ga^{\la,w}=\Sp_a^w$ for $w\in\mathrm{Cox(W,S)}$.\par  
However, for two different $w,w'\in\mathrm{Cox}(W,S)$, $X_\ga^{\la,w}$ and $X_\ga^{\la,w'}$ might coincide. For example, in the case $\la=0$ and $\ga\in G(\cO)$, all $X_\ga^{\la,w}$ coincide (hence equal to $X_\ga^{\la,\mathrm{reg}}$). So in this particular case $X_\ga^{\la,\mathrm{reg}}$ is the unique $\cP_a$-orbit of maximal dimension. In general, we know from \eqref{open-cover-V-reg-eq} that the number of $G_\ga^0(F)$ orbits in $X_\ga^{\la,\mathrm{reg}}$ is bounded above by the Cardinality of $\mathrm{Cox}(W,S)$. We will see that in many situations, this upper bound can be achieved (in other words $X_\ga^{\la,w}$ are mutually disjoint).  
\begin{thm}\label{regular-component-upper-bound-thm}
Let $\la\in\La^+$ and $\ga\in G(F)^{\mathrm{rs}}$  with $\nu_\ga\le_\bQ\la$. Let $\mu\in\La^+$ be the ``best integral approximation" of the Newton point $\nu_\ga$ as in Lemma~\ref{integral-approximation-lem}. Then we have an inequality
\[|\{G_\ga^0(F) \text{ orbits on }X_\ga^{\la,\mathrm{reg}}\}|\le |\mathrm{Cox}(W,S)|\]
where $\mathrm{Cox}(W,S)$ is the set of $S$-Coxeter elements defined in Definition~\ref{coxeter-definition}. Moreover, when $\la$ lies in the interior of the Weyl chamber and $\la-\mu$ lies in the interior of the positive coroot cone, the equality is achieved. 
\end{thm}
\begin{proof}
It remains to show the last statement. Suppose $\la$ lies in the interior of the Weyl chamber and $\la-\mu$ lies in the interior of the dominant coroot cone. Consider the following Cartesian diagram
\[\xymatrix{
\chi^{-1}_+(a)\ar[r]\ar[d] & \vin_{\gsc}\ar[d]^{\chi_+}\\
\spec\cO\ar[r]^a & \fC_+
}\]
For $g\in G(F)$ such that $g G(\cO)\in X_\ga^{\la,\mathrm{reg}}$, let $\overline{\mathrm{Ad}(g)^{-1}\ga}$ be the reduction mod $\varpi$ of $\mathrm{Ad}(g)^{-1}\ga\in \vin_{\gsc}^{\mathrm{reg}}(\cO)$. The condition that $\la$ lies in the interior of the Weyl chamber means that $\langle\la,\alpha_i\rangle>0$ for all simple roots $\alpha_i$. Hence the special fiber of $\chi_+^{-1}(a)$ lies in the asymptotic semigroup $\mathrm{As}(\gsc):=\alpha^{-1}(0)$ and in particular $\overline{\mathrm{Ad}(g)^{-1}\ga}\in\mathrm{As}(\gsc)\cap\vin_{\gsc}^\mathrm{reg}$.\par
Furthermore, the assumption that $\la-\mu$ lies in the interior of the positive coroot cone implies that $\langle\la-\mu,\omega_i\rangle>0$ for all fundamental weight $\omega_i$. Therefore the reduction mod $\varpi$ of $a$ equals to $0$ and the special fiber of $\chi_+^{-1}(a)$ is the nilpotent cone $\cN$. In particular, we get $\overline{\mathrm{Ad}(g)^{-1}\ga}\in\cN^{\mathrm{reg}}$.\par 
Consequently there is a bijection between $G_\ga^0(F)$ orbits on $X_\ga^{\la,\mathrm{reg}}$ and $G$ orbits on $\cN^{\mathrm{reg}}$, the latter of which corresponds bijectively to $\mathrm{Cox}(W,S)$ by Proposition~\ref{prop:N-components}.
\end{proof}

As an immediate consequence, we mention the following purely combinatorial result, which might be of independant interest:
\begin{cor}\label{lower-bound-multiplicity-cor}
Let $\la\ge\mu$ be dominant weights of a complex reductive group $G$. Suppose that $\la$ lies in the interior of the Weyl chamber and $\la-\mu$ lies in the interior of the positive root cone (the ``wide cone"). Then we have the following lower bound for the weight multiplicity
\[m_{\la\mu}\ge|\mathrm{Cox}(W,S)|\]
where the set $\mathrm{Cox}(W,S)$ is defined in \S\ref{coxeter-definition}.
\end{cor}
\begin{proof}
We consider the dual group $G^\vee$ of $G$ over $k$. Then $\la\ge\mu$ are dominant coweights for $G^\vee$. Let $T^\vee\subset G^\vee$ be a maximal torus and $\ga\in\varpi^\mu T^\vee(\cO)\cap G^\vee(F)^{\mathrm{rs}}$. Then the generalized affine Springer fibre $X_\ga^\la$ is nonempty and by Corollary~\ref{dim-unr-cor}, the number of $G^{\vee,0}_\ga(F)$-orbits on $\mathrm{Irr}(X_\ga^\la)$ equals to $m_{\la\mu}$. On the other hand, by Theorem~\ref{regular-component-upper-bound-thm}, the number of $G^{\vee,0}_\ga(F)$-orbits on $\mathrm{Irr}(X_\ga^{\la,\mathrm{reg}})$ equals to $|\mathrm{Cox}(W,S)|$, hence the inequality.
\end{proof}
\begin{rem}
If $\gad$ is simple of rank $r$, then $|\mathrm{Cox}(W,S)|=2^{r-1}$. In general, if the simple factors of $\gad$ has rank $r_1,\dotsc,r_m$, then 
\[|\mathrm{Cox}(W,S)|=\prod_{i=1}^m2^{r_i-1}.\]
We expect that there should be a more straightforward proof of Corollary~\ref{lower-bound-multiplicity-cor}. 
\end{rem}
\begin{rem}\label{regular-components-rem}
In general, the weight multiplicity $m_{\la\mu}$ will increase with $\la$ while the right hand side in Corollary~\ref{lower-bound-multiplicity-cor} is a fixed constant independant of $\la,\mu$. Thus in general there will be much more irreducible components in $X_\ga^\la$ than the regular open subvariety $X_\ga^{\la,\mathrm{reg}}$.
\end{rem}

\section{The Hitchin-Frenkel-Ng\^o fibration}\label{chapter:HFN}
In this section we study global analogue of Kottwitz-Viehmann varieties, the Hitchin-Frenkel-Ng\^o fibration. These are certain group analogue of Hitchin fibrations, first introduced in \cite{FN} and later studied in more detail in \cite{Bou17} and \cite{Bou15b}. \par 
Throughout this section we let $X$ be a projective smooth curve of genus $g$ over $k$ and $G$ a connected reductive group over $k$. 
\subsection{First definitions}
Let $\cL$ be a $Z^{\mathrm{sc}}_+=\tsc$ torsor on $X$. Then we can twist the schemes $\vin_{\gsc}$ (resp. $\fC_+$, $A_{\gsc}$) by $\cL$ to form corresponding affine spaces $\vin_{\gsc}^\cL$ (resp. $\fC_+^\cL$, $A_{\gsc}^\cL$) over $X$.
\begin{defn}
The \emph{Hitchin-Frenkel-Ng\^o moduli stack} associated to the $\tsc$-torsor $\cL$ is the mapping stack 
\[\cM_\cL:=\Hom(X,[\vin_{\gsc}^\cL/\mathrm{Ad}(G)])\]
\end{defn}
In other words, $\cM_\cL$ classifies pairs $(\cE,\varphi)$ where $\cE$ is a $G$-torsor on $X$ and $\varphi$ is a section of $\cE\wedge^G\vin_{\gsc}^\cL$ where $G$ acts on $\vin_{\gsc}^\cL$ by adjoint action, and the action factors through $\gad$. We refer to such pairs $(\cE,\varphi)$ as \emph{Higgs-Vinberg pairs}.\par 
Replacing $\vin_{\gsc}$ by $\vin_{\gsc}^0$ (resp. $\vin_{\gsc}^{\mathrm{reg}}$) in the definition of $\cM_\cL$, we define open substacks $\cM_\cL^0$ (resp. $\cM_\cL^{\mathrm{reg}}\subset\cM_\cL$). Also we define
\[\cA_\cL:=\Hom_X(X,\fC_+^\cL),\quad\cB_\cL:=\Hom_X(X,A_{\gsc}^\cL)\]
as the space of sections of the affine space $\fC_+^\cL$ (resp. $A_{\gsc}^\cL$) over $X$. More concretely, we can describe $\cA_\cL$ and $\cB_\cL$ as follows.\par
For each $\omega\in X^*(T)$, let $\omega(\cL)$ be the invertible sheaf on $X$ defined by pushing $\cL$ along the morphism $\omega:T\to\bG_m$. 
Then we have 
\[\cB_\cL=H^0(X,A_{\gsc}^\cL)=\bigoplus_{i=1}^r H^0(X,\alpha_i(\cL))\]
and
\[\cA_\cL=\cB_\cL\oplus\bigoplus_{i=1}^r H^0(X,\omega_i(\cL)).\]
\begin{defn}
The \emph{Hitchin-Frenkel-Ng\^o} fibration is the morphism
\[h_\cL:\cM_\cL\to\cA_\cL\]
induced by $\chi_+:\vin_{\gsc}\to\fC_+$.
\end{defn}
Let $\beta_\cL:\cA_\cL\to\cB_\cL$ be the natural projection and $\alpha_\cL:=\beta_\cL\circ h_\cL:\cM_\cL\to\cB_\cL$ be the map induced by $\alpha:\vin_{\gsc}\to A_{\gsc}$. We call the fibres of $\alpha_\cL$ \emph{restricted Hitchin-Frenkel-Ng\^o moduli stack}.\par  
\subsubsection{}\label{sec:lambda-divisor}
Each point $b\in\cB_\cL$ can be written as $b=(b_1,\dotsc,b_r)$ where $b_i\in H^0(X,\alpha_i(\cL))$.  Let $\cB_\cL^\circ\subset\cB_\cL$ be the open subset consisting of those $b$ such that $b_i$ is nonzero for all $i$.
To each point $b\in\cB_\cL^\circ$, we can associate an $X_*(\tad)_+$-valued divisor $\la_{b}$ on $X$ defined by
\[\la_{b}:=\sum_{i=1}^r \check{\omega}_i D(b_i)\]
where $D(b_i)$ is the effective divisor on $X$ associated to $b_i$ and $\check{\omega}_i$ is the $i$-th fundamental coweight. For any $a\in\cA_\cL$ with $\beta_\cL(a_+)\in\cB_\cL^\circ$, we denote $\la_a:=\la_{\beta_\cL(a)}$.

\begin{defn}
The \emph{generically regular semisimple locus} $\cA_\cL^\heartsuit$ is the open subset of $\cA_\cL$ consisting of sections
$a:X\to\fC_+^\cL$ such that $\beta_\cL(a)\in\cB_\cL^\circ$ and $a(X)$ generically lies in the open subset $\fC_+^{\mathrm{rs},\cL}=\fC_+^\cL-\fD_+^\cL$.
\end{defn}

\subsubsection{Global Steinberg section}\label{sec:global-Steinberg}
Let $c=|Z(G_{\mathrm{der}})|$ be the order of the center of the derived group of $G$. Suppose there exists a $\tsc$-tosor $\cL'$ such that $\cL\cong (\cL')^{\otimes c}$. By definition, there is a canonical map $[\mathrm{ev}]_\cL:\cA_\cL\times X\to [\fC_+/\tsc]$ making the following diagram commutative:
\[\xymatrix{
\cA_\cL\times X\ar[r]^{[\mathrm{ev}]_\cL}\ar[d] & [\fC_+/\tsc]\ar[d]\\
X\ar[r]^\cL & \bB\tsc
}\]
Here the left arrow is projection to $X$ and the bottom arrow corresponds to the $\tsc$-torsor $\cL$. \par 
The choice of $c$-th root $\cL'$ of $\cL$ defines a morphism $[\mathrm{ev}]_{\cL'}:\cA_\cL\times X\to[\fC_+/\tsc]$ lifting $[\mathrm{ev}]_\cL$. Then for each $w\in\mathrm{Cox}(W,S)$ (cf. Definition~\ref{coxeter-definition}), the composition of $[\mathrm{ev}]_{\cL'}$ and the section $\epsilon_{+,[c]}^w$ of $[\chi_+]_{[c]}$ (cf. Proposition~\ref{prop:twisted-Steinberg-section}) induces a section of $h_\cL$:
\[\epsilon^w_{\cL'}: \cA_\cL\to\cM_\cL^{\mathrm{reg}}\subset\cM_\cL.\]
We refer to $\epsilon^w_\cL$ as the \emph{global Steinberg section}. 
\subsection{Symmetries of Hitchin-Frenkel-Ng\^o fibration}
\begin{defn}
Let $\cP_\cL$ be the Picard stack over $\cA_\cL$ that associates to any $S$-point $a\in\cA_\cL(S)$ the Picard groupoid $\cP_a$ of $\cJ_a$ torsors on $X\times S$. Here $\cJ_a$ is the pull back of the universal centralizer $\cJ_\cL$ on $\fC_+^\cL$ along the map $a: X\times S\to\fC_+^\cL$.
\end{defn}
\begin{prop}\label{prop:P-smooth}
$\cP_\cL$ is a smooth Picard stack over $\cA_\cL$. 
\end{prop}
\begin{proof}
The argument of \cite[Proposition 4.3.5]{Ngo10} generalize \emph{verbatim} to our situation. The point is that $\cJ_a$ is a  smooth group scheme and the obstrction to deforming a $\cJ_a$-torsor lives in $H^2(X,\mathrm{Lie}(\cJ_a))$, which vanishes since $X$ is a curve. 
\end{proof}
The action of $\bB\cJ$ on $[\vin_{\gsc}/\mathrm{Ad}(G)]$ (resp. $[\vin_{\gsc}^{\mathrm{reg}}/\mathrm{Ad}(G)]$) induces action of $\cP_\cL$ on $\cM_\cL$ (resp. $\cM_\cL^{\mathrm{reg}}$).\par 
To understand the connected components of the fibres of $\cP_\cL$, we utilize cameral covers.
\begin{defn}
The \emph{cameral cover} associated to each $a\in\cA_\cL(k)$
is the finite flat cover $\pi_a:\widetilde{X}_a\to X$ defined by the following Cartesian diagram
\[\xymatrix{
\widetilde{X}_a\ar[r]\ar[d]_{\pi_a} & \overline{\tsc_+}^\cL\ar[d]\\
X\ar[r]^a & \fC_+^\cL
}\]
\end{defn}
For any closed point $a\in\cA_\cL^\heartsuit$, we define the \emph{discriminant divisor} for $a$ to be the effective divisor
\[\Delta_a:=a^{-1}(\fD_+^\cL)\]

Over the nonempty open subset $U_a:=X-\Delta_a$, the cameral cover $\pi_a$ is Galois \'etale with Galois group $W$. Choosing a point $\tilde{u}\in\widetilde{X}_a$ with $u:=\pi_a(\tilde{u})\in U_a$, we get a homomorphism
\[\rho_a:\pi_1(U_a,u)\to W\]
whose image is a subgroup $W_a\subset W$. Note that the conjugacy class of $W_a$ in $W$ is independant of the choice of base point $\tilde{u}$.

Let $\cJ_a^0\subset\cJ_a$ be the fibrewise neutral component and consider the Picard stack $\cP_a':=\mathrm{Bun}_{\cJ_a^0}$ of $\cJ_a^0$-torsors on $X$. Then there is a natural homomorphism of Picard stacks $\cP_a'\to\cP_a$. The following Lemma is parallel to \cite[Lemme 4.10.2]{Ngo10} with exactly the same proof. 
\begin{lem}
The homomorphism $\cP_a'\to\cP_a$ is surjective with finite kernel. Same is true for the induced homomorphism $\pi_0(\cP_a')\to\pi_0(\cP_a)$.
\end{lem}
\begin{cor}\label{cor:pi_0(P)-finite}
$\pi_0(\cP_a)$ is finite if and only if $T^{W_a}$ is finite.
\end{cor}
\begin{proof}
By previous lemma, $\pi_0(\cP_a)$ is finite if and only if $\pi_0(\cP_a')$ is finite. By \cite[Corollaire 6.7]{Ngo06}, $\pi_0(\cP_a')=\hat{T}^{W_a}$. Since the finiteness of $T^{W_a}$ is equivalent to the finiteness of $\hat{T}^{W_a}$, the result follows.
\end{proof}
\begin{defn}
The \emph{anisotropic locus} is the subset $\cA_\cL^{\mathrm{ani}}\subset\cA_\cL^\heartsuit$ consisting of $a\in\cA_\cL^\heartsuit$ such that the component group $\pi_0(\cP_a)$ is finite.
\end{defn}

For each subset $I\subset\Delta$, we consider the invariant quotient $\overline{\tsc_+}^{W_I}$. Then the natural morphism $\overline{\tsc_+}^{W_I}\to\fC_+$ is finite and $Z^{\mathrm{sc}}_+=\tsc$ equivariant. Denote $\overline{\tsc_+}^{W_I,\cL}:=\overline{\tsc_+}^{W_I}\times^{Z_+^{\mathrm{sc}}}\cL$. \par 
Let $\cA_\cL^{W_I}:=H^0(X,\overline{\tsc_+}^{W_I,\cL})$ be the space of sections of the affine scheme $\overline{\tsc_+}^{W_I,\cL}$ over $X$. Consider the map
\[\nu_I: \cA_\cL^{W_I}\to\cA_\cL\]
induced by the finite morphism $\overline{\tsc_+}^{W_I,\cL}\to\fC_+^\cL$. Let $\cA_\cL^{W_I,\heartsuit}:=\nu_I^{-1}(\cA_\cL^\heartsuit)$. 

\begin{prop}
Suppose $G$ is semisimple. Then the complement of $\cA_\cL^{\mathrm{ani}}$ in $\cA_\cL^\heartsuit$ is a finite union
\[\cA_\cL^\heartsuit\setminus\cA_\cL^{\mathrm{ani}}=\bigcup_{I\subsetneqq\Delta}\nu_I(\cA_\cL^{W_I,\heartsuit}).\]
\end{prop}
\begin{proof}
Let $a\in\cA_\cL^{\heartsuit}-\cA_\cL^{\mathrm{ani}}$. Then by Corollary~\ref{cor:pi_0(P)-finite}, $T^{W_a}$ contains a nontrivial torus $S$. Since $G$ is semisimple, the centralizer of $S$ is a \emph{proper} Levi subgroup of $G$ whose simple roots form a proper subset $I\subsetneqq W$. Then we have $W_a\subset W_I$.\par 
Consider the following diagram in which both squares are Cartesian:
\[\xymatrix{
\widetilde{X}_a\ar[r]\ar[d]\ar@/^1pc/[rr]^{\pi_a} & Y_a\ar[r]_{\pi_a^I}\ar[d] & X\ar[d]\\
\overline{\tsc_+}^\cL\ar[r] & \overline{\tsc_+}^{W_M,\cL}\ar[r] & \fC_+^\cL
}\]
Let $\widetilde{Y_a}\subset\widetilde{X_a}$ be the union of all irreducible components that contain a point in the $W_I$-orbit of $\tilde{u}$. Then the image of $\widetilde{Y}_a$ in $Y_a$ is isomorphic to $X$ and hence gives a section of the morphism $\pi_a^I$. In other words, there is a section $a_I:X\to\overline{\tsc_+}^{W_M,\cL}$ such that $\nu_I(a_I)=a$. This proves that 
\[\cA_\cL^\heartsuit\setminus\cA_\cL^{\mathrm{ani}}\subset\bigcup_{I\subsetneqq\Delta}\nu_I(\cA_\cL^{W_I,\heartsuit}).\] 
Conversely, for any $I\subsetneqq\Delta$ and $a_I\in\cA_\cL^{W_I,\heartsuit}$ with $\nu_I(a_I)=a$, the morphism $\pi_a^I$ in the diagram above has a section given by $a_I$. This implies that $W_a\subset W_I$ so that $T^{W_a}$ is not finite. By Corollary~\ref{cor:pi_0(P)-finite} again we see that $a\in\cA_\cL^{\mathrm{ani}}$.
\end{proof}
\begin{cor}\label{cor:ani-complement}
Suppose $G$ is semisimple. Then $\cA_\cL^{\mathrm{ani}}$ is an open subset of $\cA_\cL^\heartsuit$. Moreover, for any $b\in\cB_\cL^\circ$ and any integer $N$ with $N>\max\{2g-2,rg\}$, if $\deg\omega_i(\cL)>N$ for all $1\le i\le r$, then the complement of $\cA_{\cL,b}^{\mathrm{ani}}$ in $\cA_{\cL,b}^\heartsuit$ has codimension at least $N-rg$.
\end{cor}
\begin{proof}
By valuative criterion and \cite[Lemme 7.3]{Ngo06} we see that $\nu_I$ is proper. So the images $\nu_I(\cA_\cL^{W_I,\heartsuit})$ are closed subsets of $\cA_\cL^\heartsuit$ and their complement $\cA_\cL^{\mathrm{ani}}$ is open. It remains to calculate the dimension of $\cA_\cL^{W_I}$.\par 
Let $I\subsetneqq\Delta$ and $L_I$ a corresponding Levi subgroup of $\gsc$. We label the fundamental weights $\omega_1,\dotsc,\omega_r$
of $\gsc$ so that $\omega_1,\dotsc,\omega_s$ are fundamental weights for $L_I$ where $s=|I|<r$. There is a natural morphism
\[q^I: \overline{\tsc_+}^{W_I}\to A_{\gsc}\times\bA^s\]
given by the $W_I$-invariant functions $(\alpha_i,0)$ for $1\le i\le r$ and $(\omega_i,\chi^I_{\omega_i})$ for $1\le i\le s$, where $\chi^I_{\omega_i}$ is the character of the irreducible representation of $L_I$ with highest weight $\omega_i$. The map $q^I$ induces a map
\[q_X^I: \cA_\cL^{W_I,\heartsuit}\to
\cB_\cL^\circ\oplus\bigoplus_{i=1}^s H^0(X,\omega_i(\cL))\]
The fibres of $q^I$ over the open subset 
$\tad\times\bA^s\subset A_{\gsc}\times\bA^s$ are isomorphic to $\bG_m^{r-s}$. This implies that the nonempty fibres of $q_X^I$ are $(k^\times)^{r-s}$. Hence
\[\dim\cA_\cL^{W_I,\heartsuit}\le\dim\cB_\cL+\sum_{i=1}^s (\deg(\omega_i(\cL))+1-g)+r-s.\]
Therefore, the codimension of $\cA_{\cL,b}^{\heartsuit}-\cA_{\cL,b}^{\mathrm{ani}}$ is bounded below by
\[\sum_{i=1}^r(\deg(\omega_i(\cL))+1-g)-[\sum_{i=1}^s (\deg(\omega_i(\cL))+1-g)+r-s]\ge N-rg.\]
\end{proof}
Denote $\cM_\cL^{\mathrm{ani}}:=h_\cL^{-1}(\cA_\cL^{\mathrm{ani}})$ the anisotropic open substack. This is nonempty when $G$ is semisimple. Also, let $\cP_\cL^{\mathrm{ani}}$ be the restriction of $\cP_\cL$ to $\cA_\cL^{\mathrm{ani}}$.

\begin{prop}
$\cM_\cL^{\mathrm{ani}}$ and $\cP_\cL^{\mathrm{ani}}$ are  Deligne-Mumford stacks.
\end{prop}
\begin{proof}
Let $(\cE,\varphi)\in\cM_\cL^{\mathrm{ani}}(k)$ and $a=h_\cL(\cE,\varphi)$. Then the $k$-group $\mathrm{Aut}(\cE,\varphi)$ classifies sections of the group scheme $\mathrm{Aut}_G(\cE)_\varphi$ over $X$, which is the closed subscheme of centralizer of $\varphi$ in the group scheme $\mathrm{Aut}_G(\cE)$.\par 
Choose a geometric point $\bar{\eta}$ over the generic point $\eta$ of $X$. Restricting the cameral cover to $\eta$ along $a$, we obtain a homomorphism $\rho_a^\eta:\mathrm{Gal}(\bar{\eta}/\eta)\to W$. Let $W_a$ be the image of $\rho_a^\eta$. Furthermore, choose a trivialization of $\cE$ over the generic point $\eta$ under which $\varphi$ maps to a regular semisimple element in $T_+(k(X))$. With these choice we get a closed embeddings $\mathrm{Aut}(\cE,\varphi)\subset T^{W_a}$ and $H^0(X,\cJ_a)\subset T^{W_a}$.\par 
Since $a\in\cA_\cL^{\mathrm{ani}}$, $T^{W_a}$ is finite. Since $\mathrm{char}(k)$ is coprime to the order of $W$, $T^{W_a}$ is finite unramified $k$-group. This shows that $\cM_\cL^{\mathrm{ani}}$ and $\cP_\cL^{\mathrm{ani}}$ are Deligne-Mumford stacks.
\end{proof}
\begin{thm}\label{thm:product-formula}
Assume that the $\tsc$-torsor $\cL$ admits a $c$-th root $\cL'$. Then for any $a\in\cA_\cL^\mathrm{ani}$, there is a homeomorphism of quotient stacks
\begin{equation}\label{eq:product-formula}
[\cM_a/\cP_a]\cong\prod_{x\in X-U_a}[\Sp_{a_x}/P_{a_x}]
\end{equation}
In particular, we have
\[\dim\cM_a-\dim\cP_a=\sum_{x\in\mathrm{Supp}(\Delta_a)}(\dim\Sp_{a_x}-P_{a_x}).\]
\end{thm}
\begin{proof}
Choose a Coxeter element $w\in\mathrm{Cox}(W,S)$. The $c$-th root $\cL'$ of $\cL$ induces a global Steinberg section $\epsilon_{\cL'}^w$, in particular a base point $\epsilon_{\cL'}^w(a)\in\cM_a^{\mathrm{reg}}$. Using Corollary~\ref{cor:ving-rs}, we argue as in the proof of \cite[Th\'eor\`eme4.6]{Ngo06} to show that there is a morphism as \eqref{eq:product-formula} inducing equivalence of groupoids on $k$-points. Then the argument of \cite{Ngo10} shows that the map \eqref{eq:product-formula} is a homeomorphism.
\end{proof}
\subsection{Properness over the anisotropic locus}
Throughout this section, we assume $G$ is semisimple so that $\cA_\cL^{\mathrm{ani}}$ is nonempty. Our goal is to show that the morphism $h_\cL^{\mathrm{ani}}:\cM_\cL^{\mathrm{ani}}\to\cA_\cL^{\mathrm{ani}}$ is proper. 

\subsubsection{Finiteness properties}
We first show that the Hitchin-Frenkel-Ng\^o fibration is of finite type over the anisotropic locus.\par 
We start with a more general situation. Let $\rho:G\to\mathrm{GL}(V)$ a finite dimensional representation such that $\ker(\rho)$ is contained in the center of $G$. Fix a torus $T$ and a Borel subgroup $B$ containing $T$. Let $V^{(1)},\dotsc,V^{(m)}$ be the irreducible consitituents of $V$ (counted with multiplicity) and $\la^{(1)},\dotsc,\la^{(m)}$ be the corresponding highest weight.\par 
For each $V^{(j)}$, we choose a basis $\{e_i^{(j)},1\le i\le d_j\}$ (where $d_j=\dim V_j$) as follows. Each $e_i^{(j)}$ is a weight vector with weight $\la_i^{(j)}\in X^*(T)$. Then we can express $\la^{(j)}-\la_i^{(j)}$ as a linear combination of positive simple roots with non-negative integer coefficients and we call the sum of coefficients the \emph{height} of $e_i^{(j)}$. The basis elements $e_i^{(j)}$ are indexed so that the height is non-decreasing with respect to $i$. In particular, $e_1^{(j)}$ is a highest weight vector and $e_{d_j}^{(j)}$ is a lowest weight vector in $V_j$.\par 
Then under the basis $\{e_i^{(j)},1\le i\le d_j,1\le j\le m\}$, $\rho(B)$ consists of upper triangular matrices in $\prod_j\mathrm{End}(V_j)$, which are the stabilizers of the standard flags $0=L_0^{(j)}\subset L_1^{(j)}\subset\dotsm\subset L_{d_j}^{(j)}=V^{(j)}$ where $L_i^{(j)}=\mathrm{Span}(e_1^{(j)},\dotsc,e_{i}^{(j)})$ for $1\le i\le d_j$.

Let $I\subset\Delta$ be a subset of simple roots and $P_I\subset G$ the standard parabolic subgroup whose Levi factor has simple roots in $I$. Then there exists standard parabolic subalgebras $\mathfrak{p}_I^{(j)}\subset\mathrm{End}(V^{(j)})$ such that 
\[\rho(P_I)=\rho(G)\cap(\bigoplus_{j=1}^m\mathfrak{p}_I^{(j)}).\]
More precisely, $\mathfrak{p}_I^{(j)}$ is the stabilizer of the partial flag in $V^{(j)}$ obtained from the standard flag by replacing $L_i^{(j)}$ with the span of $L_i^{(j)}$ and all basis vectors whose corresponding weight differs from the weight of $e_i^{(j)}$ by a linear combination of simple roots in $I$.\par 
Fix a divisor $D$ on a smooth projective curve $X$. Consider the following stack
\[\cM_V:=\mathrm{Hom}(X,[(\prod_{j=1}^m\mathrm{End}(V^{(j)})(D))/G])\]
where the action of $G$ on $\prod_{j=1}^m\mathrm{End}(V^{(j)})$ is induced by $\rho$. More concretely, the moduli stack $\cM_V$ classifies tuples $(E,\varphi_j,1\le j\le m)$ where $E$ is a $G$-torsor and $\varphi_j:\rho_j E\to\rho_j E(D)$ is a meromorphic endomorphism of the vector bundle $\rho_j E:=E\wedge^{(G,\rho)}V^{(j)}$.\par 
From the definition, we have
\[\cM_V=\cM_1\times_{\mathrm{Bun}_G}\cM_2\times_{\mathrm{Bun}_G}\dotsm\times_{\mathrm{Bun}_G}\cM_m\]
where for each $1\le j\le m$, we define
\[\cM_j=\mathrm{Hom}(X,[(\mathrm{End}(V^{(j)})(D))/G]).\]
By ??? we know that there exists a constant $C>0$ such that for any $G$-torsor $E$ on $X$ there exists a Borel reduction $E_B$ of $E$ so that $\deg(E_B)$ belongs to 
\[\mathsf{C}:=\{H\in\La_\bQ, \alpha(H)\ge-c\;\forall \alpha\in\Delta\}.\]
Let $N$ be a positive integer which is larger than the sum of coefficients of $\la^{(j)}-\la_i^{(j)}$ under the basis $\Delta$ for all $i,j$. Let $d$ be an integer such that
\begin{equation}\label{eq:d-lower-bound}
d>\deg(D)+2Nc
\end{equation}

For each subset $I\subset\Delta$, consider the following cone
\[\mathsf{C}_I:=\{H\in\La_\bQ, \alpha(H)\le d\; \forall \alpha\in I\textrm{ and }\alpha(H)\ge d\; \forall \alpha\in\Delta-I\}.\]

\begin{lem}
Let $(E,\varphi_j)\in\cM_V$ and $E_B$ a $B$-reduction of $E$. Suppose that $\deg E_B\in\mathsf{C}\cap\mathsf{C}_I$, then we have
\[\varphi\in\mathfrak{p}(D)\wedge^BE_B\]
where $\mathfrak{p}=\bigoplus_{j=1}^m\mathfrak{p}_I^{(j)}$
\end{lem}
\begin{proof}
We can treat each factor $\cM_j$ separately and assume that $V$ is irreducible. 
It suffices to prove that under the adjoint action of $\varphi$, $E_B\wedge^B\mathfrak{b}$ is sent into $E_B\wedge^B\mathfrak{p}(D)$. Consider a filtration of $\mathrm{End}(V^{(j)})$:
\[(0)=\mathfrak{b}_0\subset\mathfrak{b}_1\subset\dotsm\subset\mathfrak{b}_r=\mathfrak{b}\subset\mathfrak{p}=\mathfrak{p}_s\subset\mathfrak{p}_{s-1}\subset\dotsm\subset\mathfrak{p}_0=\mathrm{End}(V^{(j)}).\]
stable under adjoint action of $B$, with one-dimensional successive quotients.\par 
Suppose the image of $E_B\wedge^B\mathfrak{b}$ under $\mathrm{ad}(\varphi)$ is not contained in $E_B\wedge^B\mathfrak{p}(D)$. Then there exists $0<i\le r$ and $0\le j<s$ such that
 $\mathrm{ad}(\varphi)$ induces a \emph{non-zero} homomorphism of line bundles
\[E_B\wedge^B(\mathfrak{b}_i/\mathfrak{b}_{i-1})\to E_B\wedge^B(\mathfrak{p}_j/\mathfrak{p}_{j+1})(D).\]
In particular, the degree of the source is not larger than the degree of the target. More precisely, let $\ga$ be the weight of $B$ on $\mathfrak{b}_i/\mathfrak{b}_{i+1}$ and $\delta$ the weight of $B$ on $\mathfrak{p}_j/\mathfrak{p}_{j+1}$. Then we have the inequality
\[\langle\deg E_B,\ga-\delta \rangle\le\deg D.\]
Note that $\ga$ is the difference between the highest weight $\la^{(j)}$ and certain weight of the $G$-representation $V^{(j)}$, hence a non-negative linear combination of simple roots with the sum of coefficients bounded by $N$. Since $\deg E_B\in\mathsf{C}$, we then have
\[\langle\deg E_B,\ga\rangle\ge -Nc.\]
On the other hand, by definition of $\mathfrak{p}=\mathfrak{p}_I^{(j)}$, we see that $-\delta$ is a non-negative linear combination of simple roots such that the sum of coefficients is bounded by $N$ and the coefficient of some root in $\Delta-I$ is positive. Hence because $\deg E_B\in\mathsf{C}\cap\mathsf{C}_I$, we have
\[\langle\deg E_B,-\delta\rangle\ge d-Nc.\]
Combining the above two inequalities, we get $d-2Nc\le\deg D$ which contradicts \eqref{eq:d-lower-bound} and thus the lemma follows.
\end{proof}

\begin{prop}
The stack $\cM^{\mathrm{ani}}_\cL$ is of finite type.
\end{prop}
\begin{proof}
The natural morphism $\cM^{\mathrm{ani}}\to\mathrm{Bun}_G$ is of finite type. For each $\nu\in X^*(T)$, the moduli stack $\mathrm{Bun}_B^\nu$ of $B$-bundles on $X$ with degree $\nu$ is of finite type. It suffices to show that there is a finite subset $S\subset X^*(T)$ such that the image of $\cM^{\mathrm{ani}}$ in $\mathrm{Bun}_G$ is contained in the image of $\cup_{\nu\in S}\mathrm{Bun}_B^\nu$ in $\mathrm{Bun}_G$.\par 
Let $m=(\cE,\varphi)\in\cM_\cL^{\mathrm{ani}}(k)$ and $\cE_B$ a $B$-reduction of $\cE$ such that $\deg(\cE_B)\in\mathsf{C}$. Let $a=h_\cL(m)\in\cA_\cL^{\mathrm{ani}}$.
Suppose that $\deg(\cE_B)\in\mathsf{C}_I$ for some \emph{proper} subset $I\subset\Delta$. Then $\varphi$ maps the generic point of the curve into the proper parabolic subgroup $P_{I,+}$ of $\gsc_+$. This implies that $W_a$ is contained in the Weyl group of the Levi $L_I$ and hence $T^{W_a}$ is not finite, contradicting the fact that $a\in\cA_\cL^{\mathrm{ani}}(k)$. Consequently, we have $\det\cE_B$ lies in the intersection of $\mathsf{C}$ and the complement of $\mathsf{C}_I$ for any proper subgroup $I\subset\Delta$. This intersection is a bounded subset of $X^*(T)_\bR$ and hence the set of weights $S\subset X^*(T)$ lying in the intersection is a finite set. 
\end{proof}

\subsubsection{Valuative criterion}
First we have the existence part of the valuative criterion, which is true over the larger open subset $\cA_\cL^\heartsuit$.
\begin{prop}
Let $R$ be a complete discrete valuation ring with algebraically closed residue field containing $k$. Let $K$ be the fraction field of $R$. Then for all $a\in\cA_G^\heartsuit(R)$ and $m_K\in\cM_G^\heartsuit(K)$ such that $h_G(m_K)=a$, there exists a finite extension $K'$ of $K$ and $m\in\cM_G(R')$, where $R'$ is the integral closure of $R$ in $K'$, such that
\begin{enumerate}
\item The image of $m$ in $\cM^\heartsuit_G(K')$ is isomorphic to that of $m_K$;
\item $h_G(m)=a$.
\end{enumerate}
\end{prop}
\begin{proof}
The argument is the same as \cite[\S8.4]{ChLau10}. The key points are: 1. Any $G$-torsor extends uniquely over a codimension 2 subset; 2. the universal twisted monoid $\mathbb{V}_G$ over $\cA_G\times X$ is affine, so that Higgs fields extends over any codimension 2 subset. 
\end{proof}

\begin{prop}
Suppose $G$ is semisimple. Let $R$ be a complete discrete valuation ring with algebraically closed residue field $\kappa$ containing $k$. Let $m,m'\in\cM^{\mathrm{ani}}(R)$ be two elements and $m_K,m_K'\in\cM^{\mathrm{ani}}(K)$ their base change. Suppose that the following two conditions are satisfied:
\begin{enumerate}
\item $h(m)=h(m')$;
\item there exists an isomorphism $\iota_K:m_K\to m_K'$.
\end{enumerate}
Then there exists a unique isomorphism $\iota:m\to m'$ extending $\iota_K$.
\end{prop}
\begin{proof}
We follow the argument in \cite[\S9]{ChLau10}. Let $m=(\cE,\phi)$ and $m'=(\cE',\phi')$. 
Consider the local ring $B$ of the generic point of the special fiber of $X_R$. Then $B$ is a discrete valuation ring whose residue field is the function field $\kappa(X)$ of $X_\kappa$ and whose fraction field is the function field $F$ of $X_R$.\par  
 By \S9.2 of \emph{loc. cit.}, it suffices to extend $\iota_K$ to an isomorphism of $G$-torsors $\iota:\cE\to\cE'$ over $\spec B$. As in \S9.3 of \emph{loc. cit.}, it suffices to show that for some finie extension $K'$, the base change $\iota_{K'}$ of $\iota_K$ extends to an isomorphism between $\cE,\cE'$ over $\spec B'$. Here $B'$ is the integral closure of $B$ in the function field $F'$ of $X_{R'}$ where $R'$ is the integral closure of $R$ in $K'$.\par 
 To achieve this, after taking a finite extension $K'/K$ one can assume that $\cE,\cE'$ are trivial over $\spec B$ (since by \cite[Theorem 2]{DS}, they will be trivial in a Zariski open neighbourhood of the generic point of the special fibre of $X_{R}$ after a finite extension of $K$). Moreover, as in \cite[Lemme 9.3.1]{ChLau10}, one can choose trivialization of $\cE$ and $\cE'$ over $\spec B$ such that they map the ``Higgs fields" $\phi$ and $\phi'$ to some element $\ga\in \vin_{G}^{\mathrm{rs}}(B)$. Under these trivializations, the isomorphism $\iota_K$ is identified with an element $g\in G(F)$ such that $g^{-1}\ga g=\ga$. In other words, $g\in G_\ga(F)$. Since $m,m'$ lies in the anisotropic open substack and $\ga\in \vin_{G}^{\mathrm{rs}}(B)$,  $G_\ga$ is an anisotropic torus over $\spec B$ and hence $G_\ga(B)=G_\ga(F)$. Thus in particular, $g\in G(B)$ and the isomorphism $\iota_K$ extends.
\end{proof}
\begin{thm}\label{thm:ani-proper}
The morphism $h_\cL^{\mathrm{ani}}:\cM^{\mathrm{ani}}_\cL\to\cA^{\mathrm{ani}}_\cL$ is proper.  
\end{thm}
\begin{proof}
This follows from what have been proved in this section and the valuative criterion of properness for algebraic stacks. 
\end{proof}

\subsection{Singularities of restricted Hitchin-Frenkel-Ng\^o moduli stack}\label{sec:singularity-HFN}
Later when proving equi-dimansionality of Kottwitz-Viehmann varieties, we will need the transversality theorem of Bouthier in \cite{Bou17}, where it was shown that the singularities of certain open substack of restricted Hitchin-Frenkel-Ng\^o moduli stack are the same as some closed Schubert varieties in the affine Grassmanian. The method of Bouthier was later simplified by Yun in \cite{Yun}. In \cite{Bou17} and \cite{Yun} it is assumed that the group is simply-connected but the argument works without this assumption. For the reader's convenience we review this result following \cite{Yun}. 
\subsubsection{}
Fix a $X_*(\tad)_+$-valued divisor $\la=\sum_{i=1}^m\la_i x_i$ on the curve $X$. Then $\la$ defines a $\tad$-torsor $\cL_\la$. We assume that $\cL_\la$ can be lifted to a $\tsc$-torsor $\cL$. Then $\la$ can be identified with a closed point of $\cB_\cL=H^0(X,A_{\gsc}^\cL)$. Let $\cM_{\le\la}:=\alpha_\cL^{-1}(\la)$ be the corresponding restricted Hitchin-Frenkel-Ng\^o moduli stack. Let $\cA_{\le\la}:=\beta_\cL^{-1}(\la)$ and $h_{\le\la}: \cM_{\le\la}\to\cA_{\le\la}$ be the restricted Hitchin-Frenkel-Ng\^o fibration. Let $\cM_\la:=\cM_{\le\la}\cap\cM_\cL^0$ be the open substack where the Higgs-Vinberg field lands in $[\vin_{\gsc}^0/\tsc\times\mathrm{Ad}(G)]$.\par 
Assume moreover that $\cL$ admits a $c$-th root $\cL'$ where $c=|Z(G_{\mathrm{der}})|$. Then by the discussion in \S~\ref{sec:global-Steinberg}, 
there exists global Steinberg section $\epsilon_{\cL'}^w:\cA_{\le\la}\to\cM_{\le\la}^{\mathrm{reg}}$ for each choice of Coxeter element $w\in\mathrm{Cox}(W,S)$.
\subsubsection{}
For each $a\in\cA_{\le\la}^\heartsuit$, we write the associated discriminant divisor as
\[\Delta(a)=\Delta(a)_{\mathrm{sing}}+\Delta(a)_{\mathrm{triv}}\]
where $\Delta(a)_{\mathrm{triv}}$ is multiplicity free and the multiplicity of $\Delta(a)_{\mathrm{sing}}$ at each point is at least 2. 
\begin{defn}\label{def:transversal-open}
Let $S\subset\mathrm{Supp}(\la)$ be a nonempty subset. The \emph{transversal subset} $\cA_{\le\la}^\flat\subset\cA_{\le\la}$ consists of $a\in \cA_{\le\la}^\heartsuit$ satisfying the following two conditions
\begin{itemize}
\item $\mathrm{Supp}(\Delta(a))\cap\mathrm{Supp}(\la)\subset S$
\item  For each $1\le i\le r$, 
\[2g-2+m_0(\deg\Delta(a)_\mathrm{sing}+|S|)+\sum_{s\in S}b(\la_s)<\deg\omega_i(\cL)\]
where $m_0$ is the positive integer defined in the paragraph before Proposition~\ref{prop:chi-lifting} and $b(\la_s)$ is the non-negative integer in Lemma~\ref{lem:vinG-approx}.
\end{itemize}
\end{defn}
We call $\cM_{\le\la}^\flat:=h_{\le\la}^{-1}(\cA_{\le\la}^\flat)$ the \emph{transversal open substack} and denote $\cM_\la^\flat:=\cM_{\le\la}^\flat\cap\cM_\la$.

\subsubsection{Local evaluation map}
For each $s\in S$, the arc space $L^+G$ acts by left multiplication on $\mathrm{Gr}_{\le\la_s}:=\mathrm{Gr}^{\gad}_{\le\la_s}$ and the action factors through $L^+\gad$. We let $N$ be a positive integer such that for all $s\in S$, the action of $L^+G$ on $\mathrm{Gr}_{\le\la_s}$ factors through the $N$-th jet space $L_N^+ G$. Then the product group $L^+_{NS}G:=\prod_{s\in S}L^+_NG$ acts naturally on $\prod\limits_{s\in S}\mathrm{Gr}_{\le\la_s}$ and we define the local evaluation map
\[\mathrm{ev}_{NS}:\cM_{\le\la}\to [L_{NS}^+ G\backslash\prod_{s\in S}\mathrm{Gr}_{\le\la_s}]\]
by choosing trivialisations of $G$-torsors on the $N$-th infinitesimal neighbourhood of points $s\in S$. Let $\mathrm{ev}_{NS}^\flat$ be the restriction of $\mathrm{ev}_{NS}$ to $\cM_{\le\la}^\flat$. From the first condition in Definition~\ref{def:transversal-open}, we see that for any $(\cE,\varphi)\in\cM_{\le\la}^\flat$ the restriction of the Higgs-Vinberg field $\varphi$ to points in $\mathrm{Supp}(\la)\setminus S$ lands in the open substack $[\vin_{\gsc}^\mathrm{rs}/\tsc\times\mathrm{Ad}(G)]$, which is contained in $[\vin_{\gsc}^0/\tsc\times\mathrm{Ad}(G)]$ by Corollary~\ref{cor:vin-reg-in-v0}. Hence the inverse image of the open strata $[L_{NS}^+G\backslash\prod_{s\in S}\mathrm{Gr}_{\la_s}]$ under $\mathrm{ev}_{NS}^\flat$ is precisely $\cM_\la^\flat=\cM_{\le\la}\cap\cM_\la$.

\begin{thm}[\cite{Bou17},\cite{Yun}]\label{thm:transversality}
The morphism 
\[\mathrm{ev}_{NS}^\flat:\cM_{\le\la}^\flat\to [L_{NS}^+ G\backslash\prod_{s\in S}\mathrm{Gr}_{\le\la_s}]\] 
is smooth.
\end{thm}
The proof proceeds in several steps which occupy the rest of this section. 
\subsubsection{}
Let $\cM_{\le\la, NS}^\flat$ be the stack classifying triples $(\cE,\varphi,\tau_{NS})$ where $(\cE,\varphi)$ is a point in $\cM_\la^\flat$ and $\tau_{NS}$ is a trivialisation of $\cE$ on the $N$-th infinitesimal neibourhoods of $s$ for all $s\in S$. Then $\cM_{\le\la,NS}$ is a $L^+_{NS}G$-torsor over $\cM_{\le\la}$ and
to prove the smoothness of $\mathrm{ev}_{NS}^\flat$, it suffices to prove the smoothness of its base change
\[\widetilde{\mathrm{ev}}_{NS}^\flat:\cM_{\le\la,NS}^\flat\to\prod_{s\in S}\mathrm{Gr}_{\le\la_s}.\]
Notice that the source and target of $\tilde{\mathrm{ev}}_{NS}^\flat$ are locally of finite type. Hence it suffices to show that $\widetilde{\mathrm{ev}}_{NS}^\flat$ is formally smooth. In other words, we need to check the infinitesimal lifting property.
\subsubsection{}
Let $R$ be an artin local $k$-algebra with maximal ideal $\fm$ and let $I\subset R$ be an ideal with $I\cdot\fm=0$. Denote $\bar{R}:=R/I$. 
Consider a triple $(\bar\cE,\bar\varphi,\bar\tau_{NS})\in\cM_{\la,NS}^\flat(\bar R)$ whose image under the map $\widetilde{\mathrm{ev}}_{NS}^\flat$ is $([\bar\ga_s])_{s\in S}\in\prod_{s\in S}\mathrm{Gr}_{\le\la_s}(\bar{R})$. Let
$([\ga_s])_{s\in S}\in\prod_{s\in S}\mathrm{Gr}_{\le\la_s}(R)$ be a lifting of $([\bar{\ga}_s])_{s\in S}$. We need to find a lifting of $(\bar\cE,\bar\varphi,\bar\tau_{NS})$ to a point in $\cM_{\le\la,NS}(R)$ whose image under $\widetilde{\mathrm{ev}}_{NS}^\flat$ coincides with $([\ga_s])_{s\in S}$.\par 
Extend $\bar\tau_{NS}$ to a trivialisation $\bar\tau_{\infty S}$ of $\bar\cE$ on $\prod_{s\in S}D_s\hat\otimes\bar{R}$ where $D_s$ denotes the formal neighbourhood of $s$ in $X$. It suffices to construct a Higgs-Vinberg pair $(\cE,\varphi)\in\cM_{\le\la}^\flat(R)$ lifting $(\bar\cE,\bar\varphi)\in\cM_{\le\la}^\flat(\bar R)$ and a trivialisation $\tau_{\infty S}$ of $\cE$ on $\prod_{s\in S}D_s\hat\otimes R$ lifting $\bar\tau_{\infty S}$.
\subsubsection{}
Let $\bar{a}:=h_{\le\la}(\bar{\cE},\bar{\varphi})\in\cA^\flat_{\le\la}(\bar{R})$ and $a_0\in\cA^\flat_{\le\la}(k)$ its reduction mod $\fm$. We have the discriminant divisor $\Delta(a_0)$ associated to $a_0$. For each $v\in X$, let $d_v$ be the multiplicity of $\Delta(a_0)$ at $v$.
Let $S':=S\cup\mathrm{Supp}(\Delta(a_0)_{\mathrm{sing}})$ and $T:=S'\setminus S$ so that $S'=S\sqcup T$. Note that the first condition in Definition~\ref{def:transversal-open} implies that $T\cap\mathrm{Supp}(\la)=\varnothing$.\par 
For each $s\in S$, under the trivialisation $\tau_{\infty S}$, the Taylor expansion of $\bar{\varphi}$ at $s$ corresponds to an element $\bar\ga_s\in L^+\vin^\la(\bar{R})$ whose image in $\mathrm{Gr}_{\le\la_s}(\bar{R})$ is $[\bar\ga_s]$. Since the morphism $L^+\vin^\la\to\mathrm{Gr}_{\le\la_s}$ is formally smooth, there exists a lifting $\ga_s\in L^+\vin^\la(R)$ of $\bar{\ga_s}$ whose image in $\mathrm{Gr}_{\le\la_s}$ equals to the $[\ga_s]$ given above. Let $a_s:=\chi_+(\ga_s)\in\fC_+(\hat{\cO}_s\hat\otimes R)$. Then $\bar{a}_s=\bar{a}\in\fC_+(\hat{\cO}_s\hat\otimes \bar{R})$.\par 
For each $t\in T=S'\setminus S$, choose a trivialisation $\bar\tau_{\infty t}$ of $\bar{\cE}$ on $D_t\hat\otimes\bar R$, under which the Taylor expansion of $\varphi$ at $t$ corresponds to an element $\bar\ga_t\in L^+\gsc_+(\bar{R})$. We lift $\bar\ga_t$ arbitrarily to an element $\ga_t\in \gsc_+(\hat{\cO}_t\hat\otimes R)$ and let $a_t:=\chi_+(\ga_t)\in\fC_+(\hat{\cO}_t\hat\otimes R)$. Then in particular $\bar{a}_t=\bar{a}\in\fC_+(\hat{\cO}_t\hat\otimes\bar{R})$.
\subsubsection{}
Consider the local evaluation map of the base space
\[\cA_{\le\la}\to\bigoplus_{s\in S}\fC_+^{\la_s}(\cO_s/\varpi_s^{m_0d_s+b(\la_s)})\times\bigoplus_{t\in T}\fC(\cO_t/\varpi_t^{m_0d_t})\]

By the inequality in Definition~\ref{def:transversal-open}, this is a surjective linear map between $k$-vector spaces, hence it is smooth when viewed as morphisms between affine $k$-schemes. So there exists $a\in\cA_{\le\la}(R)$ lifting $\bar{a}\in\cA_{\le\la}(\bar{R})$ such that $a\equiv a_v\mod\varpi_v^{n_v}$ for all $v\in S'$.\par
Then for each $s\in S$, we have $a\equiv\chi_+(\ga_s)\mod\varpi_s^{m_0d_s+b(\la_s)}$. By Proposition~\ref{prop:chi-lifting} there exists $\theta_s\in \ving(R[[\varpi_s]])$ such that $\chi_+(\theta_s)=a$ and $\theta_s\equiv\ga_s\mod\varpi_s^{b(\la_s)}$. By Lemma~\ref{lem:vinG-approx}, this implies that the image of $\theta_s$ in $\mathrm{Gr}_{\la_s}(R)$ coincides with $[\ga_s]$.\par
For each $t\in T$, by Proposition~\ref{prop:chi-lifting} again, there exists $\theta_t\in\gsc(\hat{\cO}_t\hat\otimes R)$ such that $\chi_+(\theta_t)=a$.
\subsubsection{}
For each $v\in\mathrm{Supp}(\la)\setminus S$, the restriction of the Higgs-Vinberg field $\bar{\varphi}$ to $v$ lands in $[\ving^{\mathrm{rs}}/T\times\mathrm{Ad}(G)]$ by the first condition in Definition~\ref{def:transversal-open}. For each point $v$ in the complement of $\mathrm{Supp}(\la)\cup S'=\mathrm{Supp}(\la)\sqcup T$, the restriction of $\bar{\varphi}$ to $v$ lands in $[G_+^{\mathrm{sc,reg}}/T\times\mathrm{Ad}(G)]$ by Corollary~\ref{cor:0-dim-reg}.\par   
Therefore the restriction of $(\bar\cE,\bar\varphi)$ to $(X-S')\otimes_k\bar{R}$ lands in the stack 
\[[(\ving^{\mathrm{rs}}\cup G_+^{\mathrm{sc,reg}})/T\times\mathrm{Ad}(G)]\]
 which is a $\bB\cJ$ gerbe neutralized by a global Steinberg section $\epsilon_{\cL'}^w$. By the same reasoning as in Proposition~\ref{prop:P-smooth}, there exists a Higgs-Vinberg pairs $(\cE',\varphi')$ over $(X-S')\otimes_k R$ together with trivialisations $\tau_{\infty v}^\bullet$ of $\cE'$ over the formal punctured disc at each $v\in S'$ that
lifts $(\bar\cE,\bar\varphi)|_{(X-S')\otimes_k\bar{R}}$ and the restrictions of the trivialisations $\tau_{\infty v}$ to the punctured disc and moreover $\chi_+(\cE',\varphi')=a\in\Hom(X-S',\fC_+^\cL)$. 
\subsubsection{}
Finally we construct the desired lifting $(\cE,\varphi,\tau_{\infty S})$ by Beauville-Laszlo gluing lemma.  
For each $v\in S'$, restricting $\varphi'$ to the formal punctured disc $D_v^\bullet\hat\otimes R$ and using the trivialisation $\tau_{\infty v}^\bullet$, we obtain an element $\theta_v'\in\gsc_+(R((\varpi_v)))$ with $\chi_+(\theta_v')=a$. Recall that we have constructed elements $\theta_v\in L^+\vin^{\la_v}(R)\subset\gsc_+(R((\varpi_v)))$ with $\chi_+(\theta_v)=a$. Since $a\in\fC_+^{\mathrm{rs}}(R((\varpi_v)))$, the transporter $\mathrm{Isom}(\theta_v,\theta_v')$ 
from $\theta_v$ to $\theta_v'$ is a torsor under the smooth centralizer $\cJ_a|_{D_v^\bullet\hat\otimes R}$. After reduction mod $I$, we know that $\bar\theta_v$ and $\bar\theta_v'$ comes from a globally defined Higgs-Vinberg pair $(\cE',\varphi')$. In other words, $\mathrm{Isom}(\theta_v,\theta_v')$ has a $\bar{R}$-point. By smoothness, this $\bar{R}$-point lifts to an $R$-point of $\mathrm{Isom}(\theta_v,\theta_v')$. Consequently by Beauville-Laszlo lemma, the Higgs-Vinberg pairs $(\cE',\varphi')$ over $(X-S')\otimes_k R$ with the trivialisations $\tau_{\infty v}^\bullet$ over the formal punctured dics $D_v^\bullet\hat\otimes_kR$ can be glued with the Higgs-Vinberg pairs $(\cE_0,\theta_v)$ (where $\cE_0$ is the trivial $G$-torsor) on the formal discs $D_v\hat\otimes R$ to get a Higgs-Vinberg pair $(\cE,\varphi)\in\cM_{\le\la}^\flat(R)$. By construction it comes with tautalogical trivialisations $\tau_{\infty S}$ on $\sqcup_{s\in S}D_s\hat\otimes R$ lifting $(\bar\cE,\bar\varphi,\bar\tau_{\infty S})$ and its image under the evaluation map $\widetilde{\mathrm{ev}}_{NS}^\flat$ is the $R$-point $([\theta_s]=[\ga_s])_{s\in S}$ of $\prod_{s\in S}\mathrm{Gr}_{\le\la_s}$.
\begin{cor}\label{cor:CM}
The stack $\cM_{\le\la}^\flat$ is Cohen-MaCaulay and its open substack $\cM_\la^\flat$ is smooth.
\end{cor}
\begin{proof}
This follows from Theorem~\ref{thm:transversality} and the fact that $\mathrm{Gr}_{\le\la_s}$ is Cohen-MaCaulay and $\mathrm{Gr}_{\la_s}$ is smooth for all $s\in S$.
\end{proof}

\section{From global to local}\label{chapter:global-to-local}
In this section we finish the proof of Theorem~\ref{thm:dim-formula}.\par 

Let $\la\in X_*(T)_+$ and $\ga\in G(F)^{\mathrm{rs}}$. Suppose that $\kappa_G(\ga)=p_G(\la)$ and $\nu_\ga\le_\bQ\la$ so that the generalized affine Springer fibres $X_\ga^\la$ and $X_\ga^{\le\la}$ are nonempty. Let $a:=\chi_+(\ga_\la)\in\fC_+(\cO)\cap\fC_{\gsc_+}^{\mathrm{rs}}(F)$ where $\ga_\la\in G_+^\mathrm{sc,rs}(F)$ is defined in Lemma~\ref{lem:ga-la-lem}. Then we have isomorphisms
\[X_\ga^{\le\la}\cong\Sp_a,\quad X_\ga^\la\cong\Sp_a^0.\]
Moreover, the local Picard $P_a$ acts on $\Sp_a$ and $\Sp_a^{\mathrm{reg}}$ is the union of open orbits.

\subsection{Local constancy of Kottwitz-Viehmann varieties}
This subsection is devoted to the proof of the following:
\begin{thm}\label{thm:local-constancy}
There exists an integer $N$ such that for all $a'\in\fC_+(\cO_x)\cap\fC_{G_+}^{\mathrm{rs}}(F_x)$ with $a'\equiv a\mod\varpi^N$, $\Sp_{a'}$ equipped with the action of $P_{a'}$ is isomorphic to $\Sp_a$ equipped with the action of $P_a$.
\end{thm}
First we make some standard reductions. Notice that for any $a'\in\fC_+(\cO_x)\cap\fC_{G_+}^{\mathrm{rs}}(F_x)$,
$\Sp_{G,a'}$ is a union of certain connected components of $\Sp_{\gad,a'}$, ther latter of which is isomorphic to the quotient of $\Sp_{\gsc_+,a'}$ by the coweight lattice $X_*(\tsc)$ of the central torus of $\gsc_+$. Hence we may assume that $G=\gsc_+$ and for simplicity omit $G$ in the notation.\par 

Fix a Coxeter element $w\in\mathrm{Cox}(W,S)$, cf. \ref{coxeter-definition}. Let $\ga_0:=\epsilon_+^w(a)$ (resp. $\ga_0':=\epsilon_+^w(a')$) be the extended Steinberg sections for $a$ (resp. $a'$). Then we have canonical isomorphism between groups schemes over $\spec\cO$:
\[J_a\cong I_{\ga_0},\quad J_{a'}\cong I_{\ga_0'}.\]
\begin{lem}
For any $g\in G(F)$, we have $\mathrm{Ad}(g)^{-1}(\ga_0)\in \vin_{\gsc}(\cO)$ if and only if 
\[\mathrm{Ad}(g)^{-1}(\ga_0 I_{\ga_0}(\cO))\subset \vin_{\gsc}(\cO).\]
\end{lem}
\begin{proof}
Since $\ga_0\in \ga_0 I_{\ga_0}(\cO)$, the condition is sufficient. Now assume that $\ga:=\mathrm{Ad}(g)^{-1}(\ga_0)\in \vin_{\gsc}(\cO)$. Then the centralizer $I_\ga$ is a group scheme over $\spec\cO$.
By Lemma~\ref{lem:reg-centralizer-Vin}, the isomorphism of $F$ groups
\[\mathrm{Ad}(g)^{-1}: J_{a,F}=I_{\ga_0,F}\to I_{\ga,F}\]
extends to $\spec\cO$. Thus we have
\[\mathrm{Ad}(g)^{-1}(I_{\ga_0}(\cO))\subset I_\ga(\cO)\subset G(\cO)\]
from which we obtain
\[\mathrm{Ad}(g)^{-1}(\ga_0 I_{\ga_0}(\cO))=\ga\mathrm{Ad}(g)^{-1}(I_{\ga_0}(\cO))\subset\vin_{\gsc}(\cO).\]
\end{proof}

\begin{lem}
Let $a,a'\in\fC_+(\cO)\cap\fC_{\gsc_+}^{\mathrm{rs}}(F)$ with $a\equiv a'\mod\varpi^N$. Suppose that there exists a $W$-equivariant isomorphism between the cameral covers $\widetilde{X}_a$ and $\widetilde{X}_{a'}$ lifting the identity modulo $\varpi^N$. Let $\ga_0:=\epsilon_+^w(a)$ and $\ga_0'=\epsilon_+^w(a')$. Then there exists $g\in G(\cO)$ such that
\[\mathrm{Ad}(g)^{-1}(\ga_0 I_{\ga_0}(\cO))=\ga_0' I_{\ga_0'}(\cO)\] 
\end{lem}
\begin{proof}
We follow the argument of \cite[Lemme 3.5.4]{Ngo10}. 
Let $\widetilde{X}_a=\spec R_a$ and $\widetilde{X}_{a'}=\spec R_{a'}$ where $R_a$, $R_{a'}$ are finite flat $\cO$-algebras. Let $F_a:=R_a\otimes_\cO F$ (resp. $F_{a'}:=F_{a'}\otimes_\cO F$) and $R_a^\flat$ (resp. $R_{a'}^\flat$) be the normalization of $R_a$ (resp. $R_{a'}$) in $F_a$ (resp. $F_{a'}$). \par 
By assumption, we have $R_a/\varpi^N=R_{a'}/\varpi^N$ and there exists a $W$-equivariant $\cO$-isomorphism 
\[\iota: R_a\xrightarrow{\sim} R_{a'}\]
that lifts the identity modulo $\varpi^N$.

By Proposition~\ref{Galois-description-J-prop}, the isomorphism $\iota: R_a\cong R_{a'}$ induces an isomorphism $\iota_I: I_{\ga_0}\to I_{\ga_0'}$ between group schemes over $\spec\cO$. Since $\ga_0\in I_{\ga_0}(F)$, we have $\iota_I(\ga_0)\in I_{\ga_0'}(F)$. We can choose $h\in G(R_a^\flat)$ and $h'\in G(R_{a'}^\flat)$ such that on $F$-points, the map $\iota_I$ is given by the following composition
\begin{equation}\label{eq:I-gamma-0-isom}
I_{\ga_0}(F)\xrightarrow{\sim} T(F_a)^W\xrightarrow{\iota} T(F_{a'})^W\xrightarrow{\sim} I_{\ga_0'}(F).
\end{equation}
where the first map is $\mathrm{Ad}(h)$ the and third map is $\mathrm{Ad}(h')^{-1}$. In other words, $\iota_I=\mathrm{Ad}(h'^{-1}\iota(h))$ on $F$-points. In particular, we have
\[\chi_+(\iota_I(\ga_0))=\chi_+(\ga_0)=a.\]
The assumption that $\iota$ is identity modulo $\varpi^N$ implies that $\mathrm{Ad}(h'^{-1}\iota(h))\equiv\mathrm{Id}\mod\varpi^N$. Thus we get
\[\iota (I_{\ga_0}(F)\cap\vin_{\gsc}(\cO))\subset I_{\ga_0'}(F)\cap\vin_{\gsc}(\cO).\]
In particular, we have $\iota(\ga_0)\in I_{\ga_0'}\cap\vin_{\gsc}(\cO)$ and moreover 
\[\iota_I(\ga_0)=\ga_0=\ga_0'\text{ in }\vin_{\gsc}^w(\cO/\varpi^N).\]   
Since the map 
\[G\times\vin_{G}^w\to\vin_{G}^w\times_{\fC_+}\vin_{G}^w\]
is smooth and surjective, there exists $g\in G(\cO)$ with $g\equiv 1\mod\varpi^N$ such that $\mathrm{Ad}(g)^{-1}(\ga_0)=\iota_I(\ga_0)$. Therefore 
\[\mathrm{Ad}(g)^{-1}(I_{\ga_0})=I_{\iota(\ga_0)}=I_{\ga_0'}.\]
Finally by Lemma~\ref{lem:vinG-approx}, we have $(\ga_0')^{-1}\iota_I(\ga_0)\in G(\cO)\cap I_{\ga_0'}(F)=I_{\ga'}(\cO)$ which implies that $\iota_I(\ga_0)\in\ga_0' I_{\ga_0'}(\cO)$ and hence we are done.
\end{proof}

\subsection{Dimension of Kottwitz-Viehmann varieties}
By Theorem~\ref{dim-reg-locus-thm}, the dimension formula for $X_\ga^\la\cong\Sp_a$ is reduced to the following statement which we prove in this subsection:
\begin{thm}\label{thm:dim-equal}
$\dim\Sp_a=\dim P_a$.
\end{thm}
If $C\subset G$ is the maximal torus in the center of $G$, then $\Sp_{G/C,a}\cong\Sp_{G,a}/X_*(C)$ and similar isomorphism holds for the local Picard $P_a$. Thus we may assume that $G$ is semisimple.\par 
Let $X$ be a projective smooth curve over $k$ and $x\in X$ a closed point. Let $\cO_x$ be the completed local ring at $x$ and $F_x$ its fraction field. Choose a uniformiser $\varpi_x$ at $x$ so that we have $\cO_x=k[[\varpi_x]]$ and $F_x=k((\varpi_x))$. Also we let $X'=X-\{x\}$ be the open curve.\par 

We view $a\in\fC_+(\cO_x)$ as a power series in $\varpi_x$ with coefficients in $\fC_+$. Form the Cartesian diagram
\[\xymatrix{
X_{a}\ar[r]\ar[d]_{\pi_a} & \overline{T_+}\ar[d]^{\pi}\\
\spec\cO\ar[r]^a & \fC_+
}\]
where $X_{a}=\spec R_a$ for a finite flat $\cO$ algebra $R_a$. Moreover, $F_a=R_a\otimes_\cO F$ is a product of finite tamely ramified extension of $F$ of degree $e$ by our assumption that $\mathrm{char}(k)$ is coprime to the order of Weyl group. Then $a(\varpi_x^e)\in\fC_+(\cO_x)\cap\fC_{\gsc_+}^{\mathrm{rs}}(F_x)$ will be a split conjugacy class. \par 
For each $s\in k$ we define

\begin{equation}\label{eq:local-perturbation}
a_s:= a(s\varpi_x+(1-s)\varpi_x^e)\in\fC_+(\cO)\cap\fC_{G_+}(F)^{\mathrm{rs}}.
\end{equation}

Then $a_1=a$ and $a_0=a(\varpi_x^e)$. For each $s\ne0$, $\Sp_{a_s}$ is isomorphic to $\Sp_{a}$ since $a_s$ is obtained from $a=a_1$ by changing uniformizer.\par 
\subsubsection{}
Let $N>0$ be a positive integer such that both $\Sp_{a}$ and $\Sp_{a_0}$ only depends on $a$ (resp. $a_0$) modulo $\varpi_x^N$.  Then for all $s\in k$, $\Sp_{a_s}$ only depends on $a_s$ modulo $\varpi_x^N$.  \par 
Now we choose a $\tsc$-torsor $\cL$ on $X$ trivialized on the formal neighbourhood of $x$ such that
\begin{enumerate}
\item There exists a $\tsc$-torsor $\cL'$ and an isomorhpism $(\cL')^{\otimes c}\cong\cL$;
\item For all $y\in X'=X-x$, choosing a trivialisation of $\cL$ on a formal neighbourhood of $y$, the local evaluation map
\begin{equation}\label{eq:ev-C+}
\cA_\cL=H^0(X,\fC_+^\cL)\to\fC_+(\cO_x/\varpi_x^N)\times\fC_+(\cO_y/\varpi_y^2)
\end{equation}
is surjective. 
\end{enumerate}
By Riemann-Roch, condition 2 is satisfied if for all $1\le i\le r$ we have $\deg(\alpha_i(\cL))\ge 2g+N$ and $\deg(\omega_i(\cL))>2g+N$.\par
Recall that for each $a_+\in\cA_\cL^\heartsuit$, we associate an $X_*(\tad)_+$-valued divisor $\la_{a_+}$ on $X$ as in \S~\ref{sec:lambda-divisor}.
\begin{lem}\label{lem:constructible-A-Sigma}
Let $\Sigma\subset X$ be a finite subset. The subset $\cA_\cL^\Sigma\subset\cA_\cL^\heartsuit$ consisting of $a_+\in\cA_\cL^\heartsuit$ such that 
\[\mathrm{Supp}(\la_{a_+})\cap\mathrm{Supp}(\Delta_{a_+})\subset\Sigma\]
is constructible.
\end{lem}
\begin{proof}
For each $1\le i\le r$, consider the closed subscheme $\cD_i\subset \cA_\cL^\heartsuit \times X$ whose fibre over $a_+\in\cA_\cL^\heartsuit$ is the effective divisor $D(b_i)$ where $b_i$ is the $i$-th coordinate of $\beta_\cL(a_+)$ as above. Similarly, we have the closed subscheme $\Delta\subset \cA_\cL^\heartsuit\times X$ whose fibre over $a_+$ is the discriminant divisor $\Delta_{a_+}$. Let $\cD_i^\Sigma=\cD_i\cap (\cA_\cL^\heartsuit\times (X-\Sigma))$ and $\Delta^\Sigma:=\Delta\cap (\cA_\cL^\heartsuit\times (X-\Sigma))$. Then $\cD_i^\Sigma\cap\Delta^\Sigma$ is a locally closed subset of $\cA_\cL^\heartsuit\times X$. By construction $\cA_\cL^\Sigma$ is the image of $\bigcup\limits_{1\le i\le r}(\cD_i^\Sigma\cap\Delta^\Sigma)$ in $\cA_\cL^\heartsuit$, hence constructible. 
\end{proof}
\subsubsection{}
The one-parameter family \eqref{eq:local-perturbation} defines a curve $C$ in $\fC_+(\cO_x/\varpi_x^N)$. Let $L_C\subset\cA_\cL$ be the closed subset defined as the inverse image of $C$ under the map \eqref{eq:ev-C+}. For all $s\in k$, let $L_{a_s}\subset\cA_\cL$ be the inverse image of $a_s$ under the map \eqref{eq:ev-C+}. Since $a_s\in\fC_{\gsc}^{\mathrm{rs}}(F)$ for all $s\in k$, we have $L_C\subset\cA_\cL^\heartsuit$.
\begin{defn}\label{def:Z-C}
Let $Z_C\subset L_C$ be the subset consisting of $a_+\in L_C$ with $b=\beta_\cL(a_+)$ such that
\begin{itemize}
\item $a_+\in\cA_\cL^{\mathrm{ani}}$;
\item $\mathrm{Supp}(\la_{a_+})\cap\mathrm{Supp}(\Delta_{a_+})\subset\{x\}$;
\item $a_+(X')$ intersects the discriminant divisor $\fD_+^\cL$ transversally, where $X'=X-\{x\}$.
\end{itemize}
\end{defn}
\begin{lem}\label{lem:Z-C}
$Z_C$ is a constructible subset of $L_C$ that is fibrewise dense with respect to the projection $L_C\to C$. In particular, there exists a fibrewise dense open subset $U_C$ of $L_C$ such that $U_C\subset Z_C$.
\end{lem}
\begin{proof}
First we show that $Z_C$ is constructible. The first condition in Definition~\ref{def:Z-C} defines an open subset of $L_C$. By Lemma~\ref{lem:constructible-A-Sigma}, the set $L_C^x:=L_C\cap\cA_\cL^{x}$ determined by the second condition in Definition~\ref{def:Z-C} is a constructible subset of $L_C$.\par 
Let $U\subset X'\times L_C$ be the open subset whose fibre over $a_+\in L_C$ is the open curve $X'-\mathrm{Supp}(\la_{a_+})$. The local evaluation maps define a morphism
\[U\to \bT\fC_+^\cL\]
where $\bT\fC_+^\cL$ is the relative tangent bundle of $\fC_+^\cL$ over $X$. Let $U_1$ be the inverse image of 
\[\bT\fD_+^{\cL,\mathrm{sm}}\cup \bT\fC_+^\cL\times_{\fC_+^\cL}\fD_+^{\cL,\mathrm{sing}}.\]
Then the image of $U_1$ in $L_C$ is a constructible subset that satisfies the third condition in Definition~\ref{def:Z-C}. Hence $Z_C$ is a constructible subset of $L_C$. \par
Next we show that $Z_C$ is fibrewise dense with respect to the map $L_C\to C$. We fix a point $a_s\in C$.\par 
For any closed point $y\in X'$, the map
\[L_{a_s}\to \bT\fC_{+,y}^\cL=\fC_+^\cL\otimes_{\cO_y}\cO_y/\fm_y^2\]
is surjective by our choice of $\cL$.\par  
Let $X'':=X'\setminus\mathrm{Supp}(\la_b)$. By the same argument as in \cite[Lemme 4.7.2]{Ngo10}, we know that the subset $Z\subset L_{a_s}$ consisting of $a_+\in L_{a_s}$ such that $a_+(X'')$ intersects $\fD_+^\cL$ transversally is dense in $L_{a_s}$. \par 
For each $y\in\mathrm{Supp}(\la_b)-\{x\}$, since the map
$\mathrm{ev}_y: L_{a_s} \to \fC_{+,y}^\cL$
is surjective, the subset $\Sigma_y:=\mathrm{ev}_y^{-1}(\fD_+^\cL)\subset L_{a_s}$ has codimension 1.\par
Finally, since $L_{a_s}$ has codimension $2rN$ in $\cA_\cL^\heartsuit$ and the complement of $\cA_\cL^{\mathrm{ani}}$ in $\cA_\cL^\heartsuit$ has codimension strictly larger than $2rN$, we see that
\[Z_{a_s}=(Z-\bigcup_{y\in\mathrm{Supp(\la_b)}}\Sigma_y)\cap\cA_\cL^{\mathrm{ani}}\]
is dense in $L_{a_s}$.
\end{proof}
\subsubsection{}
Thus we can choose a section $\sigma$ of the surjective linear map \eqref{eq:ev-C+} such that $C':=\sigma(C)\cap U_C$ is nonempty and contains the point $\sigma(a_0)$.\par 
By the product formula \ref{thm:product-formula}, we have
\[\dim\cM_{\sigma(a_0)}-\cP_{\sigma(a_0)}=\sum_{v\in\mathrm{Supp}(\Delta_a)\cup\{x\}}(\dim\Sp_{\sigma(a_0)_v}-\dim P_{\sigma(a_0)_v})\]
where $\sigma(a_0)_v$ denotes the image of $\sigma(a_0)$ in $\fC_+(\cO_v)$.\par 
For summands with $v\ne x$, since $\sigma(a_0)\in Z_C$ we have in particular $\la_{\sigma(a_0),v}=0$ and hence by 
Corollary~\ref{cor:dim-central-coweight} $\dim\Sp_{\sigma(a_0)_v}=\dim P_{\sigma(a_0)_v}$. On the other hand, for the term $v=x$, we know that $\sigma(a_0)_x=a_0$ is split and hence by Corollary~\ref{dim-unr-cor} $\dim\Sp_{a_0}=\dim P_{a_0}$. Thus the above equality simplifies to
\[\dim\cM_{\sigma(a_0)}-\dim\cP_{\sigma(a_0)}=0.\]
Since $C'\subset\cA^{\mathrm{ani}}_\cL$, the restriction of the Hitchin-Frenkel-Ng\^o fibration to $C'$ is proper. Hence by upper semicontinuity of fibre dimension we have for 
\[\dim\cM_{\sigma(a_s)}\le\dim\cM_{\sigma(a_0)}=\dim\cP_{\sigma(a_0)}\]
for all $\sigma(a_s)\in C'$ with $s\ne 0$. Since $\cP$ is smooth over $\cA_\cL$ by Proposition~\ref{prop:P-smooth}, we have $\dim\cP_{\sigma(a_s)}=\cP_{\sigma(a_0)}$, which forces 
\[\dim\cM_{\sigma(a_s)}=\dim\cP_{\sigma(a_s)}\]
Apply product formula \ref{thm:product-formula}
again we get
\[0=\dim\cM_{\sigma(a_s)}-\dim\cP_{\sigma(a_s)}=\sum_{v\in\mathrm{Supp}(\Delta_{a_s})\cup\{x\}}(\dim\Sp_{\sigma(a_s),v}-\dim P_{\sigma(a_s),v})
\]
By similar reasoning as above, all terms in the right hand side where $v\ne x$ are zero; at $v=x$ notice that $\sigma(a_s)_x=a_s$
 and then we get
\[\dim\Sp_{a_s}-\dim P_{a_s}=0.\]
Since $s\ne0$, we have $\Sp_{a_s}\cong\Sp_a$ and hence
\[\dim\Sp_a=\dim P_a\]
This finishes the proof of Theorem~\ref{thm:dim-equal} and hence
the dimension formula in
Theorem~\ref{thm:dim-formula}.
\subsection{Equidimensionality}
To finish the proof of Theorem~\ref{thm:dim-formula} it remains to show the equi-dimensionality statement. Again our argument is of global nature, this time using a restricted Hitchin-Frenkel-Ng\^o moduli stack instead of the whole moduli stack. As in the previous subsection, we may assume that $G$ is semisimple. 
\subsubsection{}\label{sec:choose-la0}
Recall that by Theorem~\ref{thm:local-constancy} there exists a positive integer $N>0$ such that the isomorphism class of $\Sp_a$ equipped with the action of $P_a$ only depends on $a$ modulo $\varpi^N$. Let $X$ be the projective smooth curve as in the previous section. Fix two distinct closed points $x,x_0\in X$. We consider an $X_*(\tad)_+$-valued divisor on $X$ of the form $\la[x]+\la_0[x_0]$, where $\la_0\in X_*(\tad)_+$ is chosen such that the following properties are satisfied:
\begin{itemize}
\item The $\tad$-torsor associated to the divisor $\la[x]+\la_0[x_0]$ lifts to a $\tsc$-torsor $\cL$ and there exists a $\tsc$-torsor $\cL'$ together with an isomorphism $(\cL')^{\otimes c}\cong\cL$.
\item For each $1\le i\le r$, the following three inequalities are satisfied:
\[\langle\omega_i,\la+\la_0\rangle >2g-2+(N+3)r\]
\[\langle\omega_i,\la+\la_0\rangle>2g-2+m_0(d_a+1)+b(\la)\]
\[\langle\omega_i,\la+\la_0\rangle>\max\{Nr,2g-2,rg\}+1+rg\]
\end{itemize}
where $d_a=d_\ga+\langle\rho,\la\rangle$ is the valuation of extended discriminant divisor of $a=\chi_(\ga_\la)\in\fC_+(\cO)$ and the numbers $m_0$, $b(\la)$ are as in Definition~\ref{def:transversal-open}.
\subsubsection{}
Let $h_{\le\la}:\cM_{\le\la}\to\cA_{\le\la}$ be the restricted Hitchin-Frenkel-Ng\^o moduli stack associated to the divisor $\la[x]+\la_0[x_0]$. For simplicity we have omitted $\la_0$ from the notation.\par 
We apply the result of ~\ref{sec:singularity-HFN} to the our current situation. The set $S$ in Definition~\ref{def:transversal-open} is taken to be $\{x\}$ in the current situation. Then we get open subset $\cA_{\le\la}^\flat$ and open substack $\cM_{\le\la}^{\flat,\mathrm{ani}}$.
For some positive integer $n>0$ large enough, there is a local evaluation map
\[\mathrm{ev}_{nx}^\flat:\cM_{\le\la}^\flat\to [L^+_n G\backslash\mathrm{Gr}_{\le\la}]\]
which is smooth by Theorem~\ref{thm:transversality}. Moreover, the inverse image of $[L^+_n G\backslash\mathrm{Gr}_{\la}]$ under $\mathrm{ev}_{nx}^\flat$ is the open substack $\cM^\flat_\la$.

\begin{cor}\label{cor:flat-HFN}
Consider the restriction of the Hitchin-Frenkel-Ng\^o to the transversal anisotropic open substack $h_{\le\la}^{\flat,\mathrm{ani}}:\cM_{\le\la}^{\flat,\mathrm{ani}}\to\cA_{\le\la}^{\flat,\mathrm{ani}}$ is flat. 
\end{cor}
\begin{proof}
By product formula (Theorem~\ref{thm:product-formula}) and Theorem~\ref{thm:dim-equal} we have $\dim\cM_a=\dim\cP_a$ for each $a\in\cA_{\le\la}^{\flat,\mathrm{ani}}$. In particular, the fibre dimension of $h_{\le\la}^{\flat,\mathrm{ani}}$ is constant since $\cP$ is a smooth Deligne-Mumford stack over $\cA_\cL^\heartsuit$. By Corollary~\ref{cor:CM}, the source $\cM_{\le\la}^{\flat,\mathrm{ani}}$ is Cohen-MaCaulay and hence we conclude that the morphism is flat. 
\end{proof}

\begin{lem}
There exists a point $a_+\in\cA_{\le\la}^{\flat,\mathrm{ani}}$ such that 
\begin{itemize}
\item $x_0\notin\mathrm{Supp}(\Delta(a_+))$,
\item $\mathrm{Supp}(\Delta(a_+)_{\mathrm{sing}})\subset\{x\}$, in other words, $a_+(X-x)$ is transversal to the discriminant divisor. 
\item $a_+\equiv a\mod\varpi_x^N$.
\end{itemize}
\end{lem}
\begin{proof}
The proof is similar to Lemma~\ref{lem:Z-C}. Let $L\subset\cA_{\le\la}$ be the linear subspace consisting of $a_+\in\cA_{\le\la}$ such that $a_+\equiv a\mod\varpi_x^N$. Since $a\in\fC_+^\la(\cO)$ is generically regular semisimple, we have $L\subset\cA_{\le\la}^\heartsuit$. \par 
The first inequality in \S~\ref{sec:choose-la0} implies that for any point $y\in X-\{x,x_0\}$, the local evaluation map
\[\cA_{\le\la}\to\fC_+^\la(\cO_x/\varpi_x^N)\times\fC(\cO_y/\varpi_y^2)\times\fC_+^{\la_0}(\cO_{x_0}/\varpi_{x_0})\]
is surjective. By similar argument as in the proof of Lemma~\ref{lem:Z-C}, this implies that there is dense subset $Z\subset L$ consisting of points $a_+\in\cA_{\le\la}^\heartsuit$ such that $a_+(X-x)$ intersects the discriminant divisor transversally and $a_+(x_0)$ does not intersect with the discriminant divisor.\par
Then for each $a_+\in Z$, we have $\Delta(a_+)_{\mathrm{sing}}=d_a[x]$ and hence
the second inequality in \S~\ref{sec:choose-la0} ensures that $Z\subset\cA_{\le\la}^\flat$.\par
Finally since $Z\subset\cA_{\le\la}$ is a subset of codimension $Nr$,
the third inequality in \S~\ref{sec:choose-la0} ensures that $Z$ has nonempty intersection with $\cA_{\le\la}^{\mathrm{ani}}$ by Corollary~\ref{cor:ani-complement}. Then any point $a_+\in Z\cap\cA_{\le\la}^{\mathrm{ani}}$ satisfies the condition we want.
\end{proof}
Choose $a_+\in\cA_{\le\la}^{\flat,\mathrm{ani}}$ as in the Lemma above. Then by Theorem~\ref{thm:product-formula} and Corollary~\ref{cor:0-dim-reg}, there is a homeomorphism of stacks
\[[\cM_{\le\la,a_+}/\cP_{a_+}]\cong [\Sp_a/P_a].\]
Corollary~\ref{cor:flat-HFN} implies that $\cM_{\le\la,a_+}$ is equi-dimensional. Therefore $\Sp_a$ and its open subset $\Sp_a^0$ are also equidimensional. This finishes the proof of Theorem~\ref{thm:dim-formula}.
\bibliographystyle{alpha}
\bibliography{myref}
\end{document}